\documentclass[11pt,letterpaper,english]{amsart}

\usepackage[utf8]{inputenc}
\usepackage{color}
\usepackage{babel}
\usepackage{enumitem}
\usepackage{amsbsy}
\usepackage{amstext}
\usepackage{amsthm}
\usepackage{amssymb}
\usepackage{xurl}
\usepackage{csquotes}
\usepackage[pdfusetitle,bookmarks=true,bookmarksnumbered=false,bookmarksopen=true,bookmarksopenlevel=1,pdfborder={0 0 1},backref=false,colorlinks=false]{hyperref}

\makeatletter

\newcommand{\noun}[1]{\textsc{#1}}
\providecommand{\tabularnewline}{\\}

\numberwithin{equation}{section}
\numberwithin{figure}{section}
\theoremstyle{plain}
\newtheorem{thm}{\protect\theoremname}[section]
\theoremstyle{remark}
\newtheorem{rem}[thm]{\protect\remarkname}
\theoremstyle{definition}
\newtheorem{defn}[thm]{\protect\definitionname}
\theoremstyle{plain}
\newtheorem{lem}[thm]{\protect\lemmaname}
\theoremstyle{definition}
\newtheorem{example}[thm]{\protect\examplename}
\theoremstyle{plain}
\newtheorem{cor}[thm]{\protect\corollaryname}
\theoremstyle{plain}
\newtheorem{prop}[thm]{\protect\propositionname}
\theoremstyle{plain}
\newtheorem{conjecture}[thm]{\protect\conjecturename}
\newtheorem{problem}[thm]{\protect\problemname}

\usepackage{slashed}
\usepackage[all]{xy}
\usepackage{tikz}
\usetikzlibrary{intersections}

\newcommand{\pdv}{\partial}
\newcommand{\gar}{\ifmmode{\mathrm{G\mathring{a}rding}}\else\text{Gårding}\fi}
\DeclareMathOperator{\Inv}{\tau}

\newcommand{\SSGF}{S}
\newcommand{\CSGF}{C}

\def\A{\mathcal{A}}
\def\a{\textbf{a}}
\def\b{\textbf{b}}
\def\e{\textbf{e}}
\def\t{\textbf{t}}

\def\R{\mathbb R}
\def\Z{\mathbb Z}

\def\HH{\mathrm H}
\def\N{\mathbb N}
\def\L{\mathrm L}

\def\S{\mathrm S}
\def\G{{\mathrm G}^\mathfrak{a}}

\def\H {\mathrm H}
\def\GS{\mathrm G}

\def\C{\mathcal C}

\def\x{\mathbf{x}}
\def\y{\mathbf{y}}
\def\u{\mathbf{u}}
\def\v{\mathbf{v}}
\def\w{\mathbf{w}}

\def\T{\mathrm T}
\def\D{\mathrm D}

\def\n{[n]}
\def\proj{\Pi^\downarrow}
\def\pol{\Pi^\uparrow}
\def\Sym{\mathrm {sym}}

\def\<{\langle}
\def\>{\rangle}
\def\sym{\text{sym}}
\def\branden{Br\"and\'en}
\def\vtl{\vartriangleleft}
\usepackage{babel}
\providecommand{\corollaryname}{Corollary}
\providecommand{\definitionname}{Definition}
\providecommand{\examplename}{Example}

\providecommand{\lemmaname}{Lemma}
\providecommand{\propositionname}{Proposition}
\providecommand{\remarkname}{Remark}
\providecommand{\theoremname}{Theorem}
\providecommand{\conjecturename}{Conjecture}

\ifdefined\showcaptionsetup
 \PassOptionsToPackage{caption=false}{subfig}
\fi
\usepackage{subfig}
\makeatother

\usepackage[style=numeric,doi=false, url=false,giveninits=true, maxbibnames=10]{biblatex}
\AtEveryBibitem{\clearfield{note}}
\AtEveryBibitem{\clearfield{issn}}
\AtEveryBibitem{\clearfield{isbn}}
\AtEveryBibitem{\clearfield{language}}
\providecommand{\conjecturename}{Conjecture}
\providecommand{\corollaryname}{Corollary}
\providecommand{\definitionname}{Definition}
\providecommand{\examplename}{Example}
\providecommand{\lemmaname}{Lemma}
\providecommand{\propositionname}{Proposition}
\providecommand{\remarkname}{Remark}
\providecommand{\theoremname}{Theorem}

\providecommand{\problemname}{Problem}
\begin{document}
\title[Gårding polynomials]{Gårding polynomials}
\author{Hao Fang}
\email{hao-fang@uiowa.edu}
\address{Department of Mathematics, University of Iowa, Iowa City, IA 52246,
USA}
\author{Biao Ma}
\email{bma@math.ecnu.edu.cn}
\address{School of Mathematical Sciences, East China Normal University, 500
Dongchuan road, Shanghai, China.}
\thanks{H. F.'s work is partially supported by a Simons Foundation mathematics
collaboration grant and an NSF-RTG grant. B. M. is partially supported
by NSFC grant 12471052.}
\begin{abstract}

We introduce G{\aa}rding polynomials, a class of real multivariate polynomials characterized by positivity regions that are invariant under translation by positive vectors and closed under strictly positive affine transformations. We prove that this geometric formulation is equivalent both to a reduction to the multi-affine setting via polarization and to a recursive criterion in terms of partial derivatives.

The class of G{\aa}rding polynomials strictly extends that of real stable polynomials while preserving many of their structural properties. In particular, multi-affine G{\aa}rding polynomials with nonnegative coefficients satisfy the Rayleigh property, and their positive univariate specializations have ultra log-concave coefficient sequences.

The G{\aa}rding property for several matroid generating functions is preserved under natural matroid operations. As applications, we derive new negative dependence results for generating functions associated with various classes of matroids and graphs, including examples beyond the scope of real stability and Lorentzian methods. We obtain analogous results for characteristic polynomials of certain matrix classes.
\end{abstract}
%%%%%%%%% Key words
%%%%%%%%% MSC

\maketitle
\tableofcontents{}

\section{Introduction}

In this article, we introduce \emph{Gårding polynomials}, a natural geometric
extension of real stable polynomials.

In 1959, motivated by the well-posedness problem for linear hyperbolic PDEs, Lars G\aa{}rding \cite{Gar59} introduced hyperbolic polynomials, a class of real polynomials with rich geometric and algebraic structure. Closely related is the theory of real stable polynomials, whose multivariate framework has become a powerful tool in analysis, probability, combinatorics, and optimization; see
\cite{BB08Duke,MR2661242,BB09Annals,BB09I,B07Adv,BBL09,Wagner11,KPV15,Guler97,MSS15i,MSS15ii,MR2411443}.
Developed systematically by Borcea, Brändén, and others, this theory has as one of its central consequences the strong Rayleigh property, a robust form of negative dependence that applies to a broad class of generating functions in probability theory and combinatorics.

More recently, Lorentzian polynomials, introduced by Brändén--Huh
\cite{BrandenHuh20} and Anari-Liu-OveisGharan-Vinzant
\cite{ALOVi,ALOVii,ALOViii}, provide a unified framework capturing
strong log-concavity phenomena. In particular, they apply naturally
to homogeneous generating functions associated with matroids, such as
basis generating functions and related constructions. Together, these
theories have significantly advanced the understanding of log-concavity
and negative dependence in combinatorial generating functions.

However, many natural polynomials arising from matroid theory fall outside both frameworks, despite exhibiting closely related monotonicity and negative dependence phenomena. The strong Rayleigh property is often too restrictive for these examples. Lorentzian theory, although powerful, is most naturally adapted to homogeneous polynomials, and does not in general yield the full Rayleigh property.

Motivated by G{\aa}rding's geometric viewpoint, we introduce a class of
multivariate polynomials defined via positivity regions. 
Here, a positivity region is a connected component of the set where a nonzero real polynomial is positive. We require such a region to be closed under translation by positive vectors, a condition we call the positive ray test (PRT), which underlies our construction.

A polynomial is called  multi-affine if it has degree at most $1$
in each variable. \emph{Multi-affine Gårding polynomials} are characterized
precisely by the existence of a positivity region passing the positive
ray test. A \emph{general multivariate Gårding polynomial} is defined as an
appropriate pullback of a multi-affine Gårding polynomial by a strictly
positive affine map; see Definition~\ref{def:general garding}. Details of
our notation are introduced in later sections.

\medskip
\noindent\textbf{Main results.}
The main result of this paper is a structural theorem that provides equivalent geometric and algebraic characterizations of G{\aa}rding polynomials.
\begin{thm}
\label{thm:newmain}For a nonzero polynomial $f\in\R[x_{1},x_{2},\cdots,x_{n}]$,
the following are equivalent:
\begin{enumerate}
\item $f$ is \gar{}, i.e. $f=\mu^{*}g$ where $\mu$ is a strictly positive
affine map and $g$ is a multi-affine \gar{} polynomial;
\item Every admissible symmetric polarization of $f$ is a multi-affine
G{\aa}rding polynomial;

%The polarization of $f$ is multi-affine and \gar{}; 
\item For each multi-index $\alpha\in\Z_{\geq0}^{n}$ such that $\pdv^{\alpha}f\not\equiv0$,
$\pdv^{\alpha}f$ admits a unique positivity region satisfying the positive
ray test, and its positivity region is contained in those of its partial derivatives.
\end{enumerate}
\end{thm}
See later sections for details. Theorem~\ref{thm:newmain} provides the algebraic foundation of the geometric
theory by linking the key mechanisms that recur throughout the paper. First, polarization
reduces the general theory to a canonical multi-affine setting. Second, and
more importantly, the G{\aa}rding property admits an intrinsic \emph{recursive}
description governed by partial derivatives: each nontrivial $\pdv^{\alpha}f$
admits a distinguished positivity region passing PRT, and these regions are
nested under differentiation. This recursive, derivative-based viewpoint is closely aligned with Lorentzian theory and underlies the verification methods used in later sections and applications.

Real stable polynomials, up to a sign, form a proper subclass of G{\aa}rding
polynomials (Theorem~\ref{thm:stable-garding}). The G{\aa}rding
class strictly extends real stability while retaining many of its structural features.

Homogeneous $\gar{}$ polynomials form a proper subclass of Lorentzian polynomials.
These three classes coincide in the case of homogeneous degree~$2$ polynomials with
nonnegative coefficients. See Section~\ref{sec:Discussion} for further discussion.

Similarly to real stable and Lorentzian polynomials, the class of G{\aa}rding polynomials
is closed under a range of natural algebraic and topological operations and admits a
linear preserver theory \cite{BB09I,BrandenHuh20}. See Section~\ref{sec:linearpreserver}
and Section~\ref{sec:structuralconsequence}.

A key consequence of Theorem \ref{thm:newmain} is the following negative dependence result.

\begin{thm}
Let $f(\x)$ be a multi-affine Gårding polynomial with nonnegative coefficients. Let
\[
\Gamma^n_+ := \{\x=(x_1,\dots,x_n)\in \R^n : x_i>0 \text{ for all } i\}.
\]
Then:
\begin{enumerate}
    \item $f$ is Rayleigh. Equivalently, for all $\x \in \overline{\Gamma^n_+}$,
    \[
    \partial_i f(\x)\,\partial_j f(\x) \ge f(\x)\,\partial_{ij} f(\x).
    \]
    
    \item For any $\a \in \overline{\Gamma^n_+}$, consider the univariate specialization
    \[
    g(t) := f(\a t).
    \]
    If $g(t)\not\equiv 0$, then its coefficient sequence is ultra log-concave.
    
    \item If $f$ is also homogeneous, then $f$ is Lorentzian. Thus,  $f$ is log-concave on $\Gamma^n_+$. Moreover, the associated measure satisfies the negative association property.
\end{enumerate}
\end{thm}

The Rayleigh property, ultra log-concavity, and negative association  are fundamental notions that have been extensively studied in probability theory and combinatorics \cites{Pemantle2000}{COSW04}{MR2248326}{MR2428906}{BBL09}{HuhCorrelationBoundsFields2021a}.

In this paper, we study the generating functions naturally associated with matroids,
including the basis generating function and those arising from spanning and
cospanning sets. These generating functions encode rich combinatorial information
and exhibit favorable structural properties in our framework.

As an application of our construction, we establish the G{\aa}rding property, and consequently the Rayleigh property, for generating functions associated with several important classes, as stated in the following theorem.

\begin{thm}\label{newclass}
The G{\aa}rding property of the spanning-set and cospanning-set generating
functions is preserved under direct sums, series--parallel connections, deletion,
and contraction; equivalently, the corresponding classes are closed under these
operations.

\smallskip
Moreover, for each matroid in the following classes:
\begin{itemize}
    \item cycle matroids of series--parallel networks;
    \item uniform matroids;
    \item  matroids obtained from uniform matroids by removing one basis;
    \item matroids with at most six elements,
\end{itemize}
the following hold:
\begin{enumerate}
    \item The generating functions of spanning-set, cospanning-set, and basis 
    are G{\aa}rding polynomials;

    \item Consequently, these generating functions, as well as the
    independent-set generating function,
    satisfy the Rayleigh property;

    \item Their univariate specializations have ultra log-concave coefficients, and 
    satisfy the Mason-type inequalities.
\end{enumerate}
\end{thm}
Theorem~\ref{newclass} yields a broad class of
generating functions with the G{\aa}rding and Rayleigh properties. This is particularly effective in non-homogeneous settings, where existing theories are less directly applicable.

Wagner~\cite{WagnerRank3} showed that every rank~\(3\) matroid is Rayleigh,
giving a complete picture in that case. In higher ranks,  much less
is known. Our approach yields new families of Rayleigh matroids. The first,
arising from Theorem~\ref{newclass}, includes an infinite family of sparse
paving matroids. The second consists of several explicit rank~\(4\) families
constructed using the \(B\)-G{\aa}rding framework, which naturally extends to higher rank cases.

As a notable borderline example, the Fano matroid \(F_7\) exhibits an
asymmetry: its cospanning-set generating function is G{\aa}rding, while its
spanning-set generating function is not, although both are Rayleigh
(see Theorem~\ref{F7summary}). We also recover a result of
Erickson~\cite{EricksonNegativeCorrelation} showing that the independent-set
generating function of \(F_7\) is Rayleigh.

Overall, these results extend beyond the reach of stability and Lorentzian theories. 
Our approach provides a framework that strictly extends the scope of stability theory 
while remaining complementary to the Lorentzian framework, establishing Rayleigh-type 
inequalities and ultra log-concavity for a large class of generating functions, 
particularly in non-homogeneous settings. We expect further structural developments 
and applications.

\medskip
\noindent\textbf{Motivation from PDEs.}
Much of our work was originally motivated by the study of fully nonlinear PDEs. Hyperbolic polynomials,
introduced by G{\aa}rding \cite{Gar59}, arise naturally in this context,
including the Monge--Amp\`ere and $\sigma_k$-Hessian equations
\cite{caffarelli1985dirichlet,brendle2004variational,CGY1chang2002equation,chang2010liouville,chou-wang2001,CHY2chang2011prescribing,GSgursky2018formal,MR1719551,viaclovskytrans2000,shankar2021rigidity}.
However, several important geometric problems, including those related to the twisted
$J$-equation \cites{collins2017convergence}{chen2021j} and inverse
$\sigma_k$-equations \cite{Fang-Lai-Ma}, do not arise from hyperbolic polynomials.

In \cite{FANG2024109867} we studied a general class of such geometric PDEs on K\"ahler manifolds,
relating solvability to stability-type conditions and extending earlier work
\cite{song2008convergence,demailly2004numerical,Fang-Lai-Ma,Guan2014Second-order,chen2021j,collins2017convergence,szekelyhidi2018fully,Song2020NakaiMoishezonCF,datar2021numerical}.
Independently, Lin \cite{Lin23JFA,LIN2024} obtained an effective convexity criterion for level-sets
associated with symmetric multi-affine polynomials. These developments led to the introduction of G{\aa}rding polynomials studied in this paper.

\medskip
\noindent\textbf{Methods and techniques.}
The proof of the main theorem combines geometric and algebraic ideas and is
analytic in nature. 

For multi-affine polynomials, positivity regions satisfying the positive ray test
implies the closure properties of G{\aa}rding polynomials under basic operations
such as restrictions, translations and partial derivatives. We then study  two binary relations for \gar{} polynomials
which encode interactions of positivity regions. The closure properties and binary relations lead to a
linear preserver theorem in the G{\aa}rding setting, extending the seminal
work of Borcea and Brändén \cite{BB09Annals,BB09I,MR2661242}.

A key difficulty arises in extending the theory beyond the multi-affine
setting via polarization, which leads to a nonlinear and generally
non-convex optimization problem. Classical complex-analytic tools, such as
the Grace--Walsh--Szeg\"o theorem, which underlie polarization arguments in
the stable setting, do not apply here. A central contribution of this work
is the identification of intrinsic geometric and algebraic structures that
reduce this difficulty to linear optimization on convex domains.

In the matroid setting, the proofs of several special cases reveal delicate
analytic features of generating functions, reflecting combinatorial
properties of the underlying matroids. The more technical arguments are
developed in detail in the appendices.

 \medskip
\noindent\textbf{Organization of the paper.}
Section~2 introduces positivity regions and positive affine maps, including
the positive ray test and the associated geometric
filtration. Section~3 develops basic algebraic operations and revisits real
stable polynomials from our geometric perspective. Section~4 defines G{\aa}rding
polynomials and establishes their basic properties. Sections~5 and~6 develop
the multi-affine theory, including binary relations. Section~7 proves a linear
preserver theorem. Sections~8 and~9 extend the theory via polarization and
provide an intrinsic recursive characterization. Section~10 establishes closure
properties. Section~11 studies homogeneous \gar{} polynomials and their relation
to stable and Lorentzian polynomials. The final sections present applications
to matroid generating functions, ultra log-concavity, negative dependence, and
characteristic polynomials of certain matrix classes. The appendices contain
technical verification arguments and explicit classification results.

\smallskip{}
\section*{Acknowledgments}

The first-named author thanks Alice Chang for introducing him to
Gårding's foundational work~\cite{Gar59} many years ago.
Both authors thank Gang Tian for his encouragement, June Huh for his
interest and for suggesting the exploration of applications in matroid
theory, Shouda Wang for discussions and collaboration, including on the
material presented in Appendix~C, and Botong Wang and Jian Xiao for their
interest and helpful discussions.
The first-named author also thanks Renbo Zhao for discussions and
Haoran Wu for assistance with computations.
The second-named author thanks Yibo Gao, Tao Gui, Thuy-Duong Vuong, and
Suhan Zhong for helpful discussions. He also thanks Gao Chen for pointing out an error in the statement of Theorem \ref{proper-position}. 

\part{Polynomials and positivity regions}

\section{Geometry of positivity regions}\label{sec:Notations}

In this section, we develop the geometric framework that we use throughout
our paper. The central objects are positivity regions in Euclidean
spaces, together with their behavior under proper affine transforms.
This setting isolates crucial geometric notions of monotonicity and
invariance. Several concepts admit both a general form and a more
restrictive non-degenerate form, which are used according to different
requirements of algebraic and analytic arguments.

For $n\in\N$, define $[n]=\{1,2,\cdots,n\}$. Let $\R^{n}$ denote
Euclidean space with standard basis $\mathbf{e}_{1}=(1,0,\cdots,0)$,
$\mathbf{e}_{2}=(0,1,0,\cdots,0)$, etc. Points in $\R^{n}$ will
be written as $\x=(x_{1},\cdots,x_{n})\in\R^{n},$ with the same notation
as for $\y,\a,\cdots$. We write $\x>0$ (respectively, $\x\geq0$)
if each $x_{i}$ is positive (respectively, non-negative) for all
$i\in[n]$.

\subsection{Weakly and strictly positive linear and affine maps}

We introduce several basic geometric objects and transformations.
With respect to a fixed coordinate system in $\R^{n}$, the \emph{positive
hyper-octant} is defined by
\[
\Gamma_{n}^{+}:=\{\x\in\R^{n},\ \x>0\},
\]
with closure 
\[
\overline{\Gamma_{n}^{+}}:=\{\x\in\R^{n},\ \x\geq0\}.
\]
For any $\x_{0}\in\R^{n},$ the set $\x_{0}+\Gamma_{n}^{+}$ is called
an \emph{affine positive hyper-octant}; the set $\x_{0}+t\v$ for
$t>0,$ $\v>0$ is called a positive ray. 
\begin{rem}
\label{rem:easy}It is an elementary fact that a positive ray and
an affine positive hyper-octant have nontrivial intersection. Two
affine positive hyper-octants have a nontrivial intersection.

We introduce a geometric condition encoding positivity.
\end{rem}

\begin{defn}
[\noun{{}Positive Ray Test, PRT{}}]
\label{def:positive=00003D000020ray}
A subset $R\subset\mathbb{R}^{n}$ is said to pass the positive ray
test (PRT) if $R+\overline{\Gamma_{n}^{+}}\subset R$. Equivalently,
$R$ is closed under translation by positive vectors. 
\end{defn}

We also introduce a geometric non-degeneracy condition.
\begin{defn}[\noun{{}Negative Ray Terminating, NRT{}}]
\label{def:NRT region}
A subset $R\subset\R^{n}$ is said to be negative ray terminating
(NRT), if for every $\mathbf{x}\in R$, $(\{\mathbf{x}\}-\overline{\Gamma_{n}^{+}})\cap R$
is bounded. 
\end{defn}

\begin{rem}
\label{rem:intersection}Definitions \ref{def:positive=00003D000020ray}
and \ref{def:NRT region} originate in nonlinear PDEs theory, where
they serve as structural conditions ensuring ellipticity. 
\end{rem}

We record some basic properties of  the subsets in $\R^{n}$ passing PRT
or being NRT:
\begin{lem}
\label{lem:basic}Assume $R\subset\R^{n}$.
\begin{enumerate}
\item If $R$ passes PRT, then $R$ is connected; 
\item \label{enu:if--and}if $R_{1},$$R_{2}\in\R^{n}$ both pass PRT, then
$R_{1}\cap R_{2}\neq\emptyset$ ; 
\item \label{enu:directproduct}if $R_{1}\subset\R^{n_{1}}$ and $R_{2}\subset\R^{n_{2}}$
both pass PRT, then $R_{1}\times R_{2}\subset\R^{n_{1}+n_{2}}$ passes
PRT;
\item Fix an index set $I$. If $R_{i}$ passes PRT for all $i\in I$, then
$\cup_{i\in I}R_{i}$ passes PRT. If $\cap_{I}R_{i}\not=\emptyset$
then it also passes PRT.
\item \label{enu:NRTlabel}If $R\subset\R^{n}$ is NRT, so is any of its
open subsets.
\end{enumerate}
\end{lem}

\begin{proof}
All statements follow directly from Remark \ref{rem:easy} and Definition
\ref{def:positive=00003D000020ray}.
\end{proof}
To describe proper geometric transformations, we introduce notions of positive linear maps.
\begin{defn}\label{positive linear}
For a real $m\times n$ matrix $A=(a_{ij})$, let $T_{A}:\R^{n}\to\R^{m}$
be a linear map 
\[
T_{A}\x:=(\sum_{j=1}^{n}a_{1j}x_{j},\cdots\sum_{j=1}^{n}a_{mj}x_{j}).
\]
The map $T_{A}$ is called \emph{weakly positive} if 
\[
T_{A}(\overline{\Gamma_{n}^{+}})\subset\overline{\Gamma_{m}^{+}}.
\]
The map $T_{A}$ is called \emph{strictly positive} if it is weakly positive and 
\[
T_{A}(\overline{\Gamma_{n}^{+}})\cap\Gamma_{m}^{+}\neq\emptyset.
\]
\end{defn}
\begin{lem}
\label{lem:matrix characterization}
With the above notation, the operator $T_{A}$ is weakly positive if and only if $a_{ij}\geq 0$ for all $i\in[m]$ and $j\in[n]$. Furthermore, $T_{A}$ is strictly positive if and only if $a_{ij}\geq 0$ for all $i\in[m]$ and $j\in[n]$, and in addition $\sum_{j=1}^{n} a_{ij} > 0$ for every $i\in[m]$.
\end{lem}

The proof is straightforward and is therefore omitted.

\begin{rem}
\label{rem:irreducible matrix} For a weakly positive linear map $T_{A},$
if $\sum_{i=1}^{m}a_{ij}=0,$ for some $j\in[n]$, then the image
of $T_{A}$ is independent of the coordinate $x_{j}.$ We call $T_{A}$
irreducible if $\sum_{i=1}^{m}a_{ij}>0$, for every $j\in[n]$. 
\end{rem}

We list some common geometric transformations used in this paper.
\begin{defn}
[Geometric Transformations] Let $\x=(x_{1},\cdots,x_{n})\in\R^{n}.$ We
define the following basic geometric transformations between framed
Euclidean spaces. Assume $\x\in\R^{n}$.\label{def:space-transform}

\begin{enumerate}
\item Translation. For $\a\in\R^{n},$ define $\T_{\a}(\x):=\x+\a$.
\item Permutation. For a permutation $\sigma$ of $[n],$ define $\x_\sigma:=(x_{\sigma(1)},\cdots,x_{\sigma(n)})$.
\item Dilation. For $\a\in\R^{n},$ with $\a>0,$ define $\D_{\mathrm{\a}}(\x):=(a_{1}x_{1},\cdots,a_{n}x_{n})$. 
\item Deletion. For $i\in[n]$, define $\pi_{i}(x_{1},\cdots,x_{n}):=(x_{1},\cdots,x_{i-1},x_{i+1},\cdots,x_{n}).$
\item Restriction. For $i\in[n],$ define $(\x)|_{x_{i}=0}:=(x_{1},\cdots,x_{i-1},0,x_{i+1},\cdots,x_{n}).$
\item Inversion: In the domain $\{x_{i}\neq0,\ \forall i\in [n] \},$ define
\[
\iota(x_{1},\cdots,x_{n})=(-\frac{1}{x_{1}},\cdots,-\frac{1}{x_{n}}).
\]
 
\end{enumerate}
\end{defn}
Permutations, dilations, deletions, and restriction are special cases
of weakly positive linear maps, whereas the first three are strictly
positive. Inversion is \noun{not} a linear map.
\begin{defn}[Positive affine maps]\label{positive affine}
We define $\A_{+}$, the set of\emph{ weakly positive affine maps,}
which are compositions of translations and weakly positive linear
maps. We define $\A_{++}$, the set of\emph{ strictly positive affine
maps}, which are compositions of translations and strictly positive
linear maps.
\end{defn}

\begin{rem}\label{composition}

By Definitions~\ref{positive linear} and \ref{positive affine},  for any $\mu\in \A_{++}$, $\mu(\Gamma^+_n)\cap \Gamma^+_n \neq \emptyset$. A crucial geometric consequence is that the composition of two strictly positive affine transformations, whenever well defined, is again a strictly positive affine transformation.
\end{rem}

\begin{rem}
\label{rem:pullback}Closures of affine positive hyper-octants are
invariant under weakly positive affine mappings. Consequently, the
PRT 
property is preserved under weakly positive affine transformations.
More 
precisely, let $R\subset\R^{n}$ be nonempty and satisfy PRT, and
let $\mu\in\A_{+}$
. Then the pre-image
\[
\mu^{-1}(R):=\{\x\in\R^{m}\mid\mu(\x)\in R\}
\]
also passes PRT whenever it is nonempty. Furthermore, if $R\neq\emptyset$
and $\mu\in\A_{++}$, then $\mu^{-1}(R)$ is necessarily nonempty and
satisfies 
PRT. Indeed, by strict positivity, $\mathrm{Im}(\mu)$
contains 
a positive ray and by Remark~\ref{rem:intersection}, its
intersection 
with $R$ is nonempty.
\end{rem}

\begin{rem}
In this paper, the positivity of subsets of $\R^{n}$ is characterized
via the PRT property. The positivity of linear and affine maps is
defined accordingly, following ideas from matrix theory. The distinction
between weakly and strictly positive affine transforms will be crucial.
\end{rem}

\section{Algebraic and positive affine actions on polynomials}

In this section, we record some transformations in the space of polynomials. We first introduce some notation.

For $\x=(x_{1},x_{2},\cdots,x_{n})$, let $\R[\x]$ denote the polynomial
ring in $x_{1},x_{2},\cdots,x_{n}$. Let $\alpha=(\alpha_{1},\cdots,\alpha_{n})\in\N^{n}$
be a multi-index with
\[
|\alpha|:=\sum_{i=1}^{n}\alpha_{i}.
\]
Let $\pdv_{i}=\frac{\pdv}{\pdv x_{i}}$  and define
\[
\pdv^{\alpha}f:=\pdv_{1}^{\alpha_{1}}\cdots\pdv_{n}^{\alpha_{n}}f,\ \mathbf{x}^{\alpha}:=x_{1}^{\alpha_{1}}\cdots x_{n}^{\alpha_{n}}.
\]
For subset $S\subset [n]$, we denote \[\x^S:=\prod_{i\in S}x_i.\] This notation is frequently used for  multi-affine polynomials. For $k\in\N$, the $k$-th symmetric polynomial of $\x$ is denoted by 
\[
\sigma_{k}(\x):=\sum_{1\leq i_{1}<i_{2}<\cdots<i_{k}\leq n}\prod_{j=1}^{k}x_{i_{j}}.
\]
We use the convention
that $\sigma_{0}(\x)=1$ and $\sigma_{k}(\x)=0$ if $k>n$.

For $d\in\N$, and $\kappa\in{\mathbb{N}}^{n}$ define
\[
\R_{d}[\x]:=\{f\in\R[\x],\ \deg f\leq d\},\quad \R_{\kappa}[\x]:=\{f\in\R[\x],\ \deg_{x_{i}}f\leq\kappa_{i}\}.
\]

The superscript $\mathfrak{a}$ denotes the multi-affine property, the
superscript $\mathfrak{h}$ denotes homogeneity, and the subscript $+$
denotes nonnegative coefficients. For example:
\[
\R_{+}^{\mathfrak{a}}[\x]:=\{f\in\R[\x],\ \ \deg_{x_{i}}f\leq1\ \mathrm{and\ }f\ \text{has non-negative coefficients}\}.
\]

For a finite dimensional subspace $V$ of $\R[\x]$, we equip the compact-open topology, which is equivalent to the topology
of uniform convergence in compact sets. 

Finally, for a univariate $f\in\R[x]$, $r(f)$ denotes the \emph{largest
real root} of $f$ when it exists.

The geometric transformations introduced earlier naturally induce
corresponding operations on polynomials. Together with intrinsic polynomial
operations, such as partial differentiation, these induced maps will
be used to formulate invariance properties and to construct new examples.  

We list the polynomial transformations considered in this paper, many
of which have appeared in the theory of stable polynomials.
\begin{defn}
[Polynomial Transformations]\label{def:maps} Let $f,g\in\R[\x].$

\begin{enumerate}
\item Pullback by affine maps. If $\mu:\R^{n}\to\R^{m}$ is an affine transformation, $\mu^{*}:\R[\y]\to\R[\x]$ is given as $(\mu^{*}f)(\x):=f(\mu(\x)).$
Both $\A_{+}$ and $\A_{++}$ act on polynomials by
pullback.
\item Partial differentiation. For $i\in[n]$, define $\pdv_{i}f(\x):=\frac{\pdv}{\pdv x_{i}}f(\x).$
More generally, for $\mathrm{\a}\in\R^{n}$ with $\a\geq0,$ set the
directional derivative
\[
\pdv_{\a}f(\x):=\sum_{i=1}^{n}a_{i}\pdv_{i}f.
\]
\item Homogenization. For $f=\sum_{\alpha}a_{\alpha}\x^{\alpha}\in\R[\x]$
with $\deg f=d$, 
\[
\HH(f)(\x,y):=f(\frac{\x}{y})y^{d}=\sum_{\alpha}a_{\alpha}\x^{\alpha}y^{d-|\alpha|}\in\R^{\mathfrak{h}}[\x,y].
\]
\item Restriction (specialization). For $i\in[n]$, $a\in\R$, define $$f(x)|_{x_{i}=a}:=f(x_{1},\cdots,x_{i-1},a,x_{i+1},\cdots,x_{n}).$$
\item Top degree part. Let $f(x)$ have degree $d$ with $f(x)=\sum_{k=0}^{d}f_{k}$
with $f_{k}$ homogeneous of degree $k.$ Define $f^{\mathrm{top}}=f_{d}$. 
\item Projection (diagonal specialization). The diagonal restriction of
$f$ is the univariate polynomial $\proj(f):=f(t,\cdots,t)$.
\item Polarization. If $d\leq n$, $\pol$ : $\R_{d}[x]$$\to\R^{\mathfrak{a}}[\x]$
is the unique linear map such that $\pol(x^{k})=\sigma_{k}(\x)/\binom{n}{k}.$
\item $\kappa$-Projection. For $\kappa\in\N^{n}$ and $f\in\R_{\kappa}[x_{ij}]$,
$\proj_{\kappa}f:=f|_{x_{ij}=x_{i}}$, where $i\in\N$ and $j\in[\kappa_{i}]$. 
\item $\kappa$-Polarization. For $\kappa\in\N^{n}$ and $f\in\R_{\kappa}[\x]$,
$\pol_{\kappa}:\R_{\kappa}[\x]\to\R_{\kappa}^{\mathfrak{a}}[x_{ij}]$
is a linear map such that $\pol_{\kappa}(\x^{\alpha}):=\prod_{i=1}^{n}\sigma_{\alpha_{i}}(x_{i1},\cdots,x_{i\kappa_{i}})/{\binom{\kappa_{i}}{\alpha_{i}}}$.
\item Multiplication. $(f\cdot g)(\mathbf{x}):=f(\mathbf{x})\cdot g(\mathbf{x})$.
\item Inversion. If $f\in\R_{\kappa}[\x],$ then $(T_{\Inv}f)(\x):=f(-\frac{1}{x_{1}},\cdots,-\frac{1}{x_{n}})\cdot\prod_{i=1}^{n}(x_{i})^{\kappa_{i}}$.
\end{enumerate}
\end{defn}
In this paper, a subset $H$ of polynomials is said to be preserved
by $\A_{+}$ if $\mu^{*}(H)\subset H$, for all $\mu\in\A_{+}.$ A
similar definition applies to $\A_{++}.$ 

\section{From stable polynomials to Gårding polynomials}\label{sec:From-stable-polynomials}
In this section, we introduce the central objects of the paper by defining
G\aa rding components and G\aa rding polynomials, and relate them to real stable
polynomials.
\subsection{Stable and hyperbolic polynomials revisited}

We briefly recall the basic facts about stable and hyperbolic polynomials
that motivate our construction. While most of the material is standard
and is included here to fix the notation, we also introduce a new geometric
viewpoint that facilitates later discussions.
\begin{defn}[Real stable polynomials] 
\label{def:stable-polynomial}A nontrivial polynomial $f(\x)\in\R[\x]$ is called \emph{real
stable} if $f(\x)\neq0$, for all $\x\in\mathbb{C}^{n}$ with the
imaginary part $\Im x_{i}>0$ for all $i$. The zero polynomial is considered a trivial real stable polynomial.  We denote $\S[\x]$ to
be the set of real stable polynomials whose homogeneous top degree
part has nonnegative coefficients. We  denote $\S_{+}[\x]$ the set of stable polynomials with
nonnegative coefficients.

Two polynomials $f(\x),g(\x)\in\S[\x]$
are said to be \emph{in proper position}, written $f\prec g$, if
$f(\x)y+g(\x)\in\S[\x,y].$ 
\end{defn}

Since the top degree part of a real stable polynomial is also real
stable \cite[Proposition 2.2]{COSW04}, and all non-zero coefficients
of a homogeneous real stable polynomial have the same sign \cite[Theorem 6.1]{COSW04},
every real stable polynomial lies in either $\S[\x]$ or $-\S[\x]$.
By convention, the trivial polynomial $0$ is included in $\S[\x].$
The set $\S[\x]$ thus defines a distinguished positive cone in the
space of stable polynomials, which is natural in our geometric setting.

We also recall the following definition of Gårding. 
\begin{defn}
\label{def:Gardingcone}Let $h(\x)$ be an $n$-variable polynomial
of \emph{homogeneous} degree $d.$ Let $\mathbf{a}\in\R^{n}$. $h(\x)$
is \emph{hyperbolic with respect to $\mathbf{a}$} if, for any $\x\in\R^{n}$,
the univariate polynomial $h(\mathbf{a}t+\x)$ has only real roots.
The \gar{} cone of $h$ with respect to $\mathbf{a}$ is defined
by
\[
\C_{h,\mathbf{a}}:=\{\x\in\R^{n},\ (h(\mathbf{a}))^{-1}h(\mathbf{a}t+\x)>0,\ \forall t\geq0\}.
\]

We summarize some well-known facts; see, for example, \cite{Gar59},
\cite[Lemma 2.3 and Lemma 2.4]{Wagner11}, \cite[Proposition 5.3]{Pem11}),
and \cite[Proposition 7.1]{MR3055586}.
\end{defn}

\begin{lem}
\label{lem:stablepolyclosurethe}

\begin{enumerate}
\item \label{enu:part1}$f(\mathbf{x})\in\S[\x]\backslash\{0\}$ if and only if for all $\mathrm{\x_{0},\mathbf{a}\in\R}^{n},$
with $\mathrm{\a>0},$ $f(\x_{0}+t\mathbf{a})\in \S[t]\backslash\{0\}$.
\item Except for homogenization, $\S\cup(-\S)$ is preserved under the operations in Definition \ref{def:maps}.
\item $f\in\S_{+}$ if and only if its homogenization $\H f\in\S_{+}$.
\item For two univariate polynomials $f(x),g(x)\in\S,$ $f(x)\prec g(x)$
if and only if the roots of $f$ interlace roots of $g$ and the largest
roots satisfy $r(f)\leq r(g)$.
\item \label{enu:part3}For $f(x)$ a univariate polynomial of degree $d,$
$f(x)$ is stable if and only if it has only real roots. 
%In particular, the largest real roots of $f,f',\cdots$ satisfy $r(f)\geq r(f')\geq\cdots\geq r(f^{(d-1)})$.
\item \label{enu:part4}If $h$ is hyperbolic with respect to $\a\in\R^{n}$,
and $\b\in\C_{h,\a},$ then $\C_{h,\a}=\C_{h,\b}.$ Furthermore, $\C_{h,\a}$
is the connected component of the set $\{h\neq0\}$ that contains $\a$.
\item \label{enu:part6}If $h$ is hyperbolic with respect to $\a\in\R^{n}$,
then so is $\partial_{\a}h$. Furthermore, $\C_{h,\a}\subset\C_{\partial_{\a}h,\a}.$
\item \label{enu:part5}$f(\mathbf{x})\in\S$ if and only if \label{enu:hom-is-hyperbolic}$\HH f(\mathbf{x},y)$
is hyperbolic with respect to any $\v'=(\v,0)$ with $\v>0$.
\end{enumerate}
\end{lem}

For future reference, we list the following theorem, which is part
of the linear preserver theory proved by Borcea-Brändén \cite[Theorem 1.2]{BB09I}.
\begin{thm}
\label{thm:pol-restiction of stable}For any $\kappa\in\N^{n}$ and
$T:\R_{\kappa}[\x]\to\R[\y]$ to be a linear transformation. If $\sym_{T}(\y,\mathbf{u}):=T[(\x+\mathbf{u})^{\kappa}]\in\S[\y,\mathbf{u}]$,
then $T$ maps $\S_{\kappa}[\x]$ to $\S[\y].$ Moreover, for any
$\kappa\in\N^{n}$, $\pol_{\kappa}(f)\in\S$ if and only if $f\in\S$. 
\end{thm}

 We  discuss the geometry of stable polynomials in our setting. First,
we highlight the invariance property of stable polynomials under weakly
and strictly positive affine transforms.
\begin{lem}
\label{lem:stable-weak-invariance}Stable polynomials are preserved
under $\A_{+}$ and $\S[\x]$ is preserved under $\A_{++}$.
\end{lem}

\begin{proof}
The first statement is a direct consequence of Lemma \ref{lem:stablepolyclosurethe}
\ref{enu:part1} and Remark \ref{rem:pullback}. 

For the second statement, let $f\in\S$ and $\mu\in\A_{++}$. By the
first part, $\mu^{*}(f)\in\S[\x]\cup(-\S[\x]).$ Assume $f\neq0$
and let $h=f^{\mathrm{top}}$. By the definition of $\S$ and \cite[Theorem 6.1]{COSW04},
all non-zero coefficients of $h$ are positive. By Lemma \ref{lem:matrix characterization},
$\mu^{*}(h)$ has a non-vanishing top degree part with positive coefficients.
As $\mu^{*}(f)$ and $\mu^{*}(h)$ share the same homogeneous top degree 
part, it follows that $\mu^{*}(f)\in\S[\x]$. 
\end{proof}
Next,  the classical construction of the Gårding cone for
a hyperbolic polynomial depends on the choice of a distinguished direction.
To capture the geometry of stable polynomials in a canonical way,
we associate each $f\in\S[\x]$ with a distinct positivity region. 
\begin{lem}
\label{lem:stablepolynomialcone} For the zero polynomial, we define
$\C_{0}=\R^{n}.$ For any nontrivial $f\in\S$, there exists a unique
open domain $\C_{f}\subset\R^{n}$ such that 
\begin{enumerate}
\item \label{enu:a1}$\C_{f}$ is a connected component of the set $\{f>0\}$; 
\item \label{enu:a2}For all $\a\in\Gamma_{n}^{+}$ and all $\x_{0}\in\C_{f}$,
$\partial_{\a}f(\x_{0})>0$;
\item \label{enu:a3}$\C_{f}$ passes PRT;
\item \label{enu:a4}$\C_{f}\subset\C_{\partial_{i}f}$ for any $i\in[n]$,
and $\partial_{i}f$ nontrivial.
\end{enumerate}
Furthermore, $\C_{f}$ is uniquely determined by these properties.
\end{lem}
\begin{proof}
Consider a nontrivial $f\in\S$ and let $h=\H f$. By definition of
$\S,$ $h(\v,0)>0$ for any $\v\in\Gamma_{n}^{+}.$ Let $\C_{h}^{'}=\C_{h,(\v,0)}.$
By (\ref{enu:part4}) and (\ref{enu:part5}) of Lemma \ref{lem:stablepolyclosurethe},
$\C'_{h}$ is independent of the choice of $\v$, $\C'_{h}$ is a
connected component of $\{h>0\}$, and $\Gamma_{n}^{+}\times\{0\}\subset\C'_{h}.$ 

We call $\C_{f}:=\{\x\in\R^{n},\ (\x,1)\in\C'_{h}\}.$ By the definition
of $\S$ again, $h(\x,1)>0$ for $\x>>0$. Therefore, $\C_{f}\neq\emptyset.$
Since $f(\x)=h(\x,1)$, $\C_{f}$ is a connected component of $\{f>0\}$.

For $\C_{f}$ defined  above, (\ref{enu:a2}) follows from Lemma \ref{lem:stablepolyclosurethe}
(\ref{enu:part6}). (\ref{enu:a3})
then follows from (\ref{enu:a2}). The uniqueness of $\C_{f}$ follows
from (\ref{enu:a1}), (\ref{enu:a3}), and Remark \ref{rem:easy}. Finally
(\ref{enu:a4}) follows from (\ref{enu:a2}).

If $f$ is homogeneous, then $h=f$, and the above argument still works.
\end{proof}
Lemma \ref{lem:stablepolynomialcone} is a key step in reexamining and
reinterpreting the geometric construction of Gårding.
\begin{example}
For $k\leq n$, consider the $k$-th elementary symmetric function
$\sigma_{k}(\x)$ for $n$ variables. Then $\sigma_{k}(\x)$ is known
to be stable and its associated domain is
\[
\C_{\sigma_{k}(\x)}=\{\x\in\R^{n}:\ \sigma_{1}(\x)>0,\cdots,\ \sigma_{k}(\x)>0\}.
\]
Such domains arise naturally in fully-nonlinear PDE theory. See for
example \cite{caffarelli1985dirichlet,MR1284912}.
\end{example}

Lemma \ref{lem:stable-weak-invariance} and Lemma \ref{lem:stablepolynomialcone}
together identify the canonical positivity structure of stable polynomials.

\subsection{Gårding polynomial}
In this subsection, we formalize the geometric positivity structure underlying stable polynomials and extend it to a broader class of polynomials.

\begin{defn}[G\aa rding components]
\label{def:cone of polynomial}
For any nontrivial $f\in\R[\x]$, the \emph{Gårding component} $\C_{f}$ is defined to be the unique connected component of the set $\{f>0\}$ that satisfies PRT, if such a component exists. If no such component exists, we set $\C_{f}=\emptyset$. For the zero polynomial, we define $\C_{0}:=\R^{n}$.
\end{defn}

The uniqueness of $\C_{f}$, when it exists, follows from
Lemma \ref{lem:basic}.

This definition recovers all previously discussed cases. For any nontrivial
$f\in\S$, the proof of Lemma \ref{lem:stablepolynomialcone} provides
an explicit construction of its Gårding component. 

We  define \gar{} polynomials by the following geometric characterizations:
\begin{defn}[\gar{} polynomials]
\label{def:general garding} Let $n\geq1$ and $\x\in\R^{n}$. Define $$\GS_{n}^{0}:=\{f\equiv c:c\geq0\}.$$ 
Define the class of \emph{multi-affine Gårding polynomials} as
\[
\G_{n}:=\{f(\x)\in\R^{\mathfrak{a}}[\x],\ \mathcal{C}_{f}\not=\emptyset\}.
\]
The class of \emph{Gårding polynomials} is defined as follows: For \(n \ge 1\), define
\[
\GS_n
:=
\left\{
\mu^* g \in \mathbb{R}[\x] : \exists\, m\in\N,\ 
g \in \G_m,  \ 
\mu 
\text{ is strictly positive affine}
\right\}.
\]

We denote $\GS$, $\GS^{d}$ the set of all \gar{} polynomials,
the set of all degree $\leq d$ \gar{} polynomials respectively. 
\end{defn}

By convention, for $f\in\GS$, $\C_{f}=\R^{n}$ if and only if $f\equiv c\geq0.$ The zero polynomial is included.
Also, $\GS^{1}$ consists of affine-inear functions with nonnegative
coefficients. In particular, $\GS_{n}^{1}=\S_{n}^{1}.$
Many of the structural consequences of this definition, including closure
properties and recursive characterizations, will be developed in subsequent
sections.
\begin{rem}
Observe that for any $g = \mu^{*}(f) \in \GS$, with $f \in \G$ and $\mu \in \A_{++}$, 
the associated component satisfies $\C_{g} = \mu^{-1}(\C_{f}) \neq \emptyset$, and it passes PRT by Remark
\ref{rem:pullback} The existence of such a component constitutes the fundamental geometric condition underlying the G{\aa}rding property. In the multi-affine setting, this condition is also sufficient. However, by itself, it is too weak to guarantee that general G{\aa}rding polynomials are closed under important algebraic operations, such as polarization.
\end{rem}

\begin{rem}
\label{rem:basicpreservers} For any $\mu\in\A_{++}$, if $f\in\G$
and $\mu^{*}(f)\in\R^{\mathfrak{a}}[\x],$ then $\mu^{*}(f)\in\G$ by
Remark \ref{rem:pullback}. If $\mu$ is a translation,
permutation, or dilation in Definition \ref{def:space-transform},
then $\mu^{*}$ preserves $\G$. 
\end{rem}

\begin{rem}
\label{rem:group difference}
It is straightforward to see that, by definition, $\GS$ is preserved under pullback by maps in $\A_{++}$, while stable polynomials are preserved under pullback by maps in $\A_{+}$. However, $\S$ is not preserved under $\A_{+}$; see Lemma~\ref{lem:stable-weak-invariance}. This distinction is fundamental and will be illustrated in the next example and exploited in later sections.
\end{rem}

\begin{example}
\label{rem:general restriction}
To highlight a key difference between stable and Gårding polynomials, consider the restriction operation and the polynomial $f(x,y)=(x-1)y$. The polynomial $f$ is both stable and Gårding; however, its restriction $f(0,y)=-y$ is stable but not Gårding. The restriction map $(x,y)\to(0,y)$ is weakly positive linear, but not strictly positive linear in our sense.
\end{example}
We present more examples of \gar{} polynomials. 
\begin{example}
\label{exa:Let-.-We}Let $f(x,y)=axy+bx+cy+d$ and $\deg f\geq1$.
 $f\not\in\G$ if $a<0.$ If $a>0$, a direct computation shows that
$f\in\G$ if and only if 
\[
ad-bc\leq0.
\]
When $a=0$, then $f\in\G$ if $b,c\geq0$ and one of them is non-zero.
 $f$ is also stable.
\end{example}

\begin{example}
\label{exa:nonstable-garding}
We give a non-stable example of a G{\aa}rding polynomial. Let $f(x,y,z)=xyz-1$. Then 
\[
\mathcal{C}_{f}=\{(x,y,z)\in\mathbb{R}^3 : xyz>1,\ x>0,\ y>0,\ z>0\}
\]
satisfies PRT. Therefore $f\in\G$. However, $g(t)=f(t,t,t)$ has two nonreal complex conjugate roots. Hence $g$ is not stable, and by Lemma~\ref{lem:stablepolyclosurethe}, \ref{enu:part1}, $f$ is not stable.
\end{example}

In fact, our new class strictly extends that of real stable polynomials. 
\begin{thm}
\label{thm:stable-garding}One has $\S\subsetneq\GS$, $\GS^{2}=\S^{2}.$ 
\end{thm}

\begin{proof}
The inclusion $\S\subset\GS$ follows from Lemmas \ref{lem:stablepolyclosurethe},
\ref{lem:stablepolynomialcone} and Theorem \ref{thm:pol-restiction of stable}.
The strict inclusion is shown by Example \ref{exa:nonstable-garding}.

We now prove the degree 2 statement. Let $f\in\GS^{2}$ and $f\not\equiv0.$
Thus $\mathcal{C}_{f}\neq\mathbb{R}^{n}$. By Lemma \ref{lem:stablepolyclosurethe},
it suffices to show that for any $\mathbf{x}_{0}\in\mathcal{C}_{f}$
and $\mathbf{a}\in\Gamma_{n}^{+}$, $g(t):=f(\mathbf{x}_{0}+t\a)$
has only real roots. Suppose otherwise. Since $\deg g\leq2$, it follows that $g$
has no real roots. Since $g(0)>0,$ for all $t\in\mathbb{R}$,
$g(t)>0$. Therefore, by Definition~\ref{def:general garding} for
all $t\in\R,$ $\x_{0}+t\a\in\C_{f}.$ For any $\mathbf{z}\in\mathbb{R}^{n},$
for sufficiently large $T,$ $\mathbf{z}>\mathbf{x}_{0}-T\mathbf{a}$.
By PRT, $\mathbf{z}\in\mathcal{C}_{f}$. Hence $\mathbb{R}^{n}=\mathcal{C}_{f}$,
a contradiction. 

Therefore, $g$ has only real roots, and $f$ is stable. This completes
the proof.
\end{proof}

Further examples of Gårding polynomials arise naturally in probability
theory, matrix theory, graph and matroid theory, which will be discussed
in later sections.

In summary, we define a new class of polynomials via the existence
of the Gårding component and strictly positive affine invariance. This
framework strictly extends real stable polynomials while retaining
their essential geometric features. 

\part{Multi-affine G\aa rding Theory}

\section{Multi-affine Gårding polynomials}

In this section, we study multi-affine Gårding polynomials, which
form the geometric core of the general theory. 

Recall that a polynomial is multi-affine if it has degree at most
one in each variable. We denote by $\G\subset\GS$ the class of multi-affine
Gårding polynomials. Unless stated otherwise, all functions in this
part are assumed to be multi-affine.

We show that partial differentiation preserves the multi-affine G{\aa}rding property and analyze the induced inclusion relations between the G{\aa}rding components of a polynomial and those of its derivatives. We then introduce stronger geometric conditions that characterize the interior of $\G$ and establish the Rayleigh property for $\G_{+}$.

\subsection{Partial differentiation and a domain filtration}

We first establish that the class of multi-affine Gårding polynomials
is closed under partial differentiation and analyze the induced geometric
filtration on positivity regions. A key fact we exploit is that all
multi-affine polynomials are harmonic.
\begin{lem}
\label{lem:If--is-2}Assume that $f(\mathbf{x}) \in \G_n$ is nonzero. Then for any $\mathbf{x}_0 \in \mathcal{C}_f$ and any $\alpha \in \mathbb{N}^n$, we have
\[
\partial^\alpha f(\mathbf{x}_0) \ge 0.
\]
Furthermore, for each fixed $\alpha$,  $\partial^\alpha f(\mathbf{x}) > 0$ for all $\mathbf{x} \in \mathcal{C}_f$ or $\partial^\alpha f(\mathbf{x}) \equiv 0$.
\end{lem}

\begin{proof}
Since $f$ is multi-affine, $\pdv^{\alpha}f(\mathbf{x}_{0})=0$ if
$\alpha_{i}\geq2$. Therefore, with a fixed $\x_{0}\in\C_{f}$
we consider the Taylor expansion
\begin{equation}
f(\mathbf{x}_{0}+\mathbf{x})=\sum_{\alpha}\pdv^{\alpha}f(\mathbf{x}_{0})\mathbf{x}^{\alpha},\label{eq:add0}
\end{equation}
where the sum runs over all subsets of $[n]$. 

Suppose, for contradiction, that there exists $\alpha'\in\N^{n}$,
such that $\pdv^{\alpha'}f(\mathbf{x}_{0})<0$. Define $\y=\sum_{\alpha_{i}'=1}\mathbf{e}_{i}\in\R^{n}$.
Then, $\y\geq0$, and by (\ref{eq:add0}), for $t>0$,
\begin{equation}
f(\mathbf{x}_{0}+t\mathbf{y})=\pdv^{\alpha'}f(\mathbf{x}_{0})t^{|\alpha'|}+O(|t|^{|\alpha'|-1}).\label{eq:-21}
\end{equation}
For sufficiently large $t>0$, the leading term in (\ref{eq:-21})
dominates and the right hand side is then negative, contradicting
the PRT property of $\C_{f}$. Therefore, $\pdv^{\alpha}f(\mathbf{x}_{0})\geq0$.

 $\pdv^{\alpha}f(\x)$ is again multi-affine.
Therefore, $\pdv^{\alpha}f(\mathbf{x})$ is a nonnegative harmonic
function in the open set $\C_{f}$. By the maximum principle, 
$\pdv^{\alpha}f(\mathbf{x})>0$ for all $\x\in\C_{f}$, or $\pdv^{\alpha}f(\mathbf{x})\equiv0$
as a polynomial. 
\end{proof}
Next, we examine the behavior of multi-affine Gårding polynomials under
differentiation. For any $i\in[n],$ recall the deletion map $\pi_{i}$
in Definition \ref{def:space-transform}. For any $f\in\G(\x),$ since
$f$ is multi-affine, $\partial_{i}f(\x)$ is independent of $x_{i}$;
Hence, it can be viewed as a function of $\pi_{i}(\x)\in\R^{n-1}$, which,
for convenience we will also write as $\pdv_{i}f$. Therefore, 
\begin{equation}
\pdv_{i}f(\x)=\pdv_{i}f(\pi_{i}(\x)).\label{eq:add2}
\end{equation}

\begin{lem}
\label{lem:Suppose-that--2} Assume that $f\in\G$ is nonzero. Then for
any $i\in\n$, $\partial_{i}f\in\G.$ In addition:
\begin{enumerate}
\item Either $\partial_{i}f$ is identically zero, or $\pdv_{i}f\in\R^{\mathfrak{a}}[\pi_{i}(\x)]$
and $\C_{\partial_{i}f}\neq\emptyset;$ 
\item If $\pdv_{i}f$ is nontrivial, then $\mathcal{C}_{\pdv_{i}f}=\pi_{i}(\mathcal{C}_{f})\subset\mathbb{R}^{n-1}$;
\item \label{enu:C_fps}If $\pdv_{i}f$ is nontrivial, then $\C_{f}=\pi_{i}^{-1}(\C_{\pdv_{i}f})\cap\{f>0\}$.
\end{enumerate}
\end{lem}

\begin{proof}
Without loss of generality, assume $i=1$. Denote $f_{1}:=\partial_{1}f$.
If $f_{1}$ is a positive constant, then $\C_{\partial_{1}f}=\R^{n-1},$
and $f(\x)=cx_{1}+g(\x'),$ where $\x'=(x_{2},\cdots,x_{n}).$ Define
$\C:=\{(t,\x'),\ \x'\in\R^{n-1},\ t>-g(\x')\},$ which is  a connected
component of $\{f>0\}.$ For any $(t',\y')\in\C_{f},$ $(t,\y')\in\C_{f}\cap\C$
if $t>>t'$. Therefore, $\C_{f}=\C.$ The statement is valid.

Assume now that $\pdv_{1}f$ is not constant. Set $R=\pi_{1}\C_{f}$,
which is open, connected and also passes PRT in $\R^{n-1}$. Also
consider the non-constant multi-affine function $f_{1}|_{R}$, which
is harmonic. As a consequence, $f_{1}$ cannot obtain its local minimum
in $R.$ For any $\y=\pi_{1}(\x)\in R$, by Lemma \ref{lem:If--is-2}
and (\ref{eq:add2}), $f_{1}(\y)\geq0.$ We conclude that $f_{1}|_{R}>0$. 

Let $\mathcal{C}$ be the connected component of $\{f_{1}>0\}$ such
that $R\subset\C$. We claim  $\mathcal{C}=R$. We argue by contradiction,
using translation along a positive direction. If not, there exists
$\hat{\mathbf{p}}\in\mathcal{C}\backslash R\subset\R^{n-1}$. Then,
there exists $\hat{\mathbf{z}}\in\Gamma_{n-1}^{+}$ such that, $\hat{\mathbf{p}}+\hat{\mathbf{z}}\in\mathcal{C}\cap R$.
Let $t>0$ and define
\[
\mathbf{p}(t,s):=(t,\hat{\mathbf{p}}+s\hat{\mathbf{z}}).
\]
Let $c=\inf_{s\in[0,1]}\pdv_{1}f(\mathbf{p}(0,s))$. Then $c>0$ and
\begin{align*}
f(\mathbf{p}(t,s)) & =t\pdv_{1}f(\mathbf{p}(0,s))+f(\mathbf{p}(0,s))\geq ct+O(1).
\end{align*}
For sufficiently large $T$, $f(\mathbf{p}(t,s))>0$ for all $t\geq T$
and $s\in[0,1]$. Since $\C_{f}$ is a connected component, $\mathbf{p}(T,0)\in\mathcal{C}_{f}$.
Hence $\mathbf{\hat{p}}\in R$, which is a contradiction. Therefore,
$\mathcal{C}=R$. We have proved $\partial_{1}f\in\GS$, as well as
(1) and (2).

For (3), let $Q=\{f>0\}\cap\pi_{1}^{-1}(\C_{f_{1}})$. Notice $\C_{f}\subset Q$
by (2). On the other hand, if $\mathbf{p}\in Q\backslash\C_{f}$,
then $f_{1}(\mathbf{p})>0$ and $f(\mathbf{p})>0$. By the same argument
as in the proof of (2), for some $\mathbf{z}\in\Gamma_{n}^{+}$ and
$t>>0$, $\mathbf{p}+t\mathbf{e}_{1}+s\mathbf{z}\in\C_{f},$ for all
$s\in[0,1]$. Then, $\mathbf{p}+t\mathbf{e}_{1}\in\C_{f}$ for all
$t\geq 0 $ which violates the assumption that $\mathbf{p}\not\in\C_{f}$.
Hence, $Q=\C_{f}$. This completes the proof of (\ref{enu:C_fps}). 
\end{proof}
\begin{cor}
\label{cor:onemore}Assume that $f\in\G$ is nontrivial. Then for any multi-index
$\alpha,$ $\pdv^{\alpha}f\in\G$. Moreover, $\C_{\pdv^{\alpha}f}=\pi_{\alpha}(\C_{f})$
if $\pdv^{\alpha}f$ is nontrivial. 
\end{cor}
We next introduce a natural filtration of positivity domains associated
with a multi-affine Gårding polynomial, reflecting an underlying recursive
structure motivated by the subsolution concept in PDE theory and used in
later arguments.

Define for $k=0,\cdots,n$, 
\[
\Gamma_{n,k}^{+}:=\{\mathbf{x}\in\overline{\Gamma_{n}^{+}}:\mathbf{x}\text{ has at least }k\text{ positive entries}\}.
\]
It follows that $\Gamma_{n,k}^{+}\subset\Gamma_{n,k-1}^{+}$.
\begin{defn}
\label{def:derivative=00003D000020cone}If $\C\subset \mathbb{R}^n$ passes PRT, 
the \emph{$k$-th derived component} of $\C$ is defined as
\begin{equation}
\mathcal{C}_{}^{(k)}:=\{\mathbf{x}:\forall\mathbf{y}\in\Gamma_{n,k}^{+},\ \exists\ t>0\ {\rm {s.t.}\ }\mathbf{x}+t\mathbf{y}\in\mathcal{C}_{}\}.\label{eq:-20}
\end{equation}
\end{defn}

The terminology is justified by the following result:
\begin{prop}
\label{prop:Let--be-1}Let $f\in\G\backslash\{\text{const}\}$. Let
$I\subset[n]$ be the set where $\pdv_{i}f\not\equiv0.$ Then
\begin{equation}
\mathcal{C}_{f}^{(1)}=\bigcap_{i\in I}\pi_{i}^{-1}(\mathcal{C}_{\pdv_{i}f}).\label{eq:derived1cone}
\end{equation}
\end{prop}

\begin{proof}
By (\ref{eq:-20}), we have $\mathcal{C}_{f}^{(1)}=\bigcap_{i=1}^{n}R_{i},$
where 
\[
R_{i}=\{\mathbf{x}\in\mathbb{R}^{n}:\exists t>0,\mathbf{x}+t\mathbf{e}_{i}\in\mathcal{C}_{f}\}.
\]
By Lemma \ref{lem:Suppose-that--2}, $R_{i}=\pi_{i}^{-1}\left(\mathcal{C}_{\pdv_{i}f}\right)$
if $i\in I$ and $R_{i}=\C_{f}$ if $i\not\in I$. Since $\mathcal{C}_{f}\subset R_{i}$
for any $i\in I$, \ref{eq:derived1cone} holds. This completes the
proof. 
\end{proof}
\begin{rem}
By Proposition \ref{prop:Let--be-1}, every  degree $d$ polynomial $f\in\G$ 
determines a canonical increasing filtration of $\R^{n}$ by PRT domains:
\begin{equation}
\mathcal{C}_{f}=\mathcal{C}_{f}^{(0)}\subset\mathcal{C}_{f}^{(1)}\subset\cdots\subset\mathcal{C}_{f}^{(d)}=\mathbb{R}^{n}.\label{eq:-6-1-2}
\end{equation}Note also that $\mathcal{C}_{f}^{(k)}$ may be defined for any multi-affine $f$ provided that $\partial^{\alpha} f \in \G$ for all $|\alpha| = k$.
\end{rem}

Although the Gårding property is defined geometrically, in the
multi-affine setting, it admits an analytic verification. The resulting
recursive procedure takes the form of a non-linear optimization problem.
\begin{prop}
\label{prop:testing method}A multi-affine function $f(\x)$ of degree
$>1$ is Gårding if and only if the following holds:
\begin{enumerate}
\item For any $i\in[n]$, $\partial_{i}f\in\G.$
\item $f\leq0$ on the boundary of $
\C_{f}^{(1)}.$
\end{enumerate}
\end{prop}

\begin{proof}
One direction follows from Proposition \ref{prop:Let--be-1}. We prove
the converse.

 $\C_{f}^{(1)}$ is well defined with non-empty boundary.
For any $\x_{0}\in\C_{f}^{(1)}\subset\C_{f}^{(k)},$for $k\geq1$,
the Taylor expansion of $f$ at $\x_{0}$ has non-negative coefficients
except possibly the constant term. Thus, there exists $\y>\x$ such
that $\y\in\C_{f}^{(1)}$, and $f(\y)>0.$ 

Let $\mathcal{C}' = \mathcal{C}^{(1)}_f \cap \{f > 0\}$. By the same Taylor expansion argument, $\mathcal{C}'$ satisfies the PRT condition and is therefore connected by Lemma~\ref{lem:basic}. Let $\mathcal{C}$ be the connected component of $\{f > 0\}$ containing $\mathcal{C}'$. Assumption~(2) implies that $\mathcal{C} \subset \mathcal{C}_{f}^{(1)}$. Therefore, $\mathcal{C} = \mathcal{C}'$, and  $\mathcal{C}$ satisfies PRT. This shows that $f \in \G$.
\end{proof}

In practice, Proposition~\ref{prop:testing method} provides the basic recursive
verification procedure for multi-affine G\aa rding polynomials.

\subsection{Multi-affine NRT polynomials}

It may happen that $f(\x)\in\G[\x]$ is independent of certain variables.
Such a polynomial gives a degenerate elliptic operator in the theory
of PDE. To isolate a non-degenerate subclass of multi-affine Gårding
polynomials, we introduce the following
\begin{defn}
\label{def:NRTAn-open-set} A multi-affine polynomial $f$ is called
a negative-ray-terminating (NRT) polynomial, if it admits a connected
component of the set $\{f>0\}$ that is negative ray terminating. 
\end{defn}

In the multi-affine case, the NRT condition is strictly stronger than
the PRT requirement and has significant consequences for the  first derivatives.
\begin{lem}
\label{lem:Let--be}Let $f$ be a multi-affine polynomial and let
$\mathcal{C}$ be a connected component of $\{f>0\}$. The following
are equivalent:
\begin{enumerate}
\item \label{enu:NRTtoGar}$\mathcal{C}$ is NRT;
\item $f$ is \gar{} with $\C=\C_{f}$, and $f$ does not have redundant variables;
\item \label{enu:ConverselyNRT2part}$f$ is Gårding, $\C=\C_{f}$, and
$\pdv_{i}f(\mathbf{x})>0$ for all $\x\in\mathcal{C}$.
\end{enumerate}
\end{lem}

\begin{proof}
We prove $(1)\Rightarrow(2)\Rightarrow(3)\Rightarrow(1)$. 

Assume (1). Fix $\mathbf{x}_{0}\in\mathcal{C}$ and $i\in[n]$. Since
$f$ is multi-affine, 
\[
f(\mathbf{x}_{0}+t\mathbf{e}_{i})=\pdv_{i}f(\mathbf{x}_{0})t+f(\mathbf{x}_{0}).
\]
By the NRT property, there exists $t_{0}<0$ such that $\x_{0}+t_{0}\mathbf{e}_{i}\not\in\C$,
hence 
\[
0=f(\mathbf{x}_{0}+t_{0}\mathbf{e}_{i})=\pdv_{i}f(\mathbf{x}_{0})t_{0}+f(\mathbf{x}_{0}).
\]
Since $f(\x_{0})>0$ and $t_{0}<0,$ we conclude that $\pdv_{i}f(\mathbf{x}_{0})>0$.
Thus, all first derivatives are positive on $\C,$ $\C$ satisfies
PRT. Therefore, $f\in\G$ and has no redundant variables.

Next, assume (2). We argue by contradiction. Since $f\in\G$, $\pdv^{\alpha}f|_{\C_{f}}\geq0$
for all multi-index $\alpha$ by Lemma \ref{lem:If--is-2}. If for
some $\y\in\C_{f}$, $\pdv^{\alpha}f(\y)=0$, then $\pdv_{i}\pdv^{\alpha}f(\y)=0$
for all $i\in[n]$, since otherwise a neighborhood of $\x$ has a non-empty
intersection with $\{\pdv^{\alpha}f<0\}$. By induction, $\pdv^{\beta}f(\y)=0$
for all $\beta\geq\alpha$. Thus, if we can find $\x_{0}\in\C_{f}$
and $i_{0}$ such that $\pdv_{i_{0}}f(\x_{0})=0$, the Taylor expansion
at $\x_{0}$ has the form
\[
f(\x_{0}+\x)=\sum_{\alpha\in\N^{n},\alpha_{i_{0}}=0}c_{\alpha}\x^{\alpha}.
\]
Then $f$ is independent of $x_{i_{0}}$ which is a contradiction. 

Finally, assume (3). If $\C_{f}$ were not NRT, there would exist
$\x_{0}\in\C_{f}$ and $\v\geq0,$ $\v\neq0$, such that $\x_{0}+t\v\in\C_{f}$
for all $t\in\R.$ Restricting to the affine subspace $V:=\mathbf{x}_{0}+\text{span}\{\mathbf{e}_{i}:v_{i}>0\}$,
Lemma \ref{lem:For-any-.} we obtain a multi-affine \gar{} polynomial
whose Gårding component is all of $V$. Therefore, $f|_{V}$ is a
constant polynomial with all partial derivatives vanishing, contradicting
$\partial_{i}f>0$ on $\C_{f}.$ Therefore $\C_{f}$ is NRT.
\end{proof}
\begin{example}
A degree $n$ polynomial $f\in\G_{n}$ if and only if a connected
component of $\{f>0\}$ satisfies NRT. Suppose that the top homogeneous
part of $f$ is $a\sigma_{n}(\mathbf{x})$ for $a>0$. Then, $f$
is Gårding if there exists a component $\mathcal{C}$ of $\{f>0\}$
which is contained in $\{\mathbf{x}_{0}\}+\Gamma_{n}^{+}$ for some
$\mathbf{x}_{0}\in\mathbb{R}^{n}$. In fact, $\{\mathbf{x}_{0}\}+\Gamma_{n}^{+}$
satisfies NRT. Thus, if $\mathcal{C}\subset\{\mathbf{x}_{0}\}+\Gamma_{n}^{+}$,
then by Lemma \ref{thm:dominategardnew}, $f$ is Gårding. Conversely,
if $f$ is Gårding, by a translation $\hat{f}(\mathbf{x})=f(\mathbf{x}+\mathbf{x}_{0})$,
we may assume that  $\hat{f}$ has vanishing homogeneous degree $n-1$ part.
Then, we see that $\mathcal{C}_{\hat{f}}\subset\mathcal{C}_{\hat{f}}^{(n-1)}\subset\Gamma_{n}^{+}.$
This implies that $\mathcal{C}_{f}\subset\{\mathbf{x}_{0}\}+\Gamma_{n}^{+}$. 
\end{example}
Lemma~\ref{lem:If--is-2} can also be interpreted as follows: any \gar{} polynomial can be reduced to an NRT \gar{} polynomial after removing all redundant variables.
\begin{lem}
\label{lem::NRTpropirredIf-there-exists} Let $f\in\G_{n}$. Let $\mathbf{x}_{0}\in\mathcal{C}_{f}$
and set $I:=\{i\in[n]:\pdv_{i}f(\mathbf{x}_{0})=0\}$. Then, for $i\in I$,
$\pdv_{i}f\equiv0$, and $f$ is the pullback of an NRT \gar{} polynomial
in  the subspace $V=\text{span}\{\mathbf{e}_{j}:\pdv_{j}f(\mathbf{x}_{0})>0\}$. 
\end{lem}

For any $f\in\G\backslash\{\text{const}\}$, $f-c\in\G$.
By Lemma \ref{lem:basic} (5), we have the following. 
\begin{lem}
\label{lem:strong}Suppose $f$ is an NRT polynomial. Then so is $f-c$
for any $c>0$. 
\end{lem}

\begin{rem}
\label{rem:strong reduction} With notation as in
Lemma \ref{lem:strong},  for  any NRT polynomial $f$,
there is a sequence of NRT functions $h_{k}$ such that $h_{k}\to f$
and $\overline{\C_{h_{k}}}\subset\C_{f}\subset\C_{f}^{(1)}=\C_{h_{k}}^{(1)}$. 
\end{rem}

\subsection{Restriction to coordinate subspaces}

As noted in Remark \ref{rem:general restriction}, the restriction
operation does not preserve the Gårding property in general. Nevertheless, the restriction does preserve the Gårding property
in the following case.
\begin{lem}
\label{lem:splitting}Consider a nontrivial polynomial $f(\x)\in\G$
and $\a=(a_{1},\cdots,a_{n})$.
\begin{enumerate}
\item If $\a\in\C_{f}$, then the restriction $f(a_{1},x_{2},\cdots,x_{n})$
is a nontrivial Gårding polynomial;
\item For $\a\in\overline{\C_{f}}$, then the restriction $f(a_{1},x_{2},\cdots,x_{n})$
is Gårding; and, moreover either it is nontrivial or $f(\x)=(x_{1}-a_{1})g(x_{2},\cdots,x_{n}),$
for $g\in\G[x_{2},\cdots,x_{n}].$
\end{enumerate}
\end{lem}

\begin{proof}
For (1), let $\C=\pi_{1}(\{x_{1}=a_{1}\}\cap\C_{f}).$ Since $f(\a)>0,$ $\a\in\{x_{1}=a_{1}\}\cap\C_{f}$,
$\C$ is nonempty. It is straightforward to check that $\C$ is a
connected component of the positivity region of the restricted function passing PRT. The proof
is complete.

For (2), since $f(\x)\in\G,$ $\C_{f}$ satisfies PRT, so are $\overline{\C_{f}}$,
and $\pi_{1}(\{x_{1}=a_{1}\}\cap\overline{\C_{f}})\subset\R^{n-1}.$
Thus, for any $\y=(0,y_{2},\cdots,y_{n})\geq0$,
\begin{equation}
\a+\y\in\{x_{1}=a_{1}\}\cap\overline{\C_{f}}\neq\emptyset\label{eq:add21}
\end{equation}

If $\{x_{1}=a_{1}\}\cap\C_{f}\neq\emptyset$, we have $f|_{x_{1}=a_{1}}\in\G$.
Otherwise, 
\begin{equation}
\{x_{1}=a_{1}\}\cap\C_{f}=\emptyset.\label{eq:add22}
\end{equation}
By (\ref{eq:add21}) and (\ref{eq:add22}), we see $\a+\y\in\partial\C_{f}$
for all $\y=(0,y_{2},\cdots,y_{n})$, with $y_{i}\geq0$, which indicates
that $f(\x)$ has a factor of $(x_{1}-a_{1}).$ We have proved the
lemma.
\end{proof}
We also analyze the restriction to coordinate affine subspaces. 

For any subset $I=\{i_{1}<i_{2}<\cdots<i_{k}\}\subset[n]$ with
$k$ elements, define the inclusion map and the projection
map
\begin{align*}
\iota_{I} & :\mathbb{R}^{k}\to\mathbb{R}^{n},\ (x_{1},\cdots,x_{k})\mapsto\sum_{i\in I}x_{i}\mathbf{e}_{i},\\
\pi_{I} & :\mathbb{R}^{n}\to\mathbb{R}^{n-k},\x\mapsto\pi_{i_{1}}\circ\cdots\circ\pi_{i_{k}}(\x).
\end{align*}
We have the following lemma.
\begin{lem}
\label{lem:For-any-.}Suppose $f\in\G$. If $\pi_{I}(\mathbf{x}_{0})\in\pi_{I}(\mathcal{C}_{f})$,
then for $\mathbf{y}\in\mathbb{R}^{k}$, 
\[
g(\mathbf{y}):=f(\iota_{I}(\mathbf{y})+\mathbf{x}_{0})\in\G.
\]
In particular, if $\x_{0}\in\C_{f}^{(k_{0})}$ where $k_{0}\leq k$,
then $g(\y)\in\G$.
\end{lem}

\begin{proof}
If $g\equiv c\geq0$, then the result is trivial. Otherwise, denote
\[
R:=\{\mathbf{y}:\iota_{I}(\mathbf{y})+\mathbf{x}_{0}\in\mathcal{C}_{f}\},
\]
which is an open subset of $\{g(\mathbf{y})>0\}$. Since $\pi_{I}(\mathbf{x}_{0})\in\pi_{I}(\mathcal{C}_{f})$,
for $\mathbf{a}=\sum_{i\in I}\mathbf{e}_{i}\in\Gamma_{n,|I|}^{+}$,
there exists $t>>0$ such that $t\mathbf{a}+\mathbf{x}_{0}\in\C_{f}$.
Then $R$ is not empty and passes PRT . Thus, by Lemma \ref{lem:basic},
$R$ is connected. Suppose $\mathbf{y}_{0}\in\pdv R$. Then $\iota_{I}(\mathbf{y}_{0})+\mathbf{x}_{0}\in\pdv\mathcal{C}_{f}$
and hence $g(\mathbf{y}_{0})=0$. Thus, $R$ is both open and closed
in $\{g(\mathbf{y})>0\}$. Hence, $R$ is a connected component of
$\{g(\mathbf{y})>0\}$. Therefore, $g$ is Gårding. 
\end{proof}

\subsection{Multi-affine Gårding polynomials with nonnegative coefficients}

As an application, we now focus on a distinguished subclass of multi-affine
Gårding polynomials, namely those with nonnegative coefficients. This
subclass, denoted by $\G_{+}$, plays a central role in our linear
preserver theory and applications to probability theory and combinatorics. 

We begin by recording the basic stability properties of $\G_{+}$.
\begin{lem}
\label{lem:move-into-GA+}For any $f\in\G,$ and $\x_{0}\in\C_{f}$,
the translation $\text{T}_{\x_{0}}^{*}f\in\G_{+}$.
\end{lem}

\begin{proof}
By Lemma \ref{lem:If--is-2}, the claim follows from the
Taylor expansion of $f$ at $\x_{0}$. 
\end{proof}
\begin{lem}
\label{lem:G+restriction}If $f\in\G_{+},$ then for any $I\subset[n]$,
$\iota_{I}^{*}f\in\G_{+}$. 
\end{lem}
\begin{proof}
Since $f\in\G_{+},$ $f(\mathbf{0})\geq0.$
Then $\mathbf{0}\in\overline{\C_{f}}$, and the claim follows from Lemma \ref{lem:For-any-.}
and Lemma \ref{lem:splitting}.
\end{proof}

To formulate the analytic closure property of $\G_{+}$, we equip the spaces
$\mathbb{R}^{\mathfrak{a}}[\mathbf{x}]$ and $\mathbb{R}_{\kappa}[\mathbf{x}]$ with the compact--open topology, i.e., the topology of uniform convergence on compact subsets.
With this topology, $\G_{+}[\mathbf{x}]$ is a closed subset of $\mathbb{R}^{\mathfrak{a}}[\mathbf{x}]$.
\begin{lem}
\label{lem:G_+isclosed}$\G_{+}[\x]$ is closed in $\R^{\mathfrak{a}}[\x]$.
\end{lem}

\begin{proof}
We argue by contradiction. Suppose that there exists $\{f_{k}(\x)\}_{k=1}^{\infty}\subset\G_{+}[\x]$
converging uniformly on compact subsets to a multi-affine polynomial
$f$ which is not \gar{}. 

By continuity, $f(\x)\geq0$ for all $\x\geq0$. If $f$ is a constant,
then $f\in\G_{+}[\x]$. Hence, we assume that $f$ is not constant.
Since $f$ is multi-affine, it is harmonic, and by the maximum principle
we have $f>0$ on $\Gamma_{n}^{+}$. Let $\C$ denote the connected
component of $\{f>0\}$ containing $\Gamma_{n}^{+}$. 

We claim that for any $\w\in\C$ there exists $N\in\N$ such that
if $k>N$, then $\w\in\C_{f_{k}}$. Indeed, since $\C$ is connected,
there exists a continuous path $\gamma:[0,1]\to\C$ with $\gamma(0)=\mathbf{1}$
and $\gamma(1)=\w$. Set $a:=\min_{t\in[0,1]}f(\gamma(t)).$ Then
$a>0$. Uniform convergence on compact sets implies the existence
of $N\in\N$ such that if $k>N$, 
\[
\sup_{t\in[0,1]}|f_{k}(\gamma(t))-f(\gamma(t))|<\frac{a}{2},
\]
hence $\min_{t\in[0,1]}f_{k}(\gamma(t))>\frac{a}{2}$. Since
\textbf{$\mathbf{1}\in$$\C_{f_{k}}$}, it follows that $\gamma([0,1])\subset\C_{f_{k}}$
and $\w\in\C_{f_{k}}$.

Since \(f\) is not \gar{}, there exist \(\x_{0}\in\C\) and
\(\a\ge0\) such that
\[
f(\x_{0}+\a)=0.
\]
By the PRT property of \(f_k\) and the previous claim, we have
\(\x_0\in\C_{f_k}\) for all sufficiently large \(k\). 
By Lemma~\ref{lem:If--is-2}, the Taylor expansions of  \(f_k\) at $\x_0$ are nonnegative. So we have
\[
f(\x_{0}+\a)
=\lim_{k\to\infty}f_k(\x_{0}+\a)
\ge \lim_{k\to\infty}f_k(\x_0)
\ge \frac12 f(\x_0)
>0,
\]
contradicting \(f(\x_{0}+\a)=0\). Therefore \(f\) is \gar{}, and
\(\G_+\) is closed.
\end{proof}
\begin{rem}
$\G[\x]$ itself is not closed. For example, let $g_{k}=\frac{1}{k}x-1$
and let $k\to+\infty$. Then the limit $-1$ is not a \gar{} polynomial.
On the other hand, if $g_{k}(\x)\in\G[\x]$ and there exists a fixed
point $\x_{0}\in$$\cap_{k=j}^{\infty}\C_{g_{k}},$ then, by a translation
and Lemma \ref{lem:G_+isclosed}, the limit of $\{g_{k}\}$ is a \gar{}
polynomial. 
\end{rem}

Next, we recall the definition of a Rayleigh polynomial. 
\begin{defn}
\label{def:Rayleigh}A multi-affine polynomial $f$ is \emph{Rayleigh}
if it has nonnegative coefficients and for any $\mathbf{x}\in\Gamma_{+}^{n}$,
\begin{equation}
f(\mathbf{x})\pdv_{i}\pdv_{j}f(\mathbf{x})-\pdv_{i}f(\mathbf{x})\pdv_{j}f(\mathbf{x})\leq0\label{eq:-18}
\end{equation}
for each $i<j$. If (\ref{eq:-18}) holds for all $\mathbf{x}\in\mathbb{R}^{n}$,
then $f$ is called \emph{strong Rayleigh}. By \cite{BBL09}, a multi-affine
polynomial is strong Rayleigh if and only if it is stable. 
\end{defn}

We show that \gar{} polynomials are Rayleigh polynomials.
\begin{lem}
\label{lem:GardingToRayleigh}
\begin{enumerate}
\item \label{enu:GartoRay}For $i\not=j$, if $f\in\G$, and $\x\in\R^n$ satisfies $\pi_{ij}(\x)\in\pi_{ij}(\C_{f})$, then
\[
f(\x)\pdv_{i}\pdv_{j}f(\x)-\pdv_{i}f(\x)\pdv_{j}f(\x)\leq0.
\]
\item \label{enu:gartoRay2}If $f\in\G_{+}$, then $f$ is a Rayleigh polynomial. 
\item \label{enu:gartoRay3}If in a component $R$ of $\{f>0\}$ and for
any $\mathbf{x}_{0}\in R$, the translation $\T_{\x_{0}}^{*}f$ is
Rayleigh, then $f\in\G$. 
\item \label{enu:gartoRay4}If $f$ is a strong Rayleigh polynomial with
non-negative leading coefficients, then $f$ is Gårding. 
\end{enumerate}
\end{lem}

\begin{proof}
For (\ref{enu:GartoRay}), we prove for $i=1,\ j=2$.  Other cases
are similar. For $\pi_{12}(\x_{0})\in\pi_{12}(\C_{f})$, let 
\begin{align*}
g(x,y) & =f(\x_{0}+xe_{1}+ye_{2})\\
 & =\pdv_{12}f(\mathbf{x}_{0})xy+\pdv_{1}f(\mathbf{x}_{0})x+\pdv_{2}f(\mathbf{x}_{0})y+f(\mathbf{x}_{0}).
\end{align*}
Then, $g(x,y)\in \GS[x,y]$ by Lemma \ref{lem:For-any-.}.
The conclusion follows from Example \ref{exa:Let-.-We}.

For (\ref{enu:gartoRay2}), since $\mathbf{0}\in\overline{\C_f}$, by (\ref{enu:GartoRay}) and continuity,
$f$ is Rayleigh.

Suppose that $R$ is the connected component of $\{f>0\}$ given in
(\ref{enu:gartoRay3}). Then  $R+\overline{\Gamma_{n}^{+}}\subset R$
since $p(\mathbf{x})=f(\mathbf{x}_{0}+\mathbf{x})$ is Rayleigh and
has nonnegative coefficients for $\mathbf{x}_{0}\in R$. $f$ is Gårding.

If $f$ is strong Rayleigh, then $f$ is Rayleigh on a component of
$\{f>0\}$, and $f$ is Gårding. 
\end{proof}

\subsection{Truncation to top and bottom degrees}
Let $f \in \mathbb{R}[\mathbf{x}]$ have total degree $d$ and vanishing order $v$ at $\mathbf{0}$. The \emph{top- and bottom-degree parts} of $f$, denoted by $f^{\text{top}}$ and $f^{\text{bot}}$, are defined to be the homogeneous components of $f$ of degrees $d$ and $v$, respectively. The operations of truncating a polynomial to its top- or bottom-degree part are defined purely algebraically, while the preservation of the G{\aa}rding property follows from a limiting argument using the underlying topological structure.
\begin{prop}
\label{prop:truncateLet--be}Suppose that $f\in\G$. Then $f^{\mathrm{top}}\in\G_{+}.$
Furthermore, if $f\in\G_{+}$, then $f^{\text{bot}}\in\G_{+}$. 
\end{prop}

\begin{proof}
By Lemma \ref{lem:move-into-GA+}, $\text{T}_{\x_{0}}^{*}f\in\G_{+}$
for any $\x_{0}\in\C_{f}$. Since $(\text{T}_{\x_{0}}^{*}f)^{\text{top}}=f^{\text{top}}$,
it suffices to give a proof assuming $f\in\G_{+}$.

Let $d=\deg f$ and $v=\text{ord}_{\mathbf{0}}f$. Then
\[
f^{\mathrm{top}}(\x)=\lim_{k\to+\infty}k^{-d}f(k\x),\ f^{\text{bot}}(\x)=\lim_{k\to+\infty}k^{v}f\left(\frac{\x}{k}\right).
\]
Since dilation preserves \gar{} polynomials, our conclusion 
follows from Lemma \ref{lem:G_+isclosed}.
\end{proof}

\section{Decomposition and binary relations}

In this section, we develop a decomposition theory for multi-affine
Gårding polynomials. This decomposition clarifies an additional geometric
and algebraic structure and naturally leads  to the introduction of
two associated binary relations.

The first relation, \emph{proper position}, is a direct analog
of the corresponding notion for stable polynomials. The second is
a \emph{domination} relation, which appears to be specific to the
Gårding setting. For both relations, we establish geometric, algebraic,
and analytic characterizations and establish their basic properties.
Our approach differs conceptually from those typically used in the
theory of stable polynomials. Readers may find it instructive to compare
the present results and methods with their counterparts in the stable
setting, cf. \cite{BB09Annals,BB09I,Wagner11}.

This section provides the technical backbone for the constructions
and arguments developed later in the paper and is therefore necessarily
technical. 

\subsection{Definitions and a dichotomy theorem}

We begin by introducing two binary relations that encode the two alternatives
in the restriction dichotomy.
\begin{defn}[Domination]
\label{def:dominatesGA}Let $g(\mathbf{x})\in\G_{n}$. Let $f\in\R^{\mathfrak{a}}[\x]$
be a nontrivial multi-affine polynomial. We say that \emph{$g$ dominates
$f$},  denoted by  $f\vartriangleleft g$, if 
\[
f|_{\C_{g}}>0.
\]
By convention, $0\vtl g$ for any $g\in\G$. 
\end{defn}

\begin{rem}
When $f,g\in\G,$ $f\vartriangleleft g$ if and only if $\C_{g}\subset\C_{f}.$
In  Definition~\ref{def:dominatesGA}, $f$ is not required to be Gårding,
which is important for later applications.
\end{rem}

Next, we introduce the analogue of proper position in the stable theory.
\begin{defn}[Proper Position]
\label{def:properposition}For nontrivial $f(\textbf{x}),g(\textbf{x})\in\G[\x]$,
we say that $(f,g)$ is in  a \emph{proper position}, written   $f\prec g$,
if 
\[
h(\mathbf{x},y)=f(\textbf{x})y+g(\textbf{x})\in\G[\x,y].
\]
By convention, $0\prec g\prec0$ for any $g\in\G$.
\end{defn}

The following theorem gives a precise formulation of the restriction
dichotomy in terms of the two relations introduced above.
\begin{thm}
\label{thm:Suppose-that-is}Suppose that $h(\mathbf{x},y)=f(\mathbf{x})y+g(\mathbf{x})\in\G_{n+1}$
where $f$ and $g$ are nontrivial. Then exactly one of the following
alternatives holds:
\begin{enumerate}
\item \label{enu:-is-Garding}$f\in\G$ and $g|_{\C_{f}}<0$; equivalently,
$-g\vartriangleleft f$. 
\item \label{enu:or--are} Both $f,g\in\G$, and $\mathcal{C}_{g}=\mathcal{C}_{f}\cap\{g>0\}$.
In particular, $f\vartriangleleft g$ and $f\prec g$. 
\end{enumerate}
\end{thm}

\begin{proof}
 $f(\x)=\partial_{y}h$ is necessarily Gårding. Since $g$
is nontrivial and $\mathcal{C}_{f}$ is open, either $g|_{\C_{f}}<0$
or $G:=\C_{f}\cap\{g>0\}\neq\emptyset$. Assuming the latter, we claim
that $g\in\G$ and $G=\C_{g}.$ 

By Lemma \ref{lem:Suppose-that--2}, we have $\C_{h}=\{h>0\}\cap\pi_{y}^{-1}(\C_{f}),$
where $\pi_{y}:\R^{n}\times\R\to\R^{n}$ is the projection. Consequently,
\[
G=\pi_{y}(\C_{h}\cap\{y=0\}).
\]
It follows that $G$ is connected and passes the PRT. 

If $G$ were not a component of $\{g>0\}$, then there would exist
a continuous curve $\mathbf{x}(t)\not\subset G$ with $g(\x(t))>0$
for $t\in[0,1]$, and $\mathbf{x}(0)\in G$. However, this would imply
$h(\x(t),0)>0$ for $t\in[0,1]$ and hence 
\[
(\x(t),0)\subset\C_{h}\cap\{z=0\}=G,
\]
 a contradiction. Therefore, $g\in\G$ and $G=\C_{g}$. 
\end{proof}
Theorem \ref{thm:Suppose-that-is} highlights both the similarities
and the structural differences between Gårding and stable polynomials. 
Geometrically, this dichotomy shows that two-term extensions either force
a strict sign on $\mathcal{C}_f$ or carve out $\mathcal{C}_g$ as an
intersection of positivity regions.

\subsection{Domination relation}

This subsection is devoted to the domination relation. We start with
a technical lemma.
\begin{lem}
\label{lem:usefullemma} Let $\mathcal{A}$ be a path connected
 set and $\mathcal{\C}$ open connected such that $\mathcal{C}\cap\mathcal{A}\not=\emptyset$. Then,
the following are equivalent:
\begin{enumerate}
\item $\mathcal{A}\backslash\C\not=\emptyset$;
\item $\mathcal{A}\cap\pdv\mathcal{C}\not=\emptyset$.
\end{enumerate}
If further assume that $\A$ is open connected set, and $\C$ is a component of $\{f>0\}$ for  $f(\x)\in\R^\mathfrak{a}[\x]$,  then there is a third equivalence:
\begin{enumerate}[resume]
\item $\exists\, \mathbf{x}\in\mathcal{A}$ such that $f(\mathbf{x})<0$.
\end{enumerate}
\end{lem}

\begin{proof}
Choosing a continuous path $\gamma$ from $p\in \A\cap \C$ to $q\in\A\setminus \C$ in $\A$ gives an intermediate point $p'\in \A\cap\partial \C$ and proves $(1)\Rightarrow(2)$.  $(1)\Leftarrow(2)$ and $(3)\Rightarrow(1)$ (when $\A$ is open and $\mathcal{C}=\{f>0\}$) are both immediate.

(2)$\Rightarrow$(3). If $\pdv\C\not=\emptyset$, then $f$ is not
constant. Pick $\mathbf{x}_{0}\in\mathcal{A}\cap\pdv\C$. Since $\mathcal{A}$
is open, a ball $B_{r}(\mathbf{x}_{0})$ of radius $r$ is contained
in $\mathcal{A}$. Since $f$ is multi-affine and hence harmonic,
according to the maximum principle, $\{f<0\}\cap B_{r}(\mathbf{x}_{0})\not=\emptyset$. 
\end{proof}
\begin{lem}
\label{lem:basic-domination}Let $g\in\G_{n}$. 
\begin{enumerate}
\item \label{enu:If--isdominate}If $h\in\G_{n}$ , $f\vartriangleleft g$
and $g\vartriangleleft h$, then $f\vartriangleleft h$.
\item \label{enu:If--anddominate}If $f\vartriangleleft g$ and $g\vartriangleleft f$,
then $g$ and $f$ are proportional.
\item \label{enu:If--anddominate-1}If $f\vtl g$, then for any $\alpha\subset[n],$
$\pdv_{\alpha}f\vtl\pdv_{\alpha}g$. 
\item \label{enu:If--anddominate-1-1}If $f\vtl g$, then $\deg f\leq\deg g$.
\item \label{enu:For-any-,partialsub}For any multi-index $\alpha,$ $\partial_{\alpha}g\vartriangleleft g.$
In particular, for any $i\in[n]$, $\pdv_{i}g\vartriangleleft g$.
\item \label{enu:convexcombinationdomiation}Suppose that $f_{1},f_{2}$ are both multi-affine
and $f_{i}\vartriangleleft g$ for each $i$. For $c_{1},c_{2}\geq0$,
$\sum_{i=1}^{2}c_{i}f_{i}\vtl g$. 
\end{enumerate}
\end{lem}

\begin{proof}
(\ref{enu:If--anddominate}) follows from Lemma \ref{lem:propotional}.
(\ref{enu:If--anddominate-1}) follows from the fact that $\pi_{\alpha}(\C_{g})\subset\pi_{\alpha}\{f>0\}$.
(\ref{enu:If--anddominate-1-1}) follows from (\ref{enu:If--anddominate-1}).
The remaining statements follow directly from the definition.
\end{proof}

We can construct new \gar{} polynomials using the domination relation.

\begin{defn}\label{def:rho}
Let $f$ and $g$ be multi-affine polynomials with $g\in\G$. Define
\[
\rho_{g}(f):=\sup\{c>0:\{g-cf>0\}\cap\C_{g}\neq\emptyset\}
\]
Notice that $\rho_g(f)=\infty$ if $\deg g>\deg f.$
\end{defn}

\begin{thm}
\label{thm:dominategardnew}Let $g\in\G$ and let $f$ be a multi-affine
polynomial. If $f\vartriangleleft g$, and $\rho_{g}(f)>0$. then
$g-cf\in\G$ for every $c\in[0,\rho_{g}(f))$. 
\end{thm}

\begin{rem}
If $\rho_{g}(f)=\infty,$ then the conclusion of Theorem \ref{thm:dominategardnew}
holds for all $c\geq0$. 
\end{rem}

\begin{proof}
Assume $f\vtl g$ and $\rho_{g}(f)>0$. If $\mathcal{C}_{g}$ is not
NRT, then by Lemma \ref{lem::NRTpropirredIf-there-exists} there exists
an index $i\in[n]$ such that $\pdv_{i}g\equiv0$, and $\C_{g}=\pi_{i}(\C_{g})\times\R.$
In this case the domination condition implies that for any $\mathbf{x}_{0}\in\C_{g}$
and $t\in\R$, $f(\mathbf{x}_{0}+t\mathbf{e}_{i})>0$. Thus, $\pdv_{j}f(\mathbf{x}_{0})=0$,
and $f$ is independent of $x_{i}$. By restricting to a coordinate subspace, we may assume that $\mathcal{C}_{g}$ satisfies NRT.

Fix $0 < c < \rho_{g}(f)$, and set
\[
h(\mathbf{x}) := g(\mathbf{x}) - c f(\mathbf{x}).
\]
Let $\mathbf{y} \in \{h > 0\} \cap \mathcal{C}_{g}$, and let $\mathcal{C}$ denote the connected component of $\{h > 0\}$ containing $\mathbf{y}$. We claim that $\mathcal{C} \subset \mathcal{C}_{g}$. Otherwise, by Lemma~\ref{lem:usefullemma}, there exists $\mathbf{x}_1 \in \mathcal{C} \cap \partial \mathcal{C}_{g}$. Then
\[
0 < h(\mathbf{x}_1) = g(\mathbf{x}_1) - c f(\mathbf{x}_1) \le 0,
\]
a contradiction.

Therefore, $\mathcal{C}\subset\mathcal{C}_{g}$. By Lemma \ref{lem:basic}
\eqref{enu:NRTlabel}, $\mathcal{C}$ is NRT, and it follows from
Lemma \ref{lem:Let--be} that $h(\mathbf{x})$ is Gårding.
\end{proof}
\begin{example}
Suppose that $g(\x)\in \G[\x]$ and
$\mathcal{C}_{g}\subset\Gamma_{n}^{+}$. Let $p(\x)\in\R_{+}^{\mathfrak{a}}[\x]$ be any multi-affine
polynomial with  degree less than $n$. Then $p\vtl g$. By Theorem \ref{thm:dominategardnew},
for any $c>0$, $g-cp$ is Gårding.
This construction was used essentially  in our previous work~\cite{FANG2024109867}.
\end{example}

New Gårding polynomials may also be constructed as follows. 
\begin{example}
\label{exa:product}Let $g(\x)\in\G[\x]$ and $f(\y)\in\G[\y]$, and
assume $q\vartriangleleft g,\ p\vartriangleleft f.$ Then $p(\x)q(\y)\vtl g(\x)f(\y)$,
and for $c\in[0,\rho_{g}(q)\cdot\rho_{f}(p))$
\[
h(\x,\y):=g(\x)f(\y)-cq(\x)p(\y)\in\G[\x,\y],
\]
%In particular, for integers $m> l\ge0$, one has $\sigma_{l}(\y)\vtl\sigma_{m}(\mathbf{y}),$ and $\rho_{\sigma_{m}}(\sigma_{l})=\infty$. Thus, for any $c\geq0$

%\begin{equation} g(\mathbf{x})\sigma_{m}(\mathbf{y})-c\sigma_{l}(\mathbf{y})q(\mathbf{x})\in\G[\x,\y].\label{eq:-4} \end{equation}

\end{example}

Next, we have the following algebraic characterization of the domination
relation.

\begin{thm}
\label{thm:iffdominate}Let $g$ and $f$ be nonzero multi-affine
polynomials with $g\in\G$. Then $f(\mathbf{x})\vtl g(\mathbf{x})$
if and only if $g(\mathbf{x})y-f(\mathbf{x})\in\G[\mathbf{x},y]$
and the set $\{f>0\}$ contains a positive affine hyper-octant.
\end{thm}

\begin{proof}
Assume first that $f\vtl g$. Clearly, $1\vtl y$, with $\rho_{y}(1)=\infty.$
By Example \ref{exa:product}, $f(\mathbf{x})\vtl y\cdot g(\mathbf{x})$
with $\rho_{g\cdot y}(f)=+\infty$. Hence, $g(\mathbf{x})y-f(\mathbf{x})\in\G[\mathbf{x},y]$
by Theorem \ref{thm:dominategardnew}. Since $f\vtl g,$ the positivity
set $\{f>0\}$ contains $\C_{g}$, which in turn contains an affine positive
hyper-octant.

Conversely, suppose that $g(\mathbf{x})\cdot y-f(\mathbf{x})\in\G_{n+1}$
and that $f$ is positive on an affine hyper-octant. Applying the dichotomy
of Theorem \ref{thm:Suppose-that-is}, only the first alternative
can occur. Hence, $f\vtl g$. 
\end{proof}
Finally, we list an analytical result for future use. 
\begin{prop}
\label{51} Suppose that $f\vtl g$ and $f$ is nontrivial. Then
the function $\phi(\x):=\text{\ensuremath{g}}(\x)/f(\x)$ is non-decreasing
for $\x\in\mathcal{C}_{\text{\ensuremath{g}}}$ in $\overline{\Gamma_{n}^{+}}$
directions. In particular, $\deg g\geq\deg f$.
\end{prop}

\begin{proof}
It suffices
to verify the sign of $\pdv_{i}\phi(\mathbf{x})$. A straightforward computation gives
\begin{equation}
\pdv_{i}\phi(\mathbf{x})=\frac{f(\x)\pdv_{i}g(\x)-g(\x)\pdv_{i}f(\x)}{f(\x)^{2}}.\label{eq:-5}
\end{equation}
By Theorem \ref{thm:dominategardnew}, (\ref{thm:iffdominate}), the
assumption $f\vtl g$ implies 
\[
h(\x,y):=g(\x)y-f(\x)\in\G_{n+1}.
\]
For $\x\in\C_{g}=\C_{\pdv_{y}h}$, the Rayleigh property of $h$,
Lemma \ref{lem:GardingToRayleigh}, implies
\begin{align*}
0 & \geq h(\x,y)\pdv_{y}\pdv_{i}h(\x,y)-\pdv_{y}h(\x,y)\pdv_{i}h(\x,y)\\
 & =g(\x)\pdv_{i}f(\x)-f(\x)\pdv_{i}g(\x).
\end{align*}
Thus, $\pdv_{i}\phi\geq0$ for $\x\in\C_{g}$. 
\end{proof}

\subsection{The proper position relation}

Now we  focus on the proper position relation.
\begin{thm}
\label{proper-position} Let $f,g\in\G$ be two nontrivial multi-affine
\gar{} polynomials. The following are equivalent: 
\begin{enumerate}[label=(\Alph*)]
\item \label{enu:AAA}$f\prec g$.
\item \label{enu:BBB}$\mathcal{C}_{f}\cap\{g>0\}=\mathcal{C}_{g}$, equivalently, $\mathcal{C}_{g}\subset \C_f$ and 
$g\leq0$ on $\overline{\mathcal{C}_{f}}\backslash\mathcal{C}_{g}$.
\item \label{enu:CCC}For any $\x\in\mathcal{C}_{f}$, and any $i\in[n]$,
\begin{equation}
(g\pdv_{i}f-f\pdv_{i}g)|_{\mathbf{x}}\leq0.\label{add-1}
\end{equation}
\end{enumerate}
\end{thm}

\begin{thm}
\label{lem:proper-position-basic}
The following properties of the proper position relation hold:
\begin{enumerate}[label=(\roman*)]
\item \label{enu:111}If $f\prec g$, then $f\vtl g$.
\item \label{enu:222}For $f(\x)\in\G_{n,+}$ and $i\in[n]$, $\partial_{i}f\prec f$.
\item \label{enu:333}If $f\prec g$, then $\partial_{\alpha}f\prec\partial_{\alpha}g$.
\item \label{enu:444}If $f\prec g$, then $\deg f\leq\deg g\leq\deg f+1$.
\item \label{enu:555}If $f\prec g$ and $c\geq0$, then $cf+g\in\G_{n}$.
\item \label{lem:subordinatetoboth}
Assume  $f,g_{1},g_{2},\cdots,g_{k}\in\G_{n}\backslash\{0\}$ with $f(\x)\prec g_{i}(\x)$ for all $i\in[k]$. Let $g=\sum_{i=1}^{k}\theta_{i}g_{i}$ with $\theta_{i}\geq0$. Then $g\in\G_{n}$. Moreover, if $g$ is nontrivial,
\begin{equation}
f(\x)\prec g(\x).
\label{eq:dominateconvexcombination}
\end{equation}
\end{enumerate}
\end{thm}

The proofs of Theorems~\ref{proper-position} and \ref{lem:proper-position-basic} are intertwined. We first establish several elementary implications among the conditions of Theorem~\ref{proper-position}, along with the basic properties listed in Theorem~\ref{lem:proper-position-basic}. We then complete the proofs of both theorems.

\begin{proof}[Proof of Theorem~\ref{proper-position}, \ref{enu:AAA}$\Rightarrow$\ref{enu:BBB}]
Since $f\prec g$, $$h(\x,y):=f(\x)y+g(\x)\in\G[\x,y].$$ Statement~(B) follows from the dichotomy in Theorem~\ref{thm:Suppose-that-is}. 
\end{proof}

\begin{lem}
\label{lem:propotional}
For $n\geq1$, if $p_{1},p_{2}\in\G_{n}$ and $\C_{p_{1}}=\C_{p_{2}}$, then $p_{1}$ and $p_{2}$ are proportional.
\end{lem}

\begin{proof}
We argue by induction on $d(p_{1},p_{2}):=\max\{\deg p_{1},\deg p_{2}\}$. The case $d\leq 1$ is trivial. Suppose the statement holds for all pairs of multi-affine Gårding polynomials whose maximal degree is at most $d-1$. Consider $p_{1},p_{2}\in\G$ with $d(p_{1},p_{2})=d$ and $\C_{p_{1}}=\C_{p_{2}}$. By passing to a coordinate subspace if necessary, we may assume that both $p_{1}$ and $p_{2}$ are NRT.

Fix $i\in[n]$ and consider $\partial_{i}p_{1}$ and $\partial_{i}p_{2}$. By Lemma~\ref{lem:Suppose-that--2}, we have $\C_{\pdv_{i}p_{1}}=\C_{\pdv_{i}p_{2}}$. Hence, by the induction hypothesis, there exists $c=c(i)>0$ such that
\[
\pdv_{i}p_{1}(\hat{\x})=c\,\pdv_{i}p_{2}(\hat{\x}).
\]

For any $\hat{\x}_{0}\in\C_{p_{1}}$, since $\C_{p_{1}}=\C_{p_{2}}$ are NRT, there exist $t\in\R_{+}$ and $\hat{\x}_{1}\in\partial \C_{p_{1}}$ such that $\hat{\x}_{0}=t\mathbf{e}_{i}+\hat{\x}_{1}$. Then
\[
p_{1}(\hat{\x}_{0})=\int_{0}^{t}\pdv_{i}p_{1}(\hat{\x}_{1}+s\mathbf{e}_{i})\,ds
=\int_{0}^{t}c\,\pdv_{i}p_{2}(\hat{\x}_{1}+s\mathbf{e}_{i})\,ds
= c\,p_{2}(\hat{\x}_{0}).
\]

Since $\hat{\x}_{0}$ is arbitrary, $p_{1}=c\, p_{2}$ on the open set $\C_{p_{1}}$. Hence $p_{1}\equiv cp_{2}$. This completes the induction and the proof.
\end{proof}
\begin{proof}[Proof of \ref{enu:111}, \ref{enu:222}, \ref{enu:333} of Theorem
\ref{lem:proper-position-basic}]

\ref{enu:111} follows directly from the definition.

We prove \ref{enu:222}. Let $f(\x)\in\G_{n,+}$, write $f_{i}(\x)=\pdv_{i}f(\x)$
and $f^{i}(\x)=f(\pi_{i}(\x))$. By Lemmas \ref{lem:Suppose-that--2}
and \ref{lem:For-any-.}, $f_{i},f^{i}\in\G_{n-1,+}$ . Since $f=x_{i}f_{i}+f^{i}\in\G$,
$f_{i}\prec f^{i}$. Consider the multi-affine polynomial 
\[
F(\x,y):=f_{i}(\x)y+f(\x)=f(\x+y\mathbf{e}_{i}).
\]
Since the linear map $(\x,y)\to\x+y\mathbf{e}_{i}$ is strictly positive,
$F\in\G_{n+1,+}[\x,y]$. Thus, $f_{i}\prec f$ by definition, proving
\ref{enu:222}.

Since $\G$ is preserved by partial derivatives, the relation $\prec$
is also preserved under taking partial derivatives, proving \ref{enu:333}.
\end{proof}
\begin{lem}
\label{lem:preplemma}Suppose that condition \ref{enu:BBB} of Theorem
\ref{proper-position} holds for $f,g\in\G_{n}\backslash\{0\}.$ Then
for all  $\x\in\R^{n}$ with $\pi_{\alpha}(\x)\in\overline{\pi_{\alpha}(\C_{f})}\backslash\C_{\pdv^{\alpha}g}$, $\pdv^{\alpha}g(\x)\leq0.$
\end{lem}

\begin{proof}
We argue by contradiction. If $\pi_{\alpha}(\x)\in\overline{\pi_{\alpha}\mathcal{C}_{f}}\backslash\C_{\pdv_{\alpha}g}$
with $\pdv_{\alpha}g(\x)>0$. By (\ref{eq:-21}), for sufficiently
large $t$ and $\y=\sum_{i\in\alpha}\mathbf{e}_{i}$, we have
$\x+t\y\in\C_{f}\cap\{g>0\}.$
Hence $\x+t\y\in\C_{g}$ by condition \ref{enu:BBB}. However $\pi_{\alpha}(\x+t\y)=\pi_{\alpha}(\x)\in\C_{\partial^{\alpha}g}\cap\overline{\pi_{\alpha}\C_{f}}$,
a contradiction. We have proved the lemma.
\end{proof}
\begin{lem}
\label{lem:deg 0 case}Let $g\in\G_{n}\setminus\{0\}$and $f=c>0$
be constant. If condition \ref{enu:BBB} of Theorem \ref{proper-position}
holds, then $\deg g\leq1.$
\end{lem}

\begin{proof}
Suppose, for contradiction, that $\deg g\geq 2$. Then there exists a multi-index $\alpha$ such that $\deg \pdv^{\alpha}g=2$. By Theorem~\ref{thm:stable-garding}, $\pdv^{\alpha}g$ is real stable, and hence $\{\pdv^{\alpha}g>0\}$ has two connected components.

On the other hand, since $f\equiv c>0$, we have $\C_{f}=\R^{n}$. Because $f\prec g$, condition~(B) holds, and by Lemma \ref{lem:preplemma}, $\partial^{\alpha}g(\x)\leq 0$ for $\pi_{\alpha}(\x)\in \pi_{\alpha}(\R^{n})\backslash \C_{\partial^{\alpha}g}$. Consequently, $\{\partial^{\alpha}g>0\}$ has only one connected component, a contradiction.

Therefore, $\deg g\leq 1$.
\end{proof}
\begin{lem}
\label{lem:usefullemmaprec}If $f,g\in\G_{n}\backslash\{0\}$
and the condition \ref{enu:BBB} of Theorem \ref{proper-position}
holds, then
\begin{enumerate}
\item \label{enu:For-any-leq0}for any $\alpha\subset[n]$ and any $\x$
with $\x\in\overline{\C_{\pdv^{\alpha}f}}\backslash\C_{\pdv^{\alpha}g}$,
$\pdv^{\alpha}g(\x)\leq0$;
\item \label{enu:deg}$\deg g\leq\deg f+1$. 
\end{enumerate}
\end{lem}

\begin{proof}
If $\partial^{\alpha}f\not\equiv0$, then (\ref{enu:For-any-leq0})
follows from Lemmas \ref{lem:preplemma} and \ref{lem:Suppose-that--2}.
If $\partial^{\alpha}f\equiv0,$ then there exists a multi-index $\beta$
such that $\beta\leq\alpha$, $|\beta|<|\alpha|$ and $\partial^{\beta}f=c>0$.
In particular, (\ref{enu:For-any-leq0}) holds for index $\beta$.
By Lemma \ref{lem:deg 0 case}, deg$\partial^{\beta}g\leq1$. Therefore
$\deg\partial^{\alpha}g\leq0$, which implies that $\overline{\C_{\pdv^{\alpha}f}}\backslash\C_{\pdv^{\alpha}g}=\emptyset$.
(\ref{enu:For-any-leq0}) is proved.

(\ref{enu:deg}) then follows from Lemma \ref{lem:deg 0 case} and
(\ref{enu:For-any-leq0}).
\end{proof}
\begin{lem}
\label{lem:constant multiple}Suppose that $f,g\in\G_{n}\backslash\{0\}$, $f$ is independent of $x_{i}$, and $g$ is NRT. If 
condition \ref{enu:BBB} of Theorem \ref{proper-position} holds,
then $f$ is a constant multiple of $\pdv_{i}g$.
\end{lem}

\begin{proof}
Assume that $f$ is independent of $x_{i}$ and $g$ is NRT. By (\ref{enu:For-any-leq0}),
$\pdv_{i}g(\x)\leq0$ for $\x\in\R^{n-1}\backslash\C_{\pdv_{i}g}$. 

Consider first the case where $\pdv_{i}g\equiv c$. Since $g$ is
NRT, $c>0$. Since $g\in\G$, pick $\a=(a_{1},\cdots,a_{n})\in\C_{g}$
we may write $g(\x)=c(x_{i}-a_{i})+h(\x),$ where by Lemma \ref{lem:splitting},
$h(\x)\in\G_{n-1}$, which is independent of $x_{i}$. Therefore,
$c\prec h$. By Lemma \ref{lem:usefullemmaprec} (\ref{enu:deg}),
$\deg h\leq1$, hence $\deg g\leq1.$ Condition \ref{enu:BBB} implies
that $f\vtl g$, and thus $\deg f\leq1$ by Theorem \ref{51}. If
$f$ were not constant, $\C_{g}\subset\C_{f}$ would force $f$ and
$g$ have the same gradient, contradicting $\pdv_{i}g>0=\pdv_{i}f$.
Thus, $f$ is a constant and therefore proportional to $\pdv_{i}g$. 

If instead $\deg\pdv_{i}g\geq1,$ then by Lemma \ref{lem:Suppose-that--2},
$\C_{\pdv_{i}g}=\pi_{i}(\C_{g})\subset\pi_{i}(\C_{f})$. Since $f$
is independent of $x_{i}$, we have $f|_{\C_{\partial_{i}g}}>0$ and
$f\vtl\pdv_{i}g$. If $\C_{\partial_{i}g}\subsetneq\pi_{i}(\C_{f})$,
by Lemma \ref{lem:usefullemma}, there exists $\mathbf{x}\in$$\mathcal{C}_{f}\backslash\pi_{i}^{-1}(\mathcal{C}_{\pdv_{i}g})$
such that $\pdv_{i}g(\mathbf{x})<0$ and $\x\not\in\C_{g}$. Then
for sufficiently negative $T<<0$, $g(\mathbf{x}+T\mathbf{e}_{i})>0.$
Hence,
\[
\mathbf{x}+T\mathbf{e}_{i}\in(\mathcal{C}_{f}\backslash\mathcal{C}_{g})\cap\{g>0\},
\]
 contradicting  \ref{enu:BBB}. Hence, $\mathcal{C}_{\pdv_{i}g}=\mathcal{C}_{f}$,
and by Lemma \ref{lem:propotional}, $\pdv_{i}g$ and $f$ are proportional. 
\end{proof}
\begin{prop}
\label{prop:precpositivecombination}Suppose that $f,g\in\G_{n}$
and condition (B) of Theorem \ref{proper-position} holds. Then for
any $c\geq0$, 
\[
h:=g+cf\in\G_{n}.
\]
\end{prop}

\begin{proof}
The claim is trivial if $g\equiv0$ or $f\equiv0$, so assume that
both are nontrivial. Then  condition (B) implies $f\vtl g$. As in
the proof of Theorem \ref{thm:dominategardnew}, by restriction to
subspaces we may assume that $g$ is NRT . 

If $f$ is also NRT, for any $\mathbf{x}\in\pdv\mathcal{C}_{f}$,
$h\leq0$. On the other hand, since $\mathcal{C}_{g}\subset\{h>0\}\not=\emptyset$,
let $\C$ be the connected component of $\{h>0\}$ containing $\C_{g}.$
Then $\mathcal{C}_{}\subset\mathcal{C}_{f}$. Since $\mathcal{C}_{f}$
is NRT, $\C$ is also NRT. By Lemma \ref{lem:Let--be}, this implies
that $h\in G_{n}$. 

If $f$ is not NRT, then, without loss of generality, we may assume
that $f$ is independent of $x_{1}$. By Lemma \ref{lem:usefullemmaprec},
$f=c_{1}\pdv_{1}g$ for some constant $c_{1}>0$ and 
\[
h(\mathbf{x})=g(x_{1}+cc_{1},x_{2},\cdots,x_{n}).
\]
Hence, $h$ is also \gar{}. 
\end{proof}
We are ready to finish the proof of Theorem \ref{proper-position}.
\begin{proof}[Proof of Theorem \ref{proper-position}: \ref{enu:BBB}$\Longrightarrow\ref{enu:CCC}$]
 We note  that if \ref{enu:BBB} holds, then $f\vtl g$. As in the
proof of Theorem \ref{thm:dominategardnew}, by removing coordinates
if necessary, we may assume that $g$ is NRT. 

By Remark \ref{rem:strong reduction}, by a limiting argument, it
is sufficient to prove the statement with the additional assumption
that $\partial\C_{g}\subset\C_{g}^{(1)}$ and $\pdv\mathcal{C}_{g}\subset\mathcal{C}_{f}$.
Fix $\mathbf{x}_{0}\in\pdv\mathcal{C}_{g}$ and an index $i\in[n]$.
Set
\[
\phi(t):=\frac{g(\x_{0}+t\mathrm{e}_{i})}{f(\x_{0}+t\mathrm{e}_{i})},
\]
for $t>T_{0}:=\inf\{t:\mathbf{x}_{0}+t\mathbf{e}_{i}\in\mathcal{C}_{f}\}.$
By hypothesis $\phi(0)=0$, $\phi'(0)>0$, and $T_{0}<0$.

Condition \ref{enu:CCC} is equivalent to the fact that $\phi(t)$
is non-decreasing for $t>T_{0}$. By \ref{enu:BBB}, $\C_{g}\subset\C_{f}$;
hence $\phi(t)>0$ for $t>0$. 

Since $\phi'(0)>0,$ set
\[
T:=\inf\{t^{*}>T_{0}:\phi'(t)>0,\ t\in(t^{*},0]\}.
\]

If $T=-\infty$, then $T_{0}=-\infty$, which implies that $f(\x_{0}+t\mathrm{e}_{i})=c>0$.
Therefore, $\phi(t)$ is a linear function with $\phi'(t)=\phi'(0)>0.$ 

Otherwise, we claim $T=T_{0}$. By definition, $\phi'(T)=0.$ Since
$0<\int_{T}^{0}\phi'(t)dt=0-\phi(T),$ $-\phi(T)>0$. Therefore, by
Proposition \ref{prop:precpositivecombination},
\[
p(\mathbf{x}):=g(\mathbf{x})-\phi(T)f(\mathbf{x})\in\G.
\]
Set $\x_{T}:=\x_{0}+T\mathbf{e}_{i}$. Since $\phi'(T)=0$,
\[
\pdv_{i}p(\x_{T})=\pdv_{i}g(\x_{T})-\phi(T)\pdv_{i}f(\x_{T})=0.
\]
On the other hand, by the definition of $\x_{T},$$p(\x_{T})=0$.
Since $p$ is multi-affine, $p(\x_{0}+t\mathbf{e}_{i})\equiv0$, which
is a contradiction to the fact that $\phi'(0)>0$. Hence, $T=T_{0}$
and $\phi$ is non-decreasing in $(T_{0},0).$ Together with Proposition
\ref{51}, $\phi$ is non-decreasing for $t\in(0,\infty).$ 
\ref{enu:CCC} is established.
\end{proof}

\begin{proof}[Proof of Theorem \ref{proper-position}: \ref{enu:CCC}$\Longrightarrow\ref{enu:AAA}$]
 Assume that condition \ref{enu:CCC} holds. We claim that $\mathcal{C}_{g}\subset\mathcal{C}_{f}$.
Suppose not. By Lemma \ref{lem:usefullemma}, there exists $\mathbf{x}_{0}\in\mathcal{C}_{g}\setminus\mathcal{C}_{f}$
with $f(\x_{0})<0$. For any $\mathbf{v}\in\Gamma_{n}^{+}$, the ray
$\{\mathbf{x}_{0}+t\mathbf{v}\}$ intersects with $\C_{f}$, so there
exists $t>0$ such that $f(\mathbf{x}_{0}+t\mathbf{v})=0$. Since
$\x_{0}\in\C_{g}$ and $g$ is nontrivial, $g(\mathbf{x}_{0}+s\mathbf{v})>0$
for all $s>0$, then letting $s\to t^{-}$ yields $g/f\to+\infty$,
contracting the monotonicity of $\phi.$ We have proved $\mathcal{C}_{g}\subset\mathcal{C}_{f}$.

Denote $h:=f(\x)y+g(\x)$ and denote $$\C=\{(\mathbf{x},y),\ \mathbf{x}\in\mathcal{C}_{f},\ y>-g(\mathbf{x})/f(\mathbf{x})\}.$$
$\C$ is a connected component of $\{h>0\}$. For any $(\mathbf{x}_{0},y_{0})\in\C$
and $\mathbf{v}\in\Gamma_{1+n}^{+}$, $h(\x,y)=f(\x)(y+g(\x)/f(\x))$
is a multiple of two positive non-decreasing functions along $\mathbf{v}$.
In other words $h((\mathbf{x}_{0},y_{0})+t\mathbf{v})>0$ for all
$t>0$. Therefore, $\C$ satisfies PRT, which indicates that $h$
is \gar{}, and $f\prec g.$ The proof is complete.
\end{proof}
\begin{proof}[Proof of Theorem \ref{lem:proper-position-basic} \ref{enu:444}
\ref{enu:555}, \ref{lem:subordinatetoboth}]
 \ref{enu:444} follows from \ref{enu:111}, Proposition \ref{51},
Theorem \ref{proper-position} and Lemma \ref{lem:usefullemmaprec}.
\ref{enu:555} follow from Theorem \ref{proper-position} and Proposition
\ref{prop:precpositivecombination}. 

We prove \ref{lem:subordinatetoboth}. It is sufficient to prove the
case for $k=2$. We assume that $\theta_{1},$$\theta_{2}\not=0$
or otherwise the statement is trivial. Let $g=\theta_{1}g_{1}+\theta_{2}g_{2}$. 

If $f$ is not NRT, then for some $i\in[n]$, $f$ is independent
of $x_{i}$, by Lemma \ref{lem:constant multiple}, for some $c_{1},c_{2}\geq0$,
$\pdv_{i}g_{k}=c_{k}f$ for $k=1,2$, and
\[
\pdv_{i}g=(\theta_{1}c_{1}+\theta_{2}c_{2})f.
\]
Hence if $\sum_{k=1}^{2}\theta_{k}c_{k}>0$, $f\prec g$ and we are
done. If $\sum_{k=1}^{2}\theta_{k}c_{k}=0$, then $g$ is also independent
of $x_{i}$. Therefore, we may restrict the discussion to a subspace
on which $f$ is NRT. We first prove $g\in\G$.

Since \(f\prec g_i\), Theorem~\ref{proper-position} yields
\[
\mathcal C_{g_i}=\mathcal C_f\cap\{g_i>0\},
\qquad i=1,2.
\]
In particular,
\[
g_i\le 0 \qquad \text{on } \partial\mathcal C_f.
\]
Hence
\[
g=\theta_1 g_1+\theta_2 g_2 \le 0
\qquad \text{on } \partial\mathcal C_f.
\]
Since each \(\mathcal C_{g_i}\) satisfies
PRT, by Lemma~\ref{lem:basic}, their intersection is nonempty.  
Let \(\mathcal C\) be the connected component of \(\{g>0\}\) containing their intersection. We claim that
\[
\mathcal C \subset \mathcal C_f.
\]
Indeed, if not, then since  
\(\mathcal C\) is connected, \(\mathcal C\) would meet
\(\partial \mathcal C_f\). But \(g>0\) on \(\mathcal C\), whereas
\(g\le 0\) on \(\partial \mathcal C_f\), a contradiction. Hence
\(\mathcal C\subset \mathcal C_f\).

Since \(\mathcal C_f\) is NRT, every open subset of \(\mathcal C_f\) is
NRT by Lemma~\ref{lem:basic}. Thus \(\mathcal C\) is NRT, and
Lemma~\ref{lem:Let--be} implies that \(g\in \G_n\). Applying
condition~\ref{enu:CCC} to each \(g_i\) and summing the resulting
inequalities, we use Theorem~\ref{proper-position} to conclude that \(f\prec g\).
 
This completes the proof of \ref{lem:subordinatetoboth}, and the proof of Theorem~\ref{lem:proper-position-basic}.

\end{proof}

\section{Linear preservers of multi-affine \gar{} polynomials}\label{sec:linearpreserver}

In this section, we study linear transformations that preserve the class of multi-affine \gar{} polynomials. The results  will later be extended to general Gårding polynomials via polarization and specialization. Our approach follows the philosophy of Borcea--Brändén for stable polynomials \cite{BB09I} and Brändén-Huh for Lorentzian polynomials \cite{BrandenHuh20}.

Let $T:\R_{d}^{\mathfrak{a}}[x_{1},\cdots,x_{n}]\to\R^{\mathfrak{a}}[y_{1},\cdots,y_{m}]$ be a linear transformation.  We say that $T$ preserves $\G_{+}$ if, for any $f\in\GS_{n,+}^{\mathfrak{a}}$, one has $T(f)\in\G_{m,+}$.

Here we restrict the discussion to polynomials with nonnegative coefficients. By Lemma~\ref{lem:move-into-GA+}, for $f\in\G$ and $\x_{0}\in\C_{f}$, we have $T_{\x_{0}}^{*}f=f(\x_{0}+\x)\in\G_{+}$. By Lemma~\ref{lem:G+restriction}, the restriction operation preserves $\G_{+}$, while in general the restriction operation does not  preserve $\G$.

Following similar constructions in the stable polynomial theory, we
give the following:
\begin{defn}[Symbol]
For a linear transform $$T:\R^{\mathfrak{a}}[x_{1},\cdots,x_{n}]\to\R^{\mathfrak{a}}[y_{1},\cdots,y_{m}],$$
its symbol is a polynomial with $m+n$ variables defined by 
\[
\Sym_{T}(\mathbf{y},\mathbf{u})=\sum_{\alpha\subset[n]}T(\mathbf{x}^{\alpha})\mathbf{u}^{[n]\backslash\alpha}=\sum_{\alpha\subset[n]}T(\mathbf{x}^{\alpha})\pdv^{\alpha}\mathbf{u}^{[n]}.
\]
Equivalently, $\text{sym}_{T}(\mathbf{y},\mathbf{u})=T(\prod_{i=1}^{n}(x_{i}+u_{i})).$ 
\end{defn}

%\begin{example}For the identity map, its symbol is $\prod_{i=1}^{n}(y_{i}+u_{i})$, which is  real stable. \end{example}

The main theorem of this section is the following:

\begin{thm}
\label{thm:GardSymbolTransf}
Suppose that $T:\R^{\mathfrak{a}}[x_{1},\cdots,x_{n}]\to\R^{\mathfrak{a}}[y_{1},\cdots,y_{m}]$ is a linear transformation on multi-affine functions. If $\Sym_{T}(\y,\u)\in\G_{n+m,+}$, then $T$ preserves $\G_{+}$.
\end{thm}

We emphasize that Theorem~\ref{thm:GardSymbolTransf} provides a sufficient
symbol criterion for preservers in this setting. 
Our proof is similar to those of related results for stable and Lorentzian polynomials. 
However, in order to apply our more general concepts of binary relations, it is more geometric.

To prove  Theorem~\ref{thm:GardSymbolTransf}, we need the following lemma, which is similar to~\cite[Lemma~3.3]{BrandenHuh20}.
\begin{lem}
\label{lem:Lem3.3BrandenHuh20}
Suppose that $f(x_{1},x_{2},\hat{\x})\in\G_{n,+}$. Then
\begin{equation}
h(x_{2},\hat{\x}):=(\pdv_{1}f)(0,x_{2},\hat{\x})+(\pdv_{2}f)(0,0,\hat{\x})\in\G_{n,+}.
\label{eq:hisgard}
\end{equation}
\end{lem}
\begin{proof}
We have the following decomposition of multi-affine functions:
\[
f(\x)=x_{1}\pdv_{1}f(x_{2},\hat{\x})+g_{1}(x_{2},\hat{\x})
= x_{1}\bigl(x_{2}\pdv_{12}f(\hat{\x})+g_{2}(\hat{\x})\bigr)
+ g_{1}(x_{2},\hat{\x}),
\]
where
\begin{equation}
\pdv_{1}f(x_{2},\hat{\x})=x_{2}\pdv_{12}f(\hat{\x})+g_{2}(\hat{\x}).
\label{eq:add11}
\end{equation}

Then
\[
h(x_{2},\hat{\x})=x_{2}\pdv_{12}f(\hat{\x})+g_{2}(\hat{\x})+\pdv_{2}g_{1}(\hat{\x}).
\]
Note that $\partial_{1}f\prec g_{1}$, which, by Lemma~\ref{lem:proper-position-basic}, implies $\pdv_{12}f\prec\pdv_{2}g_{1}$. Also, by~\eqref{eq:add11}, we have $\pdv_{12}f\prec g_{2}$. Thus,~\eqref{eq:hisgard} follows from Lemma~\ref{lem:proper-position-basic} \ref{lem:subordinatetoboth}.
\end{proof}
Now we prove Theorem~\ref{thm:GardSymbolTransf}.
\begin{proof}[Proof of Theorem~\ref{thm:GardSymbolTransf}]
Suppose $\Sym_{T}(\mathbf{y},\mathbf{u})\in\G_{n+m,+}$. We use a new set of variables to write
\[
\Sym_{T}(\mathbf{y},\mathbf{u})=\sum_{\alpha\subset[n]}T(\mathbf{w}^{\alpha})\mathbf{u}^{[n]\backslash\alpha}.
\]
By Lemma~\ref{lem:basic}\,\ \ref{enu:directproduct}, the product of two multi-affine G\aa rding polynomials in disjoint variable sets is again 
By Lemma~\ref{lem:basic}\,\ \ref{enu:directproduct}, the product of two multi-affine G\aa rding polynomials with disjoint variable sets is G\aa rding. Thus,G\aa rding. Thus,
$\Sym_{T}(\mathbf{y},\mathbf{u})\cdot f(\mathbf{x})\in\G_{+}.$
We iterate Lemma~\ref{lem:Lem3.3BrandenHuh20} over the pairs $(u_{i},x_{i})$, $i=1,\cdots,n$, to obtain
\[
\sum_{\alpha\subset[n]}T(\mathbf{w}^{\alpha})\pdv^\alpha f(\mathbf{x})\in\G_{n+m,+}.
\]
Finally, we apply Lemma~\ref{lem:G+restriction} to obtain
\[
\left[\sum_{\alpha\subset[n]}T(\mathbf{w}^{\alpha})\pdv_{\alpha}f(\mathbf{x})\right]_{\mathbf{x}=0}
= T(f)\in\G_{m,+}.
\]
\end{proof}

Since real stable multi-affine polynomials are \gar{}, the following
corollary is the consequence of Theorem \ref{thm:stable-garding}
and Theorem \ref{thm:GardSymbolTransf}:
\begin{cor}
\label{cor:transformwithrealstablesymb}If $\Sym_{T}$ is real-stable with non-negative coefficients,
then $T$ preserves $\G_{+}$. 
\end{cor}

\begin{cor}\label{cor:preserverofG}
Suppose that  $T$ preserves $\G_{+}$. If there exists $\mathbf{a}>0$ and $T$ commutes with translations $T^*_{t\mathbf{a}}$ for $t\in \R$,  then $T$ preserves $\G$.
\end{cor}

\begin{proof}
For any $f\in\G$, using PRT, there exists a large enough $s>0$ such that $\a_0:=s\mathbf{a}\in \C_f$. By Lemma \ref{lem:move-into-GA+}, $T_{\a_{0}}^{*}f\in\G_{+}$.
Thus, by Theorem~\ref{thm:GardSymbolTransf}, 
\[
Tf=T_{-\a_{0}}^{*}\circ T\circ T_{\a_{0}}^{*}f\in\G.
\]
\end{proof}
We list some examples of Gårding linear preservers.
\begin{example}[Partial symmetrization]
\label{exa:partialsymme}Let $\theta\in[0,1]$. Define  
\[
\Phi_{\theta}^{1,2}(f)=(1-\theta)f(x_{1},x_{2},\cdots,x_{n})+\theta f(x_{2},x_{1},\cdots,x_{n}).
\]
Then by \cite[Exercise 4.4]{Wagner11}, the symbol of $\Phi_{\theta}^{1,2}\in \S_+$.
 $\Phi_\theta^{1,2}$ commutes with $T_{t\mathbf{1}}^*$, where $\mathbf{1}$ is the all-one vector. 
By Corollary \ref{cor:transformwithrealstablesymb} and Corollary  \ref{cor:preserverofG}, $\Phi_{\theta}^{1,2}$
preserves $\G$.
\end{example}

\begin{example}[Full symmetrization]
\label{exam:fullsymm} For $n\geq m\in\N$, define
the linear operator 
\[
T_{\text{sym},n}:\R_{m}^{\mathfrak{a}}[x_{1},\cdots,x_{m}]\to\R_{n}^{\mathfrak{a}}[y_{1},\cdots,y_{n}],\ \textbf{x}^{\alpha}\mapsto\frac{1}{\binom{n}{|\alpha|}}\sigma_{|\alpha|}(\textbf{y}),
\]
for any multi-index $\alpha$.  By \cite[Theorem 1.1]{BB09I}, $\Sym_{T_{\text{sym},n}}\in\S_+$.  Since $T_\text{sym,n}$ also commutes with $T_{t\mathbf{1}}^*$, $T_{\text{sym,n}}$  preserves $\G$. 
\end{example}

%In summary, for multi-affine Gårding polynomials, we have developed algebraic theories mirroring those of stable polynomials, providing a foundation for extending the theory to general Gårding polynomials in the next section.

\part{General \gar{} Theory}

\section{General \gar{} polynomial}\label{sec:General--polynomials}
We now extend the multi-affine theory to general polynomials. The main
difficulty is that the defining geometric condition is not directly preserved
under standard algebraic operations such as polarization. The strategy is to
reduce the general case to the multi-affine setting via polarization and to
recover the geometric structure through a recursive analysis of partial
derivatives.

By Remark~\ref{composition}, the class of Gårding polynomial is closed under strictly positive affine transformations.

One of the main results of this section is Theorem~\ref{thm:kappa-symmetric-realization}, which establishes the second characterization in Theorem~\ref{thm:newmain}.

\subsection{Symmetric realizations}\label{subsec:symmetric-realizations}
Let $\kappa\in\N^{n}$. Recall that
\[
\R_{\kappa}[x_{1},\cdots,x_{n}]
=\{f\in\R[x_{1},\cdots,x_{n}]:\deg_{x_{i}}f\leq\kappa_{i}\ \text{for all } i\},
\]
and that $\R^{\mathfrak{a}}[x_{1},\cdots,x_{n}]$ denotes the subspace of multi-affine polynomials. We write
\[
\GS_{n}^{\kappa}=\GS_{n}\cap\R_{\kappa}[\x],\qquad
\GS^{\mathfrak{a},\kappa}[\x]=\G\cap\R_{\kappa}[\x].
\]

For $g=\sum_{\alpha}c_{\alpha}x^{\alpha}\in\R_{\kappa}[\x]$, the $\kappa$-polarization of $g$ is defined by
\begin{equation}
(\pol_{\kappa}g)(x_{11},\cdots,x_{1\kappa_{1}},\cdots,x_{n1},\cdots,x_{n\kappa_{n}})
:=\sum_{\alpha\leq\kappa}c_{\alpha}\prod_{i=1}^{n}
\frac{\sigma_{\alpha_{i}}(x_{i1},\cdots,x_{i\kappa_{i}})}{\binom{\kappa_{i}}{\alpha_{i}}},
\label{eq:-11-2}
\end{equation}
where $\sigma_{\alpha_{i}}$ denotes the elementary symmetric polynomial of degree $\alpha_{i}$. The associated $\kappa$-projection (or $\kappa$-diagonal specialization) operator is given by
\[
\proj_{\kappa}(f):=f\big|_{x_{ij}=x_{i}},
\]
and satisfies $\proj_{\kappa}\circ\pol_{\kappa}=\mathrm{Id}$ on $\R_{\kappa}[\x]$.

Recall that a general G\aa rding polynomial $g$ is defined as a pullback
\[
g=\mu^{*}(f),
\]
where $f$ is a multi-affine G\aa rding polynomial and $\mu$ is a strictly positive affine map. We refer to the pair $(f,\mu)$ as a \emph{realization} of $g$. Thus, if $f=\pol_{\kappa}(g)\in\G_{\kappa}$ and $\mu$ is the diagonal map
\[
\mu:(x_{1},x_{2},\cdots,x_{n})\mapsto(\overbrace{x_{1},\cdots,x_{1}}^{\kappa_{1}},\cdots,\overbrace{x_{n},\cdots,x_{n}}^{\kappa_{n}}),
\]
then $(f,\mu)$ is called a $\kappa$-symmetric realization of $g$. Such realizations provide natural and convenient multi-affine models for general G\aa rding polynomials.

\begin{rem}
\label{rem:kappa-bound}
Let $\kappa_{g}=(\deg_{x_{1}}g,\ldots,\deg_{x_{n}}g)$ denote the multi-degree of $g$. If $g=\proj_{\kappa}(f)$ is realized as a $\kappa$-projection of a $\kappa$-symmetric multi-affine polynomial $f$, then necessarily $\kappa\ge\kappa_{g}$. Therefore, any $\kappa$-symmetric realization of $g$ requires at least $\deg_{x_{i}}g$ copies of the variable $x_{i}$ for each $i$.
\end{rem}

The $\kappa$-projection map $\proj_{\kappa}$ from the $\kappa$-symmetric multi-affine polynomials in $\G$ onto $\GS_{m}$ is surjective. We have the following realization theorem.
\begin{thm}
\label{thm:kappa-symmetric-realization} Let $g\in \GS_{m}$
be a general G\aa rding polynomial of positive degree. Then there
exist a multi-index $\kappa\in\N^{m}$ and a $\kappa$-symmetric
multi-affine polynomial $h\in\G$ such that 
\[
g\;=\;\proj_{\kappa}(h).
\]
 $h$ provides a $\kappa_{}$-symmetric realization
of $g$. 
\end{thm}

\begin{proof}
For any $g\in\GS_{m}[\x],$ by Definition \ref{def:general garding},
there exists $n\in\N$, $f\in\G$, and a positive affine transform
$\mu$ such that $g=\mu^{*}(f).$ Equivalently, there exist $\mathbf{a}_{1},\cdots,\mathbf{a}_{m}\in\overline{\Gamma_{n}^{+}}$,
$A=(\mathbf{a}_{1},\cdots,\mathbf{a}_{m})$ with $\sum_{i}a_{ij}>0$
and $\mathbf{b}_{}=(b_{1},\cdots,b_{n})\in\mathbb{R}^{n},$ such that
\[
g(x_{1},\cdots,x_{m})=f(\sum_{j=1}^{m}\mathbf{a}_{j}x_{j}+\mathbf{b}).
\]
Let $\mathbf{t}=(t_{ij})_{i\in[n],j\in[m]}$
and define a strictly positive affine map $\nu:\R^{mn}\to\R^{n}:$
\[
\nu(\t):=(\sum_{j=1}^{m}a_{1j}t_{1j}+b_{1},\cdots,\sum_{j=1}^{m}a_{nj}t_{nj}+b_{n}).
\]
 We observe that $\nu$ is strictly positive because $\mu$ is strictly
positive. Therefore 
\begin{align}
\tilde{g}(\mathbf{t}):=f(\nu(\t))\in\G_{m\times n}.\label{eq:tildeg-1}
\end{align}
Since $A$ is strictly positive,  choose a sufficiently large $T>0$
such that if $t_{ij}>T$, then $(t_{ij})\in\C_{\tilde{g}}$. Let $\mathbf{t}_{0}=T\cdot\mathbf{1}_{n\times m}$.
Then, by a translation, $\T_{\mathbf{t}_{0}}^{*}(\tilde{g})(\mathbf{t})=\tilde{g}(\mathbf{t}+\mathbf{t}_{0})\in\G_{+}[\mathbf{t}]$. 

Let $\kappa_{i}\geq n$ and $\kappa=(\kappa_{1},\cdots,\kappa_{m})\in\N^{m}.$
By Example \ref{exam:fullsymm}, for each $i,$ the full symmetrization
operator $T_{\text{sym},\kappa_{i}}$with respect to variables $t_{1i},\cdots,t_{ni}$
preserves $\G_{+}$. Therefore, for 
\[
\y=(y_{11},\cdots,y_{\kappa_{1}1},\cdots,y_{1m},\cdots,y_{\kappa_{m}m})\in\mathbb{R}^{|\kappa|},
\]
we have
\[
h(\mathbf{y}):=\T_{(-\mathbf{t}_{0})}^{*}\circ T_{\text{sym},\kappa_{m}}\circ\cdots\circ T_{\text{sym},\kappa_{1}}\circ\T_{\mathbf{t}_{0}}^{*}(\tilde{g})(\mathbf{y})\in\G_{|\kappa|}[\y].
\]
Then, $\proj_{\kappa}(h)=g$ and $\pol_{\kappa}(g)=h\in\G$. We have
finished the proof. 
\end{proof}

\subsection{Closure under partial derivatives and domain filtration }\label{subsec:PRT-filtration}
The symmetric realization theorem (Theorem~\ref{thm:kappa-symmetric-realization}) provides additional closure operations for general \gar{} polynomials. We show that partial derivatives preserve general \gar{} polynomials.

\begin{thm}
\label{thm:diff-closure}
Let $g$ be a G\aa rding polynomial, and let $\alpha\in\N^{n}$ be a multi-index. Then $\partial^{\alpha}g$ is G\aa rding. Moreover, the corresponding G\aa rding components satisfy
\[
\C_{\partial^{\alpha}g}\supseteq \C_{g}.
\]
\end{thm}

\begin{proof}
By Theorem~\ref{thm:kappa-symmetric-realization}, there exist multi-index $\kappa$ and $f\in\G_{\kappa}$ such that $g=\proj_{\kappa}(f)$. Note
\[
\partial_{i}g(x_{1},\cdots,x_{n})
=\proj_{\kappa}\!\bigl[\kappa_{i}(\partial_{i1}\pol_{\kappa}g)\bigr]
=\proj_{\kappa}\!\bigl[\kappa_{i}(\partial_{i1}f)\bigr]\in\G_{\kappa}.
\]
This proves the claim for the first derivatives. The inclusion of components follows directly from the projection. The general statement follows by induction on $|\alpha|$.
\end{proof}
As in the multi-affine case, we define a filtration of G{\aa}rding components.
\begin{defn}
[PRT domain filtration]\label{defprop:PRT-filtration} Let $g\in\mathcal{\GS}$
be a G\aa rding polynomial of total degree $d\geq1$. For each integer
$k\ge0$, define the $k-$th derived Gårding component, or the $k-$th
derived component, as
\[
\C_{g}^{(k)}\;:=\;\bigcap_{\substack{\alpha\in\mathbb{Z}_{\ge0}^{n}\\
|\alpha|=k
}
}\C_{\partial^{\alpha}g},
\]
where, by convention, $\C_{\partial^{\alpha}g}=\mathbb{R}^{n}$ whenever
$\partial^{\alpha}g\equiv c\geq0$.
\end{defn}

By Theorem \ref{thm:diff-closure} the family $\{\C_{g}^{(k)}\}_{k\ge0}$
is well defined and forms a nested filtration 
\begin{equation}
\C_{g}^{(0)}\;\subset\;\C_{g}^{(1)}\;\subset\;\C_{g}^{(2)}\;\subset\;\cdots\;\subset\;\C_{g}^{(d-1)}\;\subset\;\C_{g}^{(d)}\;=\;\mathbb{R}^{n}.\label{eq:nested cone}
\end{equation}
In the special case of $k=d-1$, the defining inequalities come from
nontrivial affine-linear polynomials $\partial^{\alpha}g=\a_{\alpha}\cdot\x-\b_{\alpha}$
with $|\alpha|=d-1,$and hence $\C_{g}^{(d-1)}$is a polyhedral set
cut out by finitely many hyperplanes. For this reason, it is also
called the \emph{polyhedral derived \gar{} component} or \emph{polyhedral
derived component}.

\subsection{NRT \gar{} polynomials and polyhedral components}

Next, we extend the NRT condition to the general setting.
\begin{defn}
\label{def:NRT-general} A G\aa rding polynomial $g\in\GS$ is said
to be \emph{NRT} if its G\aa rding component $\C_{g}$ satisfies
the negative ray terminating  condition (NRT).
\end{defn}

\begin{lem}
\label{lem:polyhedral-dual}Let $\Lambda$ be a finite index set.
Let 
\[
\mathcal{C}=\{\x\in\mathbb{R}^{n}:\ \a_{\alpha}\cdot\x>b_{\alpha},\ \alpha\in\Lambda\},
\]
with $\a_{\alpha}\in\overline{\Gamma_{n}^{+}}\setminus\{0\}$, and
let $\overline{\mathcal{C}}$ be its closure. For $\v\in\Gamma_{n}^{+}$,
the infimum $\inf_{\x\in\overline{\mathcal{C}}}\v\cdot\x$ is finite
if and only if there exist coefficients $\lambda_{\alpha}\ge0$ such
that $\v=\sum_{\alpha\in\Lambda}\lambda_{\alpha}\a_{\alpha}$. In
particular, $\C$ is NRT if and only if for any $i\in[n],$ there
exists $\alpha\in\Lambda$ such that $\a_{\alpha}\cdot\e_{i}>0$.
\end{lem}

\begin{proof}
The first part is a standard consequence of linear programming duality
(equivalently, Farkas' lemma \cite[pp. 200]{RockafellarConvexAnalysis})
applied to the constraints $\a_{\alpha}\cdot\x\geq b_{\alpha}$. The
second part follows from the definition of NRT region.
\end{proof}
\begin{prop}
\label{prop:NRTrealization}Let $g\in\mathrm{G}$ be a G\aa rding
polynomial of degree $d$. If $g$ is NRT, then
\begin{enumerate}
\item \label{enu:the-top-homogeneous}the top homogeneous part $g^{\mathrm{top}}$
is a G\aa rding polynomial and is NRT;
\item \label{enu:for-a--symmetric}for a $\kappa$-symmetric realization
$h=\pol_{\kappa}(g)$, the polyhedral component $\C_{h}^{(d-1)}$
is NRT.
\end{enumerate}
\end{prop}

\begin{proof}
First, assume that $g$ is multi-affine. Then, by Proposition \ref{prop:truncateLet--be},
$g^{\mathrm{top}}\in\G$. Suppose, for the sake of contradiction, that $g^{\mathrm{top}}$
is not NRT. Then there exists $i\in[n]$ such that $\partial_{i}g^{\mathrm{top}}\equiv0$.
Writing $g(\x)=x_{i}p(\x)+q(\x)$, we have $\deg q=d$ and $\deg p\le d-2$,
since the degree-$d$ part of $g$ is independent of $x_{i}$. This
contradicts \ref{lem:proper-position-basic} \ref{enu:444}, which
excludes such a decomposition. This completes the proof of (\ref{enu:the-top-homogeneous})
for $g\in\G$. 

To finish the proof for the general \gar{} polynomial, we apply Theorem
\ref{thm:kappa-symmetric-realization} to obtain a $\kappa$-symmetric
realization $h=\pol_{\kappa}(g)$. Then $h^{\text{top}}$ is NRT
which implies (\ref{enu:the-top-homogeneous}) after taking the $\kappa$-diagonal
specialization. 

To prove (\ref{enu:for-a--symmetric}), by the first part, for each
$i$ there exists $\alpha$ with $|\alpha|=d-1$ such that $\partial^{\alpha}g(\x)=\a_{\alpha}\cdot\x-b_{\alpha}$
and $\a_{\alpha}\cdot\mathbf{e}_{i}>0$. By Lemma~\ref{lem:polyhedral-dual},
the polyhedral component $\mathcal{C}_{h}^{(d-1)}$ is NRT. 
\end{proof}

\subsection{Auxiliary classes and an approximation theorem}

We introduce 2 auxiliary classes of polynomials. They will be used
to perturb general \gar{} polynomials to have additional non-degeneracy
conditions.
\begin{defn}[Auxiliary Classes]
\label{def:class-A} Denote $\mathcal{A}$ to be the class of polynomials
$g(\x)\in\R[\x]$ such that 
\[
\partial_{i}g\in\mathrm{G}\qquad\text{for all }i\in[n].
\]

Let $\mathcal{B}$ be the subclass 
consisting of $f\in \mathcal{A}$ satisfying the additional non-degeneracy
condition 
\[\label{def:class-B}
\overline{\mathcal{C}_{\partial_{i}f}}\;\subset\;\mathcal{C}_{\partial_{i}\pdv_{j}f}\qquad\text{for all }i\neq j.
\]
\end{defn}

Note that for $g\in\mathcal{A}^{d}$ of degree $d,$ derived components
can still be defined for $k\geq1$. Furthermore, we have
\[
\C_{g}^{(1)}\;\subset\;\C_{g}^{(2)}\;\subset\;\cdots\;\subset\;\C_{g}^{(d-1)}\;\subset\;\C_{g}^{(d)}\;=\;\mathbb{R}^{n}.
\]
The special component
$\C_{g}^{(d-1)}$ is also polyhedra. Define the index set 
\[
\mathcal{I}_{g}:=\{\alpha\in\mathbb{Z}_{\ge0}^{n}:|\alpha|=d-1,\ \partial^{\alpha}g\not\equiv0\}.
\]
For each $\alpha\in\mathcal{I}_{g}$, the derivative $\partial^{\alpha}g$
is affine and can be written as 
\[
\partial^{\alpha}g(x)=\a_{\alpha}\cdot\x-b_{\alpha},\qquad\a_{\alpha}\geq0,\ \a_{\alpha}\neq0,b_{\alpha}\in\R.
\]
Therefore,
\[
\mathcal{C}_{g}^{(d-1)}:=\{\,x\in\mathbb{R}^{n}:\a_{\alpha}\cdot\x>b_{\alpha}\ \text{for all }\alpha\in\mathcal{I}_{g}\}.
\]
We call $\C_{g}^{(d-1)}$ the\emph{ polyhedral derived component}
of $g\in\mathcal{A}.$ 
\begin{thm}
\label{thm:approx-by-B} Let $g\in\mathrm{G}$ be an NRT \gar{} polynomial
with $\deg g>1$ and let $\kappa\ge\kappa_{g}$ be a multi-index.
Set $h=\pol_{\kappa}(g)$ and assume $h\in\mathcal{A}$. Then there
exists a sequence of NRT Gårding functions $\{g_{m}\}_{m\ge1}\subset\mathrm{G}$
such that $g_{m}$ converges to $g$ in the compact-open topology,
and writing $h_{m}=\pol_{\kappa}(g_{m})$, one has $h_{m}\in\mathcal{B}$
for every $m$. 
\end{thm}

\begin{proof}
We construct a linear perturbation of $h$ from the symmetric
realization of $g$.

Let $f\in\G_{\kappa_{0}}$ be a symmetric realization of $g$. By
Proposition ~\ref{prop:NRTrealization}, the derived polyhedral component
$\mathcal{C}_{f}^{(d-1)}$ is NRT. Write $\mathcal{C}_{f}^{(d-1)}=\{\mathbf{z}\in\mathbb{R}^{N}:\ A\mathbf{z}>\mathbf{b}\}$,
where the rows of $A$ are the facet normals $\mathbf{a}_{\alpha}\in\overline{\Gamma_{N}^{+}}$.
Define $\ell(\mathbf{z}):=\sum_{\alpha}(\mathbf{a}_{\alpha}\cdot\mathbf{z}-b_{\alpha})$
which is $\kappa_{0}$-symmetric. Lemma~\ref{lem:polyhedral-dual}
implies that $\nabla\ell=\sum_{\alpha}\mathbf{a}_{\alpha}\in\Gamma_{N}^{+}$
and $\ell(\mathbf{z})>0$ in $\mathcal{C}_{f}^{(d-1)}$. 

Set $f_{c}=f-c\,\ell$. Since $\C_f\subset\C_f^{(d-1)}$ and $\ell>0$ on $\C_f^{(d-1)}$, we have $\ell\vtl f$. Since $\deg f=\deg g>1$, $\rho_f (\ell)$ given in Definition~\ref{def:rho} is $+\infty$. Hence, Theorem \ref{thm:dominategardnew}
implies that $f_{c}$ is NRT multi-affine Gårding for all $c>0$.
Since $\nabla\ell>0$, each polarized partial satisfies $\partial_{x_{ij}}f_{c}=\partial_{x_{ij}}f-c\,a_{ij}$
with $a_{ij}>0$, Lemma \ref{lem:strong} and Remark \ref{rem:strong reduction}
imply that $f_{c}\in\mathcal{B}$.

Let $g_{c}=\proj_{\kappa_{0}}f_{c}$ and $\bar{\ell}=\proj_{\kappa_{0}}\ell$.
Then, $g_{c}=g-c\,\bar{\ell}\in\GS$ for all $c>0$. 

Let $h\in\pol_{\kappa}g,$ which belongs to $\mathcal{A}$ by assumption.
Set $L=\pol_{\kappa}(\bar{\ell})$. Note that $\nabla L>0$. Define 
\[
h_{c}:=\pol_{\kappa}(g_{c})=h-c\,L.
\]
Then, for any $i\in[n]$ and $j\in[\kappa_{i}],$$\partial_{x_{ij}}h_{c}=\partial_{x_{ij}}h-ca'_{ij},$
with some constants $a'_{ij}>0.$ Lemma \ref{lem:strong} and Remark
\ref{rem:strong reduction} imply that $h_{c}\in\mathcal{B}$ for all
$c>0$.

Finally, choosing any sequence $c_{m}\downarrow0$ yields a sequence that
satisfies the conclusion of the theorem. 
\end{proof}

\subsection{Closure under $\kappa$-polarization }

For every general
\gar{} polynomial $g$, Theorem \ref{thm:kappa-symmetric-realization} provides a $\kappa$-symmetric realization for some 
$\kappa\geq\kappa_{g}$. In this subsection, we show that we can always
pick $\kappa=\kappa_{g}$. 
\begin{prop}
\label{prop:n=00003D2}Let $f(x_{0},x_{1},x_{2})\in\mathcal{B}$ be
a multi-affine polynomial and symmetric in $(x_{1},x_{2})$. Let $(u_{0},u_{1},u_{2})\in\partial\mathcal{C}_{\partial_{0}f}$.
Then there exists a diagonal $u$ with $\partial_{0}f(u_{0},u,u)=0$
such that $\min\{u_{1},u_{2}\}\le u\le\max\{u_{1},u_{2}\},$ and 
\[
f(u_{0},u_{1},u_{2})\le f(u_{0},u,u).
\]
Moreover, if $u_{1}\neq u_{2}$ then $\min\{u_{1},u_{2}\}<u<\max\{u_{1},u_{2}\}.$
\end{prop}

\begin{proof}
Considering symmetry, write
\[
f(x_{0},x_{1},x_{2})=ax_{0}x_{1}x_{2}+b_{0}x_{1}x_{2}+bx_{0}(x_{1}+x_{2})+c_{0}x_{0}+c(x_{1}+x_{2})+d.
\]
Then 
\[
\partial_{0}f(x_{1},x_{2})=ax_{1}x_{2}+b(x_{1}+x_{2})+c_{0},\qquad f=x_{0}\,\partial_{0}f+h,
\]
where $h(x_{1},x_{2})=b_{0}x_{1}x_{2}+c(x_{1}+x_{2})+d$.

Since $f\in\mathcal{B}$ polynomials $\partial_{0}f$ and $\partial_{1}f$
are G\aa rding with degrees at most 2 and hence real--stable. In
particular, 
\[
b^{2}-ac_{0}\ge0\qquad\text{and}\qquad bb_{0}-ac\ge0.
\]

Define $p(x)=\partial_{0}f(x,x)=ax^{2}+2bx+c_{0}$. Since $p$ is
stable, its G\aa rding component is the interval to the right of
its largest real root. Let $u$ be this boundary point, i.e. 
\[
u=\frac{-b+\sqrt{\,b^{2}-ac_{0}\,}}{a}\quad(a\ne0),\ \text{or }u=-\frac{c_{0}}{2b}\quad(a=0).
\]
Then $\partial_{0}f(u,u)=0$, so $(u_{0},u,u)\in\partial\mathcal{C}_{\partial_{0}f}$.
Because $(u_{0},u_{1},u_{2})\in\partial\mathcal{C}_{\partial_{0}f}$,
we also have $\partial_{0}f(u_{1},u_{2})=0$. Using $f=x_{0}\partial_{0}f+h$,
\[
f(u_{0},u_{1},u_{2})-f(u_{0},u,u)=h(u_{1},u_{2})-h(u,u).
\]

If $a=0$, then $\partial_{0}f=b(x_{1}+x_{2})+c_{0}$ and $\partial_{0}f(u_{1},u_{2})=0$
imply  $u=(u_{1}+u_{2})/2$. Hence $\min\{u_{1},u_{2}\}<u<\max\{u_{1},u_{2}\}$
when $u_{1}\ne u_{2}$. A direct computation gives 
\[
h(u_{1},u_{2})-h(u,u)=-\frac{b_{0}}{4}(u_{1}-u_{2})^{2}\le0,
\]
and since $\partial_{1}f$ is G\aa rding we have $b_{0}\ge0$. Thus
$f(u_{0},u_{1},u_{2})\le f(u_{0},u,u)$.

Now assume $a>0$. Set $A_{1}:=au_{1}+b$, $A_{2}:=au_{2}+b$, and
$A:=au+b$. Then $A_{1}=\partial_{02}f(u_{1},u_{2})$, $A_{2}=\partial_{01}f(u_{1},u_{2})$,
and $A=\partial_{01}f(u,u)$, all strictly positive by boundary non-degeneracy.
From $\partial_{0}f(u_{1},u_{2})=0$ one finds 
\[
A_{1}A_{2}=b^{2}-ac_{0},
\]
while by definition $A=\sqrt{b^{2}-ac_{0}}$, hence $A^{2}=A_{1}A_{2}$.
If $u_{1}\ne u_{2}$, then $A_{1}\ne A_{2}$; assuming that $A_{1}<A_{2}$
gives $A^{2}=A_{1}A_{2}<A_{2}^{2}$, so $A<A_{2}$. Since $x\mapsto ax+b$
is strictly increasing, this implies $u<u_{2}$, and symmetrically
$u>u_{1}$.

Eliminating products using $\partial_{0}f(u_{1},u_{2})=\partial_{0}f(u,u)=0$
yields 
\[
h(u_{1},u_{2})-h(u,u)=-\frac{bb_{0}-ac}{a}\,(u_{1}+u_{2}-2u).
\]
Because $bb_{0}-ac\ge0$ and $u_{1}+u_{2}-2u\ge0$, the right--hand
side is nonpositive. Hence $f(u_{0},u_{1},u_{2})\le f(u_{0},u,u)$. 
\end{proof}
\begin{prop}
\label{prop:symmetry}Let $f\in\mathcal{B}$ be multi-affine and symmetric
in $x_{1},\dots,x_{n}$. Let $(u_{0},\u)=(u_{0},u_{1},\dots,u_{n})\in\partial\mathcal{C}_{\partial_{0}f}$
and set $I=[\min_{i}u_{i},\max_{i}u_{i}]$. Then there exists $u^{*}\in I$
such that 
\[
(u_{0},\u^{*})=(u_{0},u^{*},\dots,u^{*})\in\partial\mathcal{C}_{\partial_{0}f}
\]
and 
\[
f(u_{0},u_{1},\dots,u_{n})\le f(u_{0},u^{*},\dots,u^{*}).
\]
\end{prop}

\begin{proof}
Define $\phi(\x)=\max_{i\in[n]}x_{i}-\min_{i\in[n]}x_{i}$ for $\x\in I^{n}$,
and set 
\[
\mathcal{F}=\bigl\{\x\in I^{n}:\ (u_{0},\x)\in\partial\mathcal{C}_{\partial_{0}f}\ \text{and}\ f(u_{0},\x)\ge f(u_{0},u_{1},\dots,u_{n})\bigr\}.
\]
The set $\mathcal{F}\not=\emptyset$ and is compact, hence, $\phi$ attains
a minimum at some $\mathbf{v}=(v_{1},\dots,v_{n})\in\mathcal{F}$.

By the non-degeneracy condition in the definition of $\mathcal{B}$,
there exists $T\ge0$ such that 
\[
\partial_{i}f(u_{0}+T,\mathbf{v})>0\qquad\text{for all }i=1,\dots,n.
\]
Since $\partial_{0}f\equiv0$ on $\mathcal{F}$, $f(u_{0},\v)=f(u_{0}+T,\v)$
and replacing $u_{0}$ by $u_{0}+T$ does not affect $\mathcal{F}$
or $\phi$, and $\mathbf{v}$ remains a minimizer. 

Suppose that not all coordinates of $\v$ coincide. Let $w^{(1)}<\cdots<w^{(m)}$
be the distinct values among the set $\{v_{i}\}_{i\in[n]}$ with $m\ge2$,
and let 
\[
K=\{k:\ v_{k}=w^{(m)}\}.
\]
Fix an index $i$ with $v_{i}=w^{(1)}$. For each $j\in K$, freeze
every coordinate $x_{k}$ with $k\notin\{i,j\}$ to $v_{k}$ and consider
the three--variable restriction 
\[
g_{ij}(x_0,x_{i},x_{j})=f(x_0,\dots,x_{i},\dots,x_{j},\dots).
\]
Because $\partial_{\ell}f(u_{0}+T,\mathbf{v})>0$ for all frozen indices
$\ell$, the polynomial $g_{ij}$ lies in $\mathcal{B}$ by Lemma \ref{lem:For-any-.} and Proposition
\ref{prop:n=00003D2} applies and we may replace $(v_{i},v_{j})$
by $(u_{ij},u_{ij})$ with 
\[
u_{ij}\in[\min\{v_{i},v_{j}\},\max\{v_{i},v_{j}\}),
\]
hence $u_{ij}\neq w^{(m)}$. Thus, each such replacement strictly decreases
$|K|$ which is the multiplicity of the maximal value $w^{(m)}$.
After at most $|K|$ steps, the value $w^{(m)}$ disappears from the
coordinate list, so the maximum coordinate strictly decreases and
$\phi$ strictly decreases. This contradicts the minimality of $\phi(\mathbf{v})$.

Therefore, $\v=(u^{*},\dots,u^{*})$ for some $u^{*}\in I$. Since
$\v\in\mathcal{F}$, we have $(u_{0}+T,\u^{*})\in\partial\mathcal{C}_{\partial_{0}f}$.
Hence $(u_{0},\u^{*})\in\pdv\C_{\pdv_{0}f}$ and $f(u_{0},\u^{*})=f(u_{0}+T,\u^{*})$.
By the definition of $\mathcal{F}$, 
\[
f(u_{0},u_{1},\dots,u_{n})\le f(u_{0}+T,u^{*},\dots,u^{*})=f(u_{0},\u^{*}).
\]
We have finished the proof.
\end{proof}
\begin{prop}
\label{prop:POLofNRTB}Let $g\in\GS$ be NRT and $\deg g>1.$ Let
$f=\pol_{\kappa}(g)$. If $f\in\mathcal{B}$, then $f\in\GS$. 
\end{prop}

\begin{proof}
Let $(x_{1},\cdots,x_{n})$ be variables for $g$ and let $(x_{ij})_{i\in[n],j\in[\kappa_{i}]}$
be variables for $f$. By Proposition \ref{prop:testing method},
if $f\in\mathcal{B}$, it suffices to verify that $f\le0$ on each
boundary piece $\partial\mathcal{C}_{\partial_{ij}f}$. Without loss
of generality, we prove only the $i=1,j=1$ case. Fix an arbitrary
point $p=(u_{11},\x')\in\partial\mathcal{C}_{\partial_{11}f}$. Apply
Proposition \ref{prop:symmetry}, to produce a point $\mathbf{q}=(u_{11},\mathbf{q}')\in\partial\mathcal{C}_{\pdv_{11}f}$
such that 
\[
f(\mathbf{p})\le f(\mathbf{q}).
\]

By the coordinate replacements allowed in the lemma, we may first
make every block $i\neq1$ constant, then make all coordinates in
the first block equal except possibly the distinguished coordinate
$x_{11}$. Finally, since $\pdv_{11}f(\mathbf{q})=0$ move freely, $x_{11}$
will not affect $\pdv_{11}f(\mathbf{q})$ or $f(\mathbf{q})$.
Hence, we may assume that $\mathbf{q}$ is $\kappa$-symmetric,
i.e.\ $q_{i,1}=\cdots=q_{i,\kappa_{i}}=:q_{i}^{*}$ for each $i$.

Since $\mathbf{q}$ is $\kappa$--symmetric, by the definition of polarization
we have $f(\mathbf{q})=g(\mathbf{q}^{*})$, where $q^{*}=(q_{1}^{*},\dots,q_{n}^{*})$.
Moreover $\mathbf{q}\in\partial\mathcal{C}_{\partial_{11}f}$ implies
$\mathbf{q}^{*}\in\partial\mathcal{C}_{\partial_{1}g}$ . Because
$g$ is \gar{}, $g(\mathbf{q}^{*})\le0$, and hence 
\[
f(\mathbf{q})=g(\mathbf{q}^{*})\le0.
\]
Combining with $f(\mathbf{p})\le f(\mathbf{q})$ yields $f(\mathbf{p})\le0$.

Since $\mathbf{p}\in\partial\mathcal{C}_{\partial_{11}f}$ was arbitrary,
this shows that $f\le0$ on each $\partial\mathcal{C}_{\partial_{ij}f}$,
and therefore on $\partial\mathcal{C}_{f}^{(1)}$. By Proposition
\ref{prop:testing method}, it follows that $f$ is G{\aa}rding.
\end{proof}

\begin{thm}

\label{thm:polarization_keeps_garding}

If $g(\x)\in\GS[\x]\cap \R_\kappa[\x]$ is NRT, then
$\pol_{\kappa}(g)\in\mathcal{\GS}$. 
\end{thm}
\begin{proof}

Let $f=\pol_{\kappa}(g)$ and set $d=\deg g$. We argue by induction.

\emph{Base case $d\le2$.} Proposition 3.4 in \cite{BB09I} uses Grace-Walsh-Szeg\"o
Theorem \cite{grace_zeros_1902,szego_bemerkungen_1922} to show the
$\kappa$-polarization preserves real stability in degree $\le2$.
Thus, $f\in\mathcal{\GS}$.

\emph{Induction step.} Assume the statement holds for all polynomials
in $\GS$ of degree less than $d$, and let $\deg g=d>2$. Since the
polarization commutes with translation in all one direction, without loss of generality,
assume that 
$$
 g\in\GS_{+},\qquad g(\mathbf{0})>0.
$$
For each $i\in[n]$ and $j\in[\kappa_{i}]$, the polarization identity
\begin{equation}
\pdv_{x_{ij}}\pol_{\kappa}(g)=\frac{1}{\kappa_{i}}\Pi_{\kappa-\mathbf{e}_{i}}^{\uparrow}(\pdv_{x_{i}}g)(\cdots,x_{i1},\cdots,\hat{x}_{ij},\cdots,x_{i\kappa_{i}},\cdots).\label{eq:polarizationIdentity}
\end{equation}
gives 
\[
\partial_{x_{ij}}f=\frac{1}{\kappa_{i}}\,\pol_{\kappa-e_{i}}\!\bigl(\partial_{x_{i}}g\bigr).
\]
Since $\partial_{x_{i}}g\in\mathcal{\GS}$ has degree at most $d-1$,
the induction hypothesis implies 
\[
\pol_{\kappa-e_{i}}\!\bigl(\partial_{x_{i}}g\bigr)\in\mathcal{\GS}.
\]
Therefore, each first partial $\partial_{x_{ij}}f\in\GS$, and $f$
is in the auxiliary class $\mathcal{A}$. 

We apply Proposition~\ref{thm:approx-by-B} to obtain  a sequence of polynomials  $g^{(m)}$ that converge to $g$ in the compact-open
topology such that $f^{(m)}=\pol_{\kappa}(g^{(m)})\in \mathcal{B}$. By the previous
paragraph, each $f^{(m)}$ satisfies the hypotheses of Proposition~\ref{prop:POLofNRTB},
hence each $f^{(m)}$ is G{\aa}rding. Since $f^{(m)}\to f$ in compact-open topology and $g(\mathbf{0})>0$,
 for large enough $m,$ $f^{(m)}\in\G_{+}.$ We apply Lemma \ref{lem:G_+isclosed}
to conclude that $f=\pol_{\kappa}(g)$ is G{\aa}rding.

This completes the induction and the proof. 
\end{proof}

\section{A recursive characterization of \gar{} polynomials}\label{sec:A-recursive-characterization}
In this section, we extend the methods of the previous section to obtain a recursive characterization of \gar{} polynomials and complete the proof of Theorem~\ref{thm:newmain}. This alternative characterization also allows us to classify univariate \gar{} polynomials, which coincide with the right Noetherian polynomials of Lin~\cite{Lin23JFA}.

We first define a broader class of polynomials with weaker requirements. Let $f(\mathbf{x})$ be a polynomial in $n$ variables. If the set $\{f>0\}$ has a connected component that contains a positive affine hyper-octant, we denote this component by $\widehat{\C_{f}}$; if $f(\mathbf{x})\equiv\text{const}\geq0$, then we set $\widehat{\C_{f}}=\mathbb{R}^{n}$; otherwise, we set $\widehat{\C_{f}}=\emptyset$.   $\widehat{\C}_{f}$ is unique by Lemma~\ref{lem:basic}.
\begin{defn}
\label{def:gard_recursive}
Fix $n\ge 1$. For degrees $d\le 2$, let $\mathcal{G}_{n}^{d}=\GS_{n}^{d}$. For degrees $d\ge 3$, let $\mathcal{G}_{n}^{d}$ be the class of $n$-variate polynomials $g(\x)$ such that
\begin{enumerate}
\item $\widehat{\C}_{g}\neq\emptyset$,
\item $\partial_{i}g\in\mathcal{G}_{n}^{\,d-1}\ \text{for all } i\in[n],$
\item $\widehat{\C}_{g}\subseteq\bigcap_{i=1}^{n}\widehat{\C}_{\partial_{x_{i}}g}$.
\end{enumerate}
Set
\[
\mathcal{G}:=\bigcup_{n\ge 1}\ \bigcup_{d\ge 0}\ \mathcal{G}_{n}^{d}.
\]
We denote
\[
\widehat{\C}_{f}^{(k)}=\bigcap_{|\alpha|=k}\widehat{\C}_{\pdv^{\alpha}f}.
\]
Then $f\in\mathcal{G}_{n}$ if and only if $\pdv_{i}f\in\mathcal{G}$ for all $i\in[n]$ and $\emptyset\neq\widehat{\C}_{f}\subseteq\widehat{\C}_{f}^{(1)}$.

Finally, $f$ is called an \emph{NRT-$\mathcal{G}$ polynomial} if $f\in\mathcal{G}$ and $\widehat{\C}_{f}$ is NRT.
\end{defn}

\begin{rem}
The existence of positive components is the central feature common in the definitions of
both $\GS$ and $\mathcal{G}.$ In the multi-affine case, $\G=\mathcal{G}^{\mathfrak{a}}$
according to Lemma \ref{lem:Suppose-that--2}. More generally, $\GS$
is defined through strictly positive invariance, while $\mathcal{G}$
is defined through the compatibility of Gårding components through a recursive
procedure. By Theorem \ref{thm:diff-closure}, $\GS\subset\mathcal{G}$.
We will show later that  $\GS=\mathcal{G}$.
\end{rem}

\subsection{From $\mathcal{G}$ to \gar{} polynomials}

We aim to prove $\mathcal{G} =\GS$. We start with some basic
properties whose proofs are parallel to those in the multi-affine
case and are omitted.
\begin{lem}
\label{lem:trivialGarding}
Suppose that $f\in\mathcal{G}_{n}$. Then
\begin{enumerate}
\item $\widehat{\C}_{f}$ satisfies the PRT;
\item $\widehat{\C}_{f}=\R^{n}$ if and only if $f\equiv c$;
\item for any $c\geq 0$, $f-c\in\mathcal{G}$.
\end{enumerate}
\end{lem}

Similarly to Lemma~\ref{lem:splitting}, the restriction of $f\in\mathcal{G}$ to an affine coordinate subspace passing through a point in $\widehat{\C}_{f}$ is also in $\mathcal{G}$.

\begin{lem}
\label{lem:RestrforgeneralGrad}
If $f\in\mathcal{G}$ and $\a=(a_{1},\cdots,a_{n})\in\widehat{\C}_{f}$, then $f(a_{1},x_{2},\cdots,x_{n})\in\mathcal{G}[x_{2},\cdots,x_{n}]$.
\end{lem}

The next lemma characterizes the  NRT-$\mathcal{G}$ polynomials as in Lemma~\ref{lem:Let--be}. The proof is similar.

\begin{lem}
\label{lem:NRTforgeneralgard}
Suppose that $f\in\mathcal{G}$. Then the following are equivalent:
\begin{enumerate}
\item \label{enu:-is-anNRTto1}$f$ is an NRT-$\mathcal{G}$ polynomial;
\item \label{enu:-is-anNRTto2}$f$ does not have redundant variables, i.e., $\deg_{x_{i}}f\geq 1$ for all $i\in[n]$;
\item \label{enu:-is-anNRTto3}for any $i\in[n]$, $\pdv_{i}f|_{\widehat{\C}_{f}}>0$.
\end{enumerate}
\end{lem}

\begin{cor}
\label{cor:NRTforeachvariable}
Suppose that $f\in\mathcal{G}$ and $\deg_{x_{1}}f\geq 1$. Then for any $\x_{0}\in\widehat{\C}_{f}$, there exists $T=T(\x_{0})>0$ such that $\x_{0}-T\mathbf{e}_{1}\in\partial\widehat{\C}_{f}$.
\end{cor}

\begin{proof}
For any $f\in\mathcal{G}$ and $\x_{0}\in\widehat{\C}_{f}$, by Lemma~\ref{lem:RestrforgeneralGrad}, $g(t):=f(\x_{0}+t\mathbf{e}_{1})\in\mathcal{G}[t]$. If $g$ is nonconstant, then the conclusion follows. Otherwise, if $\deg_{x_{1}}f\geq 1$, then Lemma~\ref{lem:NRTforgeneralgard} implies that $\pdv_{1}f(\x_{0})>0$, and hence $g$ is an NRT-$\mathcal{G}$ polynomial. This completes the proof.
\end{proof}

If $f$ is an NRT-$\mathcal{G}$ polynomial with $\deg f\geq 2$, we show that $\widehat{\C}_{f}^{(1)}$ is an NRT component.
\begin{lem}
\label{lem:linear-in-y} Let $g(\x)\in \R[\x]$ 
and let $c>0$. If the polynomial  $G(\x,y):=g(\x)+c\,y$
belongs to $\mathcal{G}$, then $g(\x)$ is affine linear; equivalently,
\[
\partial_{x_{i}}\partial_{x_{j}}g\equiv0\qquad\text{for all }i,j.
\]
\end{lem}

\begin{proof}
Since $G(\x,y)=g(\x)+c\,y$ with $c>0$, then 
\begin{equation}
\widehat{\mathcal{C}_{G}}=\{(\x,y)\in\mathbb{R}^{n+1}:\ y>-g(\x)/c\}.\label{eq:ppp}
\end{equation}
Suppose, toward a contradiction, that $\deg g\ge2$. Then there exists
an index $i$ such that $\partial_{x_{i}}g\not\equiv0$, and hence
$\partial_{x_{i}}G=\partial_{x_{i}}g$ is a non-constant polynomial
depending only on $\x$. Choose a point $\x^{\ast}\in\partial\widehat{\mathcal{C}_{\partial_{x_{i}}G}}$.
Because $\partial_{x_{i}}G$ is independent of $y$, the boundary
$\partial\mathcal{C}_{\partial_{x_{i}}G}$ is invariant under translation
in the $y$--direction, and therefore $(\x^{\ast},y)\in\partial\mathcal{C}_{\partial_{x_{i}}G}$
for all $y\in\mathbb{R}$. On the other hand, for all sufficiently
large $y$, by \eqref{eq:ppp}, $(\x^{\ast},y)\in\widehat{\C_{G}}$.
This contradicts the condition $\widehat{\mathcal{C}_{G}}\subseteq\widehat{\mathcal{C}_{\partial_{x_{i}}G}}$
required for membership in $\mathcal{G}$. Hence $\deg g\le1$, and
$g$ is affine linear. 
\end{proof}
\begin{lem}
\label{lem:top-nrt-affine} Let $g$ be an NRT-$\mathcal{G}$ polynomial
of degree $d$. Then for $k\leq d-1$, 
$\widehat{\C_{g}}^{(k)}$ is NRT.
\end{lem}

\begin{proof}
Denote $g^{\mathrm{top}}$
for the homogeneous top--degree part of $g$. We claim
that for any  $i\in[n]$, 
\begin{equation}
\partial_{x_{i}}g^{\mathrm{top}}\not\equiv0.\label{eq:aaa}
\end{equation}

Assuming (\ref{eq:aaa}), for each $i\in[n]$, there exists a multi-index
$\alpha$ with $|\alpha|=d-1$ such that $\nabla\partial^{\alpha}g=\mathbf{a}_{\alpha}$
has a strictly positive $\mathrm{\mathbf{e}}_{i}$--direction. By
Lemma \ref{lem:polyhedral-dual}, this implies that $\widehat{\mathcal{C}_{g}}^{(d-1)}$
is an NRT polyhedral component. By Lemma \ref{lem:basic} (\ref{enu:NRTlabel}),
for $k<d-1$, $\widehat{\C_{g}}^{(k)}\subset\widehat{\C_{g}}^{(d-1)}$ is also NRT.

It remains to prove (\ref{eq:aaa}). Suppose, in contradiction,
that for some $i$ the polynomial $g^{\mathrm{top}}$ contains no
$x_{i}$. Then every monomial of $g$ that involves $x_{i}$ has degree
$\leq d-1$. Choose a multi-index $\beta$ with $|\beta|\leq d-1$
such that 
\[
\partial^{\beta}g=p(x)+c\,x_{i},
\]
where $c>0$ is a constant and $\deg p>1$. Since $g\in\mathcal{G}$
and $\mathcal{G}$ are closed under partial differentiation, $\partial^{\beta}g\in\mathcal{G}$
. This contradicts Lemma~\ref{lem:linear-in-y}. Hence, (\ref{eq:aaa})
holds.  
\end{proof}
Similarly to Theorem \ref{thm:approx-by-B}, we state an approximation
result for NRT-$\mathcal{G}$ polynomials. 
\begin{prop}
\label{prop:approxiG}
Let $f(\mathbf{x})$ be an NRT-$\mathcal{G}$ polynomial. Let
\[
\ell(\mathbf{x}) = \sum_{i=1}^{n} a_i x_i + b
\]
with $a_i \ge 0$. If $\ell(\mathbf{x})\big|_{\widehat{\mathcal{C}}_{f}^{(1)}} \ge 0$, then there exists $\epsilon > 0$ such that $f(\mathbf{x}) - \epsilon \ell(\mathbf{x})$ is also an NRT-$\mathcal{G}$ polynomial. If $\mathbf{0} \in \widehat{\mathcal{C}}_{f}$, we may also assume that $\mathbf{0} \in \widehat{\mathcal{C}}_{f - \epsilon \ell}$.
\end{prop}

\begin{proof}
By a translation, we may assume $\mathbf{0}\in\widehat{\C}_{f}$. Since $f$
is NRT, $\pdv_{i}f(\mathbf{0})>0$. Choose $\epsilon_{0}>0$ small
enough such that $(f-\epsilon_{0}\ell)|_{\mathbf{0}}>0$ and $\pdv_{i}(f-\epsilon_{0}\ell)|_{\mathbf{0}}>0$
for all $i$. Let $\epsilon\in(0,\epsilon_{0}]$ and $g(\x)=f(\x)-\epsilon \ell(\x)$.
Then, by the choice of $\epsilon$, the Taylor expansion of $g$ at
the origin is non-negative, implying $\mathbf{0}\in\widehat{\C}_{g}\not=\emptyset$.
Moreover, since $\ell|_{\widehat{\C}_{f}}\geq0$, $\widehat{\C}_{g}\subset\widehat{\C}_{f}$,
implying that $\widehat{\C}_{g}$ is NRT. 

We claim  $g\in\mathcal{G}$. Since $\pdv_{i}g=\pdv_{i}f-\epsilon a_{i}\in\mathcal{G}$
by Lemma \ref{lem:trivialGarding}, it suffices to prove $\widehat{\C}_{g}\cap\pdv\widehat{\C}_{g}^{(1)}=\emptyset$.
We argue by contradiction. Assume there exists $\w\in\pdv\widehat{\C}_{g}^{(1)}\cap\widehat{\C}_{g}$.
Without loss of generality, we assume $\w\in\pdv\widehat{\C}_{\pdv_{1}g}$.
Hence, 
\[
\pdv_{1}g(\w)=\pdv_{1}f(\w)-\epsilon a_{1}=0,\ g(\w)>0.
\]
We distinguish 2 cases.

Case 1. $\deg_{x_{1}}g\leq1$. Then $g(t\mathbf{e}_{1}+\w)>0$ for
all $t\in\R,$ contradicting that $\widehat{\C}_{g}$ is NRT. 

Case 2. $\deg_{x_{1}}g\geq2$. Then $\pdv_{1}f$ is a non-constant
polynomial in $\mathcal{G}$. Moreover, Corollary \ref{cor:NRTforeachvariable}
states that for any $\x\in\widehat{\C}_{\pdv_{1}f}$, there exists
$T'(\x)>0$ such that $\x-T'(\x)\mathbf{e}_{1}\in\pdv\widehat{\C}_{\pdv_{1}f}.$
Since $\widehat{\C}_{\pdv_{1}f}$ contains both $\widehat{\C}_{f}^{(1)}$
and $\widehat{\C}_{g}^{(1)}$, there exists $T=T(\w)\geq0$ such that
$\w-T\mathbf{e}_{1}\in\pdv\widehat{{\C}}_{f}^{(1)}$. Then $\ell(\w-T\mathbf{e}_{1})\geq0$,
$f(\w-T\mathbf{e}_{1})\leq0$, and
\begin{equation}
g(\w-T\mathbf{e}_{1})\leq0.\label{eq:glessthan0}
\end{equation}
 For $t\in(-T,0]$, 
\[
0\leq\pdv_{1}f(\w+t\mathbf{e}_{1})\leq\pdv_{1}f(\w)=\epsilon a_{1}.
\]
 Integrating both sides to get
\begin{align}
f(\w) & -f(\w-T\mathbf{e}_{1})=\int_{-T}^{0}\pdv_{1}f(\w+t\mathbf{e}_{1})dt\leq\epsilon(\ell(\w)-\ell(\w-T\mathbf{e}_{1})).\label{eq:case2Appro}
\end{align}
 Combining (\ref{eq:case2Appro}) and (\ref{eq:glessthan0}), 
 we have $g(\w)\leq g(\w-T\mathbf{e}_{1})\leq0,$
contradicting $g(\w)>0$. 

Thus, no such $\w$ exists. Hence, $g$ is NRT-$\mathcal{G}.$ 
\end{proof}
Now we identify $\GS$ and $\mathcal{G}$, completing the proof of
Theorem \ref{thm:newmain}.

\begin{thm}
\label{thm:recursivedef}$\mathcal{G}=\GS$ and $\widehat{\C}_{g}=\C_{g}$
for all $g\in\mathcal{G}$. 
\end{thm}

\begin{proof}
Theorem \ref{thm:diff-closure} implies that $\GS\subset\mathcal{G}$.
It therefore suffices to prove the reverse inclusion.

Fix $g(\x)\in\mathcal{G}\cap\mathbb{R}_{\kappa}[\x]$, we prove $g\in\GS$
by induction on $d=\deg g.$ The base case $d\leq2$ follows directly
by the definition of $\mathcal{G}$. Assume $d>2$ and let $g\in\mathcal{G},$
with $\deg g=d.$ It is sufficient to prove that for any $\kappa\geq\kappa_{g}$,
\[
f(\w)=\pol_{\kappa}(g)(\w)\in\G_{N},\ N=|\kappa|.
\]

After deleting redundant variables and applying a translation, we
may assume that $g$ is NRT\emph{-}$\mathcal{G}$ with $\mathbf{0}\in\widehat{\C}_{g}$.
By induction, $\pdv_{i}g\in\GS$. Then Lemma \ref{lem:top-nrt-affine}
shows that $\C_{g}^{(k)}=\widehat{\C}_{g}^{(k)}$ is NRT for $1\leq k\leq d-1$.
Using the polarization identity (\ref{eq:polarizationIdentity}),
and Lemma \ref{lem:polyhedral-dual}, we conclude that $\C_{f}^{(d-1)}$
is also NRT. Arguing as in the proof of Theorem \ref{thm:approx-by-B},
we obtain a $\kappa$-symmetric affine function $L$ such that  $\nabla L\in\Gamma_{N}^{+}$,
$L|_{\C_{f}^{(1)}}>0$. Define $l=\proj_{\kappa}(L)$. Then $l|_{\widehat{\C}_{g}^{(1)}}>0$
and $\nabla l\in\Gamma_{n}^{+}$. 

Similarly to the proof of Theorem \ref{thm:polarization_keeps_garding},
by induction, $f$ lies in the auxiliary class $\mathcal{A}$. Furthermore, if
$f$ is  in $\mathcal{B}$, we conclude the proof using a similar
argument in the proof of Theorem \ref{thm:polarization_keeps_garding}.
Otherwise, we apply Proposition \ref{prop:approxiG} to $g$ and $l$
to obtain a sequence of NRT-$\mathcal{G}$ polynomials with non-negative
coefficients $g_{k}\to g$ in the compact-open topology such that $f_{k}=\pol_{\kappa}g_{k}\in\mathcal{B}$
for all $k.$ Then, $f_{k}\in\G$ and Lemma \ref{lem:G_+isclosed}
shows $f$ is also \gar{}. This completes the induction and proves
$\mathcal{G}\subset\GS$, hence $\mathcal{G}=\GS$. 

\end{proof}
%\begin{rem}
%As mentioned in the introduction, different characterizations in
%our Main Theorem \ref{thm:newmain} emphasize different aspects
%of \gar{} polynomials. While Definition \ref{def:general garding}
%is formulated in purely geometric  terms with strictly positive affine invariance,
%Definition \ref{def:gard_recursive} encodes the same notion
%through a recursive algebraic compatibility of positive regions under
%partial differentiation. 
%\end{rem}

\subsection{Right-Noetherian polynomial}

In \cite{Lin23JFA,LIN2024}, motivated by problems in geometric PDEs,
Lin introduced \emph{right-Noetherian polynomials}. In this subsection,
we show that Lin's constructions are special cases of \gar{} polynomials.
Recall that for a univariate polynomial $f(x),$ $r(f)$ denotes
the \emph{largest real root of $f$}, when it exists.
\begin{defn}\label{MRS}
Let $f$  be a univariate polynomial of degree $d$. Define its \emph{root
sequence} by 
\[
R(f):=(r(f),r(f'),\cdots,r(f^{(d-1)}))\in\R^{d}.
\]
We say that $f$ has a \emph{monotone root sequence} \emph{(MRS)}
if $r(f^{(i)})$ exists for all $0\leq i\leq d-1$ and $$r(f)\geq r(f')\geq\cdots\geq r(f^{(d-1)}).$$
\end{defn}

\medskip{}

\begin{defn}
(Lin \cite{Lin23JFA}) A univariate polynomial $f$ is called \emph{right-Noetherian
if it has a monotone root sequence.}
\end{defn}

We give a geometric interpretation. For univariate  $f(x)\in \GS[x]$,
the Gårding component is
\begin{equation}
\C_{f}=(r(f),+\infty)\label{eq:cone for RN}
\end{equation}
when $r(f)$ exists. A right-Noetherian polynomial $f$ can  be
characterized by
\[
\emptyset\neq\mathcal{C}_{f}\subset\mathcal{C}_{f'}\subset\cdots\subset\mathcal{C}_{f^{(d-1)}}.
\]
A direct consequence of Definition \ref{def:gard_recursive} and Theorem
\ref{thm:recursivedef} is that the class of right-noetherian polynomials
is exactly that of univariate Gårding polynomials.
\begin{example}
\label{exa:univariate-rightnoetherian} Given a sequence of monotone
real numbers $r_{0}\geq r_{1}\geq\cdots\geq r_{d-1}$, a right-noetherian
polynomial may be constructed inductively as follows: Set $f_{d}(x)=1$,
and for each $0\leq k\leq d-1$, define
\[
f_{k}(x):=\int_{r_{k}}^{x}f_{k+1}(t)dt.
\]
Then $f(x):=f_{0}(x)$ is a univariate Gårding polynomial with the
prescribed root sequence. Furthermore, $f$ is unique up to a positive
constant multiple.
\end{example}

\begin{rem}
An alternative way to show that right-Noetherian polynomials coincide with
$\GS_{1}$ is to apply Proposition \ref{prop:symmetry} to prove that
a $\kappa$-polarization of a right noetherian polynomial is multi-affine
Gårding. For $\kappa=d$, this result is essentially due to Lin \cite{Lin23JFA}
where such polarizations are called $\Upsilon$-stable. For $\kappa>d$,
especially in the degenerate case where successive derivatives have common largest real roots,
Lin's argument requires additional
justification. The approach
in Theorem \ref{thm:recursivedef} explicitly addresses the degenerate case.

\end{rem}
In one variable, real stability is equivalent to real-rootedness. By the classical interlacing theorem, real-rootedness implies MRS, and hence
\[
\{\text{univariate real-stable polynomials}\}\subseteq\{\text{univariate G{\aa}rding polynomials}\}.
\]
MRS is strictly broader than real stability in the univariate setting.

\begin{rem}
\label{rem:MRSnecessary}Let $g$ be a multivariate G{\aa}rding
polynomial and let $\mathbf{a}\in\Gamma_{n}^{+},\mathbf{b}\in\R^{n}$.
Then the univariate specialization $f(t)=g(\mathbf{a}t+\mathbf{b})$
is a univariate G{\aa}rding polynomial. Equivalently, $f$ satisfies
the MRS property. Thus, MRS 
for all affine univariate specializations provides an analytic necessary
condition for the G{\aa}rding property. 
\end{rem}

\section{Structural Consequences}\label{sec:structuralconsequence}

In the preceding sections, we have shown that the G{\aa}rding property is preserved under strictly positive affine changes of variables, partial differentiation, and $\kappa$-polarization. Moreover, using polarization together with diagonal restriction, most results proved in the multi-affine setting extend directly to general G{\aa}rding polynomials.

We record here three simple but useful consequences of these closure properties.

\begin{prop}
Let $g(\x)\in\GS[\x]$. If   $\mathbf{a}\in\R^n$, and $\mathbf{a}\geq0$, then the directional derivative $$\partial_{\a}g:=\sum_{i=1}^{n}a_{i}\,\partial_{i}g\in\GS[\x].$$
\end{prop}
\begin{proof}
Fix $\x_{0}\in\C_{g}$ and define $G(\x,t):=g(\x+\x_{0}+t\,\a)$. Since $\x_{0}+t\,\a\in\C_{g}$ for all $t>0$, the polynomial $G$ is G{\aa}rding in $(\x,t)$. By closure under partial differentiation, $\partial_{t}G$ is G{\aa}rding. Restricting to $t=0$ yields $\partial_{t}G(\x,0)=\partial_{\a}g(\x+\x_{0})$, which is therefore G{\aa}rding. By translation invariance, we have proved the proposition.
\end{proof}
\begin{prop}
\label{prop:product}
If $f,g\in\GS[\x]$, then $f\cdot g\in\GS[\x]$.
\end{prop}

\begin{proof}
By polarization, we may assume that $f$ and $g$ are multi-affine. Consider the product in independent variables $F(\x,\y)=f(\x)g(\y)$. Since the multi-affine G{\aa}rding class is closed under products in disjoint variable sets, $F$ is G{\aa}rding. Restricting to the diagonal $\y=\x$ preserves the G{\aa}rding property and yields $F(\x,\x)=f(\x)g(\x)\in\GS[\x]$.
\end{proof}

\begin{thm}
\label{thm:linear-preserver}
Let $T:\mathbb{R}_{\kappa}[x_{1},\dots,x_{n}]\longrightarrow\mathbb{R}_{\gamma}[y_{1},\dots,y_{m}]
$
be a linear operator. The symbol of $T$ is defined by
\[
\mathrm{sym}_{T}(\y,\u):=T\bigl((\x+\u)^{\kappa}\bigr).
\]
If $\mathrm{sym}_{T}(\y,\u)\in\GS_{+}\cap\mathbb{R}_{\gamma\oplus\kappa}[\y,\u]$, then for every $f\in\GS_{+}\cap\mathbb{R}_{\kappa}[\x]$ one has
\[
T(f)\in\GS_{+}\cap\mathbb{R}_{\gamma}[\y].
\]
\end{thm}
\begin{proof}
The statement is similar to the linear preserver theorem for Lorentzian polynomials in Brändén--Huh \cite{BrandenHuh20}; we include a short argument for completeness.

Let $\widetilde{f}:=\Pi_{\kappa}^{\uparrow}(f)$ be the $\kappa$-polarization of $f$. Polarizing the operator $T$ yields a linear operator $\widetilde{T}$ acting on the multi-affine space, whose symbol $\mathrm{sym}_{\widetilde{T}}$ is the polarization of $\mathrm{sym}_{T}$ and therefore still lies in $\GS_{+}$. By Theorem \ref{thm:GardSymbolTransf},   $\widetilde{T}(\widetilde{f})\in\GS^{\mathfrak{a}}_{+}$. Finally, applying the $\kappa$-projection operator gives
\[
T(f)=\Pi_{\kappa}^{\downarrow}\bigl(\widetilde{T}(\widetilde{f})\bigr).
\]
The diagonal restriction preserves $\GS_{+}$; hence $T(f)\in\GS_{+}\cap\R_{\gamma}[\y]$.
\end{proof}

We present a few examples of linear preservers.
\begin{example}[Normalization]
\label{exa:normalization}
Consider the univariate normalization operator $T_{N}: x^{k}\mapsto \frac{x^{k}}{k!}$ defined in $\R_{d}[x]$. Then
$\Sym_{T_{N}}(x,1)=L_{d}(-x),$ where $L_{d}(x)$ is the degree $d$ Laguerre polynomial and is real-rooted. By Lemma~\ref{exa:univariate-rightnoetherian}, after homogenization, $\Sym_{T_{N}}(x,u)\in\S_{+}[x,u]$. By Theorem~\ref{thm:linear-preserver}, $T_{N}$ preserves $\GS_{+}[x]$. By induction on the number of variables, the multivariate normalization operator
\[
T_N:\sum_{\alpha\in\N^{n}}a_{\alpha}\x^{\alpha}\mapsto\sum_{\alpha\in\N^{n}}a_{\alpha}\prod_{i=1}^{n}\frac{x_{i}^{\alpha_{i}}}{\alpha_{i}!}
\]
also preserves $\GS_{+}$.
\end{example}

\begin{example}[Symmetric multiplier]
\label{ex:sym-multiplier}
For $(p_{1},\dots,p_{d})\geq0$, we can define a linear operator
$$T:\R_d[x]\to \R_d[x],\; x^k\mapsto \sigma_k(\mathbf{p})x^k.$$
Then $T$ preserves $\GS^d[x]$. To see that,  consider $T_{1}:x^{k}\mapsto \frac{d!\,\sigma_{k}(\mathbf{p})}{\binom{d}{k}}x^{k}$. Then the symbol of $T_1$ is $\prod_{i=1}^{d}(p_{i}x+u)$ and is real stable. Since the symbol of $T$ is the normalization of that of $T_1$, by Example~\ref{exa:normalization}, $\Sym_{T}(x,u)\in\S_{+}$.  By Theorem~\ref{thm:linear-preserver}, $T$ preserves $\GS^d[x]$.
\end{example}

\subsection{The inversion operation}
Recall that the inversion operator $T_{\tau}$ on $\mathbb{R}_{\kappa}[\x]$ is defined by
\[
T_{\tau}(f)(x_{1},\dots,x_{n}):=f(-1/x_{1},\dots,-1/x_{n})\,x_{1}^{\kappa_{1}}\cdots x_{n}^{\kappa_{n}},
\]
so that $T_{\tau}$ sends a monomial $\x^{\alpha}$ to $(-1)^{|\alpha|}\x^{\kappa-\alpha}$ and satisfies $T_{\tau}^{2}=\mathrm{id}$ on $\mathbb{R}_{\kappa}[\x]$.

\begin{thm}
\label{thm:inversion-1}
Let $f\in\GS_{+}(\x)\cap\mathbb{R}_{\kappa}[x]$ and assume $f(\mathbf{0})>0$. Then
\[
T_{\tau}(f)\in\GS(\x).
\]
\end{thm}
\begin{proof}
By polarization, it suffices to treat the multi-affine case. Thus, assume
\[
f(\x)=\sum_{S\subset[n]}a_{S}\x^{S}\in\G[\x]
\]
with $a_{\varnothing}=f(\mathbf{0})>0$. Since $\partial_{x_{i}}f\prec f\prec y_{i}f$, we have $\pdv_{x_{i}}f\vtl y_{i}f$. Hence
\[
(y_{i}-\partial_{x_{i}})f\in\G[\x,y_{i}]\qquad(1\le i\le n),
\]
and therefore, by repeated application,
\[
g(\x,\y):=\Bigl(\prod_{i=1}^{n}(y_{i}-\partial_{x_{i}})\Bigr)f(\x)\in\G[\x,\y].
\]

Write $A=\{S\subset[n]:a_{S}\neq 0\}$. Because $a_{\varnothing}>0$, the family $A$ is downward closed: if $S\in A$, then every $T\subset S$ lies in $A$. Hence
\[
g(\x,\y)=\sum_{S\in A}\x^{S}\,g_{S}(\y),
\]
where each coefficient polynomial $g_{S}(\y)$ has a leading coefficient $a_{S}>0$. Thus, for all sufficiently large $\y>0$, we have $g_{S}(\y)\geq 0$ for every $S\in A$, so $(\mathbf{0},\y)$ lies in the distinguished component $\C_{g}$ for large $\y$. Thus, $g(\mathbf{0},\y)\in\G[\y]$.

Finally, observe that $g(\mathbf{0},\y)=T_{\tau}(f)(\y)$. Since diagonal specialization preserves the G{\aa}rding property, we conclude that $T_{\tau}(f)\in\GS(\y)$.
\end{proof}

\begin{example}The condition $f(\mathbf{0})>0$ is necessary in Theorem~\ref{thm:inversion-1}.
\label{ex:inversion-necessity}
Indeed, consider
\[
f(x)=(x+1)^{4}-1-4x.
\]
Then $f\in\GS[x]$, since $R(f)=(0,0,-1,-1)$ is an MRS. A direct computation yields
\[
T_{\tau}(f)(x)=6x^{2}-4x+1,
\]
which is a degree-$2$ polynomial that is not real-stable. By Theorem~\ref{thm:stable-garding}, this implies that $T_{\tau}(f)$ is not G{\aa}rding.
\end{example}

\subsection{Binary relations for general G{\aa}rding polynomials}
We extend to general G{\aa}rding polynomials the notions of domination and proper position.
\begin{defn}
\label{def:domination-general}
Let $g(\x)\in\GS[\x]$. If $f(\x)\in\mathbb{R}[\x]$ satisfies
\[
\C_{\partial^{\alpha}g}\subset\{\partial^{\alpha}f>0\}\quad\text{for every nontrivial partial derivative }\partial^{\alpha}f,
\]
then we say that $g$ \emph{dominates} $f$, and write $f\vtl g$.
\end{defn}

Similarly, we have the following definition.
\begin{defn}
\label{def:proper-position-general}
Suppose $f(\x),g(\x)\in\GS[\x]\setminus\{0\}$. The pair $(f,g)$ is said to be in \emph{proper position} if
\[
f(\x)\,y+g(\x)\in\GS[\x,y].
\]
In this case, we also write $f\prec g$. We adopt the convention that $0\prec g\prec 0$ for any $g\in\GS[\x]$.
\end{defn}

It is straightforward to verify that Definitions~\ref{def:dominatesGA} and~\ref{def:properposition} in the multi-affine case are special cases of Definitions~\ref{def:domination-general} and~\ref{def:proper-position-general}. Proceeding as in the multi-affine setting, we prove the following result.
\begin{thm}
\label{thm:dominate-position-general-thm}
Let $f(\x),g(\x)\in\mathbb{R}_{\kappa}[\x]$, and suppose $g(\x)\in\GS[\x]$. If $f\vtl g$, then
\[
\Pi_{\kappa}^{\uparrow}f\;\vtl\;\Pi_{\kappa}^{\uparrow}g.
\]
\end{thm}

\begin{proof}
Let $N=|\kappa|$. We argue by induction on the total degree. If $\deg f,\deg g\leq 1$, the statement is immediate.

Assume that the result holds for all pairs of degree less than $d$. Let $f,g\in\mathbb{R}_{\kappa}[\x]\setminus\{0\}$ with $\deg g=d$ and $f\vtl g$. As in the proof of Lemma~\ref{lem:basic-domination}, we have:
\begin{enumerate}
\item If $f\vtl g$, then $\deg f\leq \deg g$.
\item If $g$ is independent of $x_{i}$, then $f$ is independent of $x_{i}$.
\end{enumerate}

By Lemma~\ref{lem:NRTforgeneralgard}, we may assume that $g$ is NRT. Set
\[
\tilde{f}=\pol_{\kappa}f(\w),\quad \tilde{g}=\pol_{\kappa}g(\w),\quad h(\w,z)=z\tilde{g}(\w)-\tilde{f}(\w).
\]
By the induction hypothesis and Theorem~\ref{thm:iffdominate},
\[
\partial_{w_{ij}}h\in\G \ \text{for all } i\in[n],\ j\in[\kappa_{i}],\qquad \partial_{z}h=\pol_{\kappa}g\in\G.
\]
Thus $h$ belongs to the auxiliary class $\mathcal{A}$.

It suffices to prove the result for $h\in\mathcal{B}$. Otherwise, since $z\cdot g(\x)\in\GS[\x,z]$ is NRT, we may find a linear function $l(\x,z)$ such that $l|_{\C_{zg}^{(d)}}>0$ and $\nabla l>0$ by Lemma~\ref{lem:top-nrt-affine}. For any positive integer $m$, define
\[
f_{m}=f+\frac{1}{m}l(\x,0),\quad g_{m}=g-\frac{1}{m}\pdv_{z}l,\quad h_{m}=h-\frac{1}{m}\pol_{\kappa}l.
\]
Then $f_{m}\vtl g_{m}$ and $h_{m}\in\mathcal{B}$. Once we show $h_{m}\in\G$ for all sufficiently large $m$, it follows that $h=\lim_{m\to\infty}h_{m}$ is G{\aa}rding.

We claim that if $h\in\mathcal{B}$, then
\begin{equation}
\tilde{f}\big|_{\partial\C_{\tilde{g}}}\geq 0.
\label{eq:claim}
\end{equation}
Indeed, suppose that  there exists $\w\in\partial\C_{\tilde{g}}=\partial\C_{\partial_{z}h}$ such that $\tilde{f}(\w)<0$. Applying Proposition~\ref{prop:symmetry} to $h$, we obtain $(\u^{*},z^{*})\in\partial\C_{\partial_{z}h}$ with $h(\u^{*},z^{*})>0$ such that for each $i\in[n]$,
\[
u_{i1}^{*}=\cdots=u_{i\kappa_{i}}^{*}.
\]
Let $\u=(u_{11}^{*},\dots,u_{n1}^{*})$. Then $g(\u)=0$ and $-f(\u)>0$. Since $\u\in\partial\C_{g}$, this contradicts $\C_{g}\subset\{f>0\}$. Hence \eqref{eq:claim} holds.

Let $\C$ be a connected component of $\{\tilde{f}>0\}$ that intersects $\C_{\tilde{g}}$. If $\C_{\tilde{g}}\setminus\C\neq\varnothing$, Lemma~\ref{lem:usefullemma} gives a point $\w_{1}\in\C_{\tilde{g}}$ such that $\tilde{f}(\w_{1})<0$. For any $\w_{2}\leq\w_{1}$ with $\w_{2}\in\partial\C_{\tilde{g}}$, we have $\tilde{f}(\w_{2})\geq 0$ by \eqref{eq:claim}. Since $\C_{\tilde{g}}$ is NRT, it follows that $\nabla\tilde{f}(\w_{1})<0$, and hence for all $\v\geq 0$ with $|\v|\leq 1$,
\begin{equation}
\tilde{f}(\w_{1}+r\,\v)\leq \tilde{f}(\w_{1})<0 \quad \text{for small } r>0.
\label{eq:openset}
\end{equation}
Let $G=\{\w:\w\geq \w_{1},\ \tilde{f}(\w)\leq \tilde{f}(\w_{1})\}$. Then $G$ is closed and, by \eqref{eq:openset}, also open in $\w_{1}+\overline{\Gamma_{N}^{+}}$. Hence $G=\w_{1}+\overline{\Gamma_{N}^{+}}$. Choosing $r$ large, we obtain $r\mathbf{1}_{N}\in G\cap\C_{\tilde{g}}$, a contradiction. Therefore $\tilde{f}>0$ on $\C_{\tilde{g}}$, and the proof is complete.
\end{proof}
The following result is a direct consequence of Theorems \ref{thm:dominate-position-general-thm},  \ref{thm:dominategardnew} and Proposition \ref{proper-position}.
\begin{cor}
\label{cor:domination-consequences} Let $g(\x)\in\mathcal{\GS}[\x]$. Let $f(\x)$ be a non-trivial polynomial such that $f\vtl g$.
\begin{enumerate}
\item  For any $c\ge0$ such that 
$
\{g-cf>0\}\cap\mathcal{C}_{g}\neq\varnothing$, $g(\x)-c\,f(\x)\in\mathcal{\GS}[\x]$.
\item For $\x\in\mathcal{C}_{g}$ and $\mathbf{a}\ge 0$, the directional derivative \[
\partial_\mathbf{a}\left (\frac{g(\x)}{f(\x)}\right)\ge 0.\]
\item One has $f\vtl g$ if and only if $\{f>0\}$ contains a positive affine
hyper-octant, and 
\[
g(\x)\,y-f(\x)\in\GS[\x,y].
\]
\end{enumerate}
\end{cor}

Next, we prove
\begin{thm}
\label{thm:dominate-general}Suppose that $f(\mathbf{x}),g(\x)\in\R_{\kappa}[\x]\cap\GS[\mathbf{x}]$
and both are nontrivial. Then $f\prec g$ if and only if for each
$\alpha\in\N^{n}$, $\C_{\pdv^{\alpha}g}\subset {\C_{\pdv^{\alpha}f}}$ and 
\begin{equation}
\pdv^{\alpha}g(\x)\leq0,\quad \forall\mathbf{x}\in\overline{\C_{\pdv^{\alpha}f}}\backslash\C_{\pdv^{\alpha}g}.\label{eq:position}
\end{equation}
\end{thm}
To prove Theorem~\ref{thm:dominate-general}, we state two lemmas. The first lemma generalizes Lemma~\ref{lem:propotional}. We omit the proof, since it is identical.

\begin{lem}
\label{lem:propotional3}
For $n\geq 1$, if $p_{1},p_{2}\in\GS_{n}$ and, for each $\alpha\in\N^{n}$, $\C_{\pdv^{\alpha}p_{1}}=\C_{\pdv^{\alpha}p_{2}}$, then $p_{1}$ and $p_{2}$ are proportional.
\end{lem}

The second lemma generalizes Lemma~\ref{lem:usefullemmaprec}.

\begin{lem}
\label{lem:precusefulforgeneralstable}
Suppose that $f,g\in\GS_{n}\backslash\{0\}$, $\C_{\pdv^{\alpha}g}\subset {\C_{\pdv^{\alpha}f}}$ and
\[
\pdv^{\alpha}g\big|_{\overline{\C_{\pdv^{\alpha}f}}\setminus\C_{\pdv^{\alpha}g}}\leq 0
\quad \text{for all } \alpha\in\N^{n}.
\]
If $f$ is independent of $x_{i}$, then $\deg_{x_{i}}g\leq 1$ and $\pdv_{i}g$ and $f$ are proportional.
\end{lem}
\begin{proof}
Suppose $f$ is independent of $x_{i}$; then $\pdv_{i}f\equiv 0$. For any $\x\in\C_{g}$, define the univariate polynomial
\[
p_{\x}(t)=g(\x+t\mathbf{e}_{i}).
\]
Since $f\vtl g$, we have $f(\x+t\mathbf{e}_{i})\equiv f(\x)$, and hence $\x+\mathbb{R}\mathbf{e}_{i}\subset\C_{f}$. Thus $p_{\x}(t)$ has exactly one real root unless $p_{\x}(t)\equiv c$. Hence $\deg p_{\x}(t)$ cannot be a positive even number. The same holds for $p'_{\x}(t)=\pdv_{i}g(\x+t\mathbf{e}_{i})$. Therefore $\deg p_{\x}(t)=0$ or $1$.

Since $\C_{g}$ is open in $\mathbb{R}^{n}$, it follows that $\deg_{x_{i}}g\leq 1$. Thus $g$ is affine in $x_{i}$. By the argument in Lemma~\ref{lem:usefullemmaprec}, we have $\pdv_{i}g|_{\C_{f}\setminus\C_{\pdv_{i}g}}\leq 0$ and $\C_{\pdv_{i}g}=\C_{f}$. Similarly, $\C_{\pdv^{\alpha}\pdv_{i}g}=\C_{\pdv^{\alpha}f}$ for all $\alpha\in\N^{n}$. By Lemma~\ref{lem:propotional3}, $\pdv_{i}g$ and $f$ are proportional.
\end{proof}
We are ready to prove Theorem \ref{thm:dominate-general}.
\begin{proof}[Proof of Theorem~\ref{thm:dominate-general}]
($\Rightarrow$) By Theorem~\ref{thm:polarization_keeps_garding}, the relation $f\prec g$ is preserved under polarization. Then the result follows from Lemma~\ref{lem:usefullemmaprec} and Proposition~\ref{proper-position}.

($\Leftarrow$) We argue by induction. Suppose that $f,g\in\GS[\x]$ satisfy $\C_{\pdv^{\alpha}g}\subset {\C_{\pdv^{\alpha}f}}$ and
\[
\pdv^{\alpha}g(\x)\leq 0 \quad \text{for all } \x\in\overline{\C_{\pdv^{\alpha}f}}\setminus\C_{\pdv^{\alpha}g}.
\]

Base case: If $\max\{\deg f,\deg g\}\leq 1$, then by Proposition~\ref{proper-position}, $f\prec g$.

Induction hypothesis: For some $d\geq 2$, if $f,g\in\GS^{d-1}[\x]$ satisfy \eqref{eq:position}, then $f\prec g$.

Now let $f,g\in\GS^{d}[\x]$. We consider two cases.

\emph{Case 1.} $f$ is not NRT. Then $f$ is independent of some variables, denoted as $\{x_{i}\}_{i=1}^{k}$. By Lemma~\ref{lem:precusefulforgeneralstable}, $g$ is affine in $\{x_{i}\}_{i=1}^{k}$ and $\pdv_{i}g=c_{i}f$. Without loss of generality, assume $c_{1}=1$. Then
\[
fy+g=g(x_{1}+y,x_{2},\cdots,x_{n})\in\GS[\x,y],
\]
and hence $f\prec g$.

\emph{Case 2.} $f$ is NRT. Let $N=|\kappa|$, and set $\tilde{f}(\w)=\pol_{\kappa}f(\w)$ and $\tilde{g}(\w)=\pol_{\kappa}g(\w)$. By an approximation argument as in Theorem~\ref{thm:dominate-position-general-thm}, we have $\tilde{g}|_{\partial\C_{\tilde{f}}}\leq 0$. Suppose that $\tilde{g}(\w_{0})>0$ for some $\w_{0}\in\C_{\tilde{f}}\setminus\C_{\tilde{g}}$. Since $\C_{\tilde{f}}$ is NRT, for each $i$ there exists $t_{i}>0$ such that $\w_{0}-t_{i}\mathbf{e}_{i}\in\partial\C_{\tilde{f}}$. Since $\tilde{g}$ is multi-affine, it is increasing on $\w_{0}+\Gamma_{N}^{+}$, and $\w_{0}+\Gamma_{N}^{+}\subset\C_{\tilde{g}}$, a contradiction. Thus $\tilde{g}|_{\C_{\tilde{f}}\setminus\C_{\tilde{g}}}\leq 0$. By Proposition~\ref{proper-position}, $\tilde{f}\prec\tilde{g}$, and hence $f\prec g$.
\end{proof}

 \section{Homogeneous Gårding Polynomials are Lorentzian}
\label{sec:Discussion}
In this section, we explore the connections among \gar{}, real stable, and
Lorentzian polynomials. We examine homogenization and convexity properties
of \gar{} polynomials, and show that homogeneous \gar{} polynomials are
Lorentzian, thereby implying Hodge-type inequalities and log-concavity of
their coefficient sequences. Through examples, we illustrate the structural
differences among these classes and indicate further research directions. 

First, we recall the definition of Lorentzian polynomials in \cite{BrandenHuh20}. 

\begin{defn}[Lorentzian polynomials]\label{defn:Lorentzian}
    A homogeneous degree $d\geq2$ polynomial $f$ with non-negative coefficients is \emph{Lorentzian} if it satisfies 2 conditions:
\begin{enumerate}
    \item For all multi-index $\alpha$ with $|\alpha|=d-2$, $\pdv^\alpha f$ is real stable.\label{enu:Lorentzian1}
    \item The support of $f$ is M-convex.\label{enu:LorentzianMconvex}
\end{enumerate} 
We denote $\L^d$ the class of Lorentzian polynomials of degree $d$.
\end{defn} 

Homogeneous degree $d$ real stable polynomials in $\S^{d,\mathfrak{h}}$ are Lorentzian. The same holds for $\GS^{d,\mathfrak{h}}$, the class of homogeneous degree $d$ \gar{} polynomials.  In fact $\GS^{d,\mathfrak{h}}$ is sandwiched between real-stable and Lorentzian ones.

\begin{thm} Homogeneous \gar{} polynomials are  Lorentzian. Moreover, for $d\geq2$, \begin{align}\label{eq:sandwich}
    \S^{d,\mathfrak{h}}\subset \GS^{d,\mathfrak{h}}\subset \L^d.
\end{align}
\end{thm}

\begin{proof} 

The first inclusion follows from Theorem \ref{thm:stable-garding}. A homogeneous \gar{}  polynomial $f$ has non-negative coefficients. Since both \gar{} and Lorentzian polynomials are preserved under polarization and projection, it suffices to prove the multi-affine case. 

Let $f\in\GS_+^{d,\mathfrak{a,h}}$. For $|\alpha|=d-2$, $\pdv^\alpha f\in \GS^{2,\mathfrak{a},\mathfrak{h}}$. In degree 2, $\GS^{2,\mathfrak{a}}=\S^{2,\mathfrak{a}}$. Hence, $f$ satisfies (\ref{enu:Lorentzian1}) in Definition \ref{defn:Lorentzian}.  

By Theorem \ref{lem:GardingToRayleigh}, $f$ is Rayleigh. By \cite[Theorem 2.23]{BrandenHuh20}, $f$ have M-convex support. Hence,  $f$ satisfies (\ref{enu:LorentzianMconvex}) in Definition \ref{defn:Lorentzian}.  We have finished the proof.
\end{proof}

When $d=2$, all three classes coincide. The case $d=3$ is also special in the following sense.

\begin{prop}\label{H3}
    $\S^{3,\mathfrak{h}}=\GS^{3,\mathfrak{h}}.$
\end{prop}
\begin{proof} By Theorem~\ref{lem:GardingToRayleigh}, it suffices to show that any $f \in \GS^{3,\mathfrak{h}}$ is real stable. Equivalently, it suffices to prove that $f(t\a+\b)$ is real-rooted for any $\a > 0$ and $\b\in \mathbb{R}^n$. 

Since $\a > 0$, by replacing $t$ with $t + T_0$ for a sufficiently large $T_0$, we may assume $\b > 0$. Since $f \in \GS^{3,\mathfrak{h}}$, 
\[h(t,s):=f(t\a+s\b)\in \GS^{3,\mathfrak{h}}[t,s].\]  

We claim that either $h(t,1)$ or $h(1,s)$ is real-rooted. Once proved, $h$ is the homogenization of a real stable polynomial, and hence $h$ is real stable by Lemma~\ref{lem:stablepolyclosurethe} (3).
 Write  $$h(t,s)=a_3t^3+a_2t^2s+a_1ts^2+a_0s^3,\quad a_i\geq0.$$ 

If $a_3=0$, then $h(t,1)\in\GS^2[t]=\S^2[t].$ By symmetry, $h(1,s)$ is real stable when $a_0=0$.

Assume $a_3,a_0\not=0$.  We argue by contradiction. Let $(r_0,r_1,r_2)$ be the root sequence of $h(t,1)$ which is decreasing. Note $r_0<0$. Let $\rho_0<0$ be the largest real root $h(1,s)$. Then $1/\rho_0$ is also a root of $h(t,1)$.  If $h(t,1)$ is not real stable, $1/\rho_0=r_0$ is a single root hence the only real root of $h(t,1)$. Note $1/r_1$ is a root of  $\pdv_th(1,s)$. Since $\pdv_th\vtl h$, we have $\pdv_th(1,s)\vtl h(1,s)$ which implies $\rho_0\geq 1/r_1$. However,  since $r_1<r_0<0$,  
$$0>\frac{1}{r_1}>\frac{1}{r_0}=\rho_0.$$ A contradiction. Thus, we have proved the claim.

Since $h$ is real rooted, $f$ is real stable by Lemma \ref{lem:stablepolyclosurethe} (1). We have finished the proof.
\end{proof}

The following example illustrates the strict  inclusions in \eqref{eq:sandwich} for $d\geq4$.

\begin{example}\label{ex:gard_sta_lorentzian}
Let $c\geq 0$. Consider the bivariate polynomial
\[
f(x,y)=(x^{2}+cxy+y^{2})xy.
\]
Then $f$ is real stable if and only if $c\geq 2$. By \cite[Example 2.26]{BrandenHuh20},  $f$ is Lorentzian if and only if  the coefficient sequence is ultra log-concave, equivalently, if and only if $c\geq 3/2$. We claim that $f$ is \gar{} if and only if $c\geq \sqrt{3}$. In fact, when this holds,
\[
\pdv_x f=\frac{y}{3}(3x+(c-\sqrt{c^2-3})y)(3x+(c+\sqrt{c^2-3})y),
\]
which is real stable. The same holds for $\pdv_y f$. Since $\widehat{\C}_{f}=\Gamma_2^+\subset \C_{\pdv_x f}$ for $c\geq \sqrt{3}$, it follows that $f\in\GS$. If $c<\sqrt{3}$, then $f(x,1)$ is not \gar{}, and hence neither is $f$.

% Further computation indicates that for $n\geq 1$,
% \[
% f_{n}(x,y)=(x^{2}+c_{n}xy+y^{2})x^{n}y^{n}\in\GS_{+}
% \]
% if $c_{n}\geq\frac{2}{\sqrt{3}}\frac{n+2}{n+1}$. In contrast, $f_n\in\S_+$ whenever $c_n\geq 2$, and $f_n\in\L^{n+2}$ provided $c_n\geq \frac{n+2}{n+1}$.
\end{example}

\begin{example}
    The basis generating polynomial of the Fano matroid $F_7$ is an example of a polynomial that is Lorentzian and Rayleigh, but not G{\aa}rding. See Proposition~\ref{prop:F7-rayleigh-not-garding} and Theorem~\ref{F7summary}. 
\end{example}

Next, we show that the homogenization of a Gårding polynomial with nonnegative coefficients need not be a Gårding polynomial, which differs notably from real stable cases.
\begin{example} Let $
f(t)=t(t^{2}+3t+3).$
Since $f$ has the monotone root sequence $R(f)=(0,-1,-1)$, $f(t)\in\GS_+[t]$.
Consider the homogenization of $f$
\[
\text{H}f(t,y)=t^{3}+3t^{2}y+3ty^{2}.
\]
Notice that $\text{H}f(1,y)$ has degree 2 but is not real rooted; hence, it
is not \gar{}. Therefore, $\text{H}f(t,y)$ is not \gar{}.
\end{example}

The above example also shows that if $f=\sum_{k=0}^{d}a_{k}x^{k}$
is univariate \gar{} with nonnegative coefficients, $\tilde{f}=\sum_{k=0}^{d}a_{k}x^{d-k}$
is not necessarily so. 

The $\gar{}$ component of a real-stable polynomial is convex; this property does not extend to $\gar{}$ polynomials.

\begin{example}
Consider 
\[
f(x,y)=(x^{2}+y^{2}+\sqrt{3}xy-1)xy.
\]
As in Example \ref{ex:gard_sta_lorentzian}, $\pdv_xf$ is real stable and $\C_{\pdv_{x}f}$ contains $$\hat{\C}_{f}=\Gamma_{2}^{+}\backslash\{(x,y):x^{2}+y^{2}+\sqrt{3}xy\leq1\}.$$
Similarly, $\C_{\pdv_{y}f}$ also contains
$\hat{\C}_{f}$. Hence, $f$ is \gar{}. Notice that $\C_{f}=\hat{\C}_f$ is not
convex. 
\end{example}

\begin{rem}
A central feature of G{\aa}rding's theory of hyperbolic polynomials is the convexity of their associated G{\aa}rding cones, a property closely tied to analytic log-concavity. From the perspective of partial differential equations, our G{\aa}rding property corresponds only to ellipticity. In general, however, the concavity assumption required for the Caffarelli-Nirenberg-Spruck theory fails.

It is also worth noting that Lorentzian polynomials exhibit log-concavity in the positive hyper-octant. In subsequent work, we will focus on a distinguished subclass of G{\aa}rding polynomials with additional convexity properties and establish a precise connection with Lorentzian polynomials through homogenization.
\end{rem}

\part{Applications}
In this part, we apply the preceding theory to generating functions that arise
in matroid theory, graph theory, probability, and matrix theory. These
functions are typically multi-affine with nonnegative coefficients, making
them a natural setting for the G{\aa}rding framework, which yields new
Rayleigh-type inequalities and negative dependence results beyond those
accessible via stability or Lorentzian methods.

\section{Generating functions of matroids}

We study generating functions of matroids, which provide a natural
testing ground for the Gårding framework and its interaction with
Rayleigh properties in combinatorics. Matroid
theory occupies a central position in modern combinatorics, unifying
phenomena from graphs, linear algebra, and optimization, and has seen
much recent developments surrounding Lorentzian polynomials and combinatorial
Hodge theory.

\subsection{Generating functions on distinguished sets of a matroid}

We introduce the four standard generating functions associated with
a matroid: those of independent sets, bases, spanning sets, and
cospanning sets. Throughout we use the matroid notation from \cite{oxley_matroid_2011}.

Let $M=(E,\mathcal{I},\mathcal{B},r_{M})$ be a finite matroid on the
ground set $E$, where $\mathcal{I}$ denotes the family of independent
sets, $\mathcal{B}$ the family of bases, and $r_{M}$ the rank function.
For each $e\in E,$ let $w_{e}$ be a variable and write $\w=(w_{e})_{e\in E}$.
For any set $S\subset E,$ $w_{S}=\prod_{e\in S}w_{e}$.
\begin{defn}[Generating Functions]
\label{def:matroid}Let $M$ be a matroid as above. We associate with
$M$ the following four multi-affine generating functions.

The \emph{independent-set generating function} (ISGF) of $M$ is 
\[
I_{M}(\w):=\sum_{I\in\mathcal{I}}w_{I}.
\]

The \emph{basis generating function} (BSGF) of $M$ is 
\[
B_{M}(\w):=\sum_{B\in\mathcal{B}}w_{B}.
\]

The \emph{spanning-set generating function }(SSGF) of $M$ is 
\[
S_{M}(\w):=\sum_{\substack{T\subseteq E\\
T\text{ spanning}
}
}w_{T},
\]
where a subset $T\subseteq E$ is called \emph{spanning} if it contains
a basis of $M$.

The \emph{cospanning-set generating function }(CSGF) of $M$, is 
\[
C_{M}(\w):=\sum_{\substack{X\subseteq E\\
E\setminus X\in\mathcal{I}
}
}w_{X}.
\]
\end{defn}

All these polynomials are multi-affine with nonnegative coefficients and
are related by simple combinatorial transformations. Taking
the dual matroid interchanges the SSGF and CSGF; ISGF and CSGF are
related by inversion; and BSGF is the homogeneous component of degree
$r_M$ of both $I_M$ and $S_M$.

From a combinatorial perspective, the theory of real stable polynomials
provides a powerful framework for negative dependence, implying the strong
Rayleigh property for a broad class of generating functions \cite{BBL09}.
A fundamental example is that the basis generating function of every
regular matroid is stable, and hence strongly Rayleigh.
In parallel, log-concavity for sequences arising from matroid generating
functions was established by June Huh and his collaborators 
\cite{Huh2012,HuhKatz2012,huhHvectorsMatroidsLogarithmic2015,AHK18,HuhCorrelationBoundsFields2021a} via combinatorial Hodge theory. More
recently, Lorentzian polynomials
\cite{BrandenHuh20,ALOVi,ALOVii,ALOViii} provided a unified framework for
such log-concavity phenomena.

To simplify notation, we introduce the following definitions.
\begin{defn}
\label{def:Rayleighmatroids}Let $M$ be a matroid with generating
functions $I_{M},B_{M},S_{M}$, $C_{M}$ as above. For $X\in\{I,B,S,C\}$,
we say that $M$ is \emph{$X$--Rayleigh} (respectively, \emph{$X$--Gårding})
if the corresponding generating function $X_{M}$ is Rayleigh (respectively
Gårding). Following Wagner~\cite{MR2248326}, a \emph{Rayleigh matroid}
is one whose basis generating function is Rayleigh; in the present
notation, this is precisely a $B$--Rayleigh matroid. 
\end{defn}

The study of negative dependence and Rayleigh-type inequalities for
matroids has developed along several distinct lines. Feder-Mihail
\cite{Feder1992BalancedM} introduced the concept of negatively correlated
matroids which is equivalent to the following condition: for $i\not=j\in E$,
\[
(B_{M}\pdv_{i}\pdv_{j}B_{M}-\pdv_{i}B_{M}\pdv_{j}B_{M})|_{\w=\mathbf{1}}\leq0,
\]
where $\mathbf{1}$ is the all one vector. Later, Choe-Wagner \cite{MR2248326}
defined the class of Rayleigh matroids which is equivalent to $B$-Rayleigh
matroids in Definition \ref{def:Rayleighmatroids}. $B$-Rayleigh
matroids include all rank 3 matroids \cite{WagnerRank3}, binary
matroids without $S_{8}$ factor, \cite{MR2248326} etc. Therefore all regular matroids
are Rayleigh matroids. In fact, they are HPP matroids (matroids with half-plane-property), which means that their BSGFs are real stable \cite{COSW04}. Many negative dependence conditions
can be derived from the seminal work of Borcea-\branden-Ligget \cite{BBL09}.
Semple-Welsh \cite{semple_negative_2008} introduced correlation inequalities
for spanning sets, extending earlier probabilistic results for graphs. $I$-Rayleigh and $S$-Rayleigh matroids were introduced  by Wagner in \cite{MR2428906}.

In the following, we provide a new framework connecting these generating functions via natural algebraic operations and show that the \gar{} property for the spanning-set and cospanning-set generating functions implies Rayleigh-type inequalities for the basis-set and independent-set generating functions. This yields a structural approach to negative dependence across different matroid generating functions beyond the classical real stable setting. Although the \gar{} property does not capture all Rayleigh polynomials, it applies in settings where stability-based methods are not available and provides a flexible tool for studying negative dependence and other combinatorial properties of matroids.

\subsection{Main Theorems}
We now state the main result of this section, formulated for the spanning-set generating function.

Let $\widehat{U}_{r,n}$ be the matroid obtained by removing one basis from the uniform matroid $U_{r,n}$ \cite{LeeNasrRadcliffe20} which is a paving matroid.  Let $\mathcal{S}$ be the smallest class of matroids that contains all matroids of the form $\widehat{U}_{r,n}$ and is closed under the $2$-sums.

\begin{thm}
\label{thm:main}
Let $M$ be a matroid. Then the spanning-set generating function $S_M(\w)$ is a G{\aa}rding polynomial in each of the following cases:
\begin{enumerate}
\item $M$ is a series--parallel network;
\item $M = U_{r,n}$ is a uniform matroid;
\item $M\in\mathcal S;$
\item $|E(M)| \le 6$.
\end{enumerate}
Moreover, the G{\aa}rding property of $S_M(\w)$ is preserved under deletion, contraction, series--parallel reductions and $2$-sums. Equivalently, the class of $S$-G{\aa}rding matroids is closed under these operations.
\end{thm}

By duality and bottom truncation, we have the following.

\begin{cor}
\label{thm:extension}

Let \(M\) be a matroid belonging to one of the classes listed in
Theorem~\ref{thm:main}. Then the cospanning-set generating function
\(C_M(w)\) and the basis generating function \(B_M(w)\) are
G{\aa}rding polynomials.

Moreover, the class of \(C\)-G{\aa}rding matroids is closed under
 deletion, contraction, series--parallel reductions, and
\(2\)-sums. 
\end{cor}

\begin{rem}
Theorem~\ref{thm:main} and Corollary \ref{thm:extension} provide
natural families of G\aa rding polynomials that are not real stable in general.  
\end{rem}

\subsection{A formula for SSGF}
In this subsection, we prepare for the proof of the main theorems by deriving a combinatorial formula for the SSGF of a matroid.

Let $M=(E,\mathcal{I},\mathcal{B},r_{M})$ be a matroid as above. Define the modified weights
\[
v_{e}=w_{e}+1.
\]
For any subset $S\subset E$, write $v_{S}:=\prod_{e\in S}v_{e}$, and set $v_{\emptyset}=1$. Let $\mathcal{L}(M)$ denote the lattice of flats of $M$, partially ordered by inclusion, with $E$ as the top element and $\emptyset$ as the bottom element. Let $\mu(\cdot,\cdot)$ be the Möbius function of $\mathcal{L}(M)$.

\begin{thm}
\label{thm:SSGF}
With the above notation, the spanning-set generating function of $M$ satisfies
\begin{equation}
S_{M}(\w)=\sum_{F\in\mathcal{L}(M)}\mu(F,E)\,v_{F}.
\label{eq:flat formula}
\end{equation}
Moreover, $S_{M}$ satisfies the recursive identity
\begin{equation}
S_{M}(\w)=v_{E}-\sum_{\substack{F\in\mathcal{L}(M)\\ F\neq E}}S_{M|F}(\w),
\label{eq:recursive}
\end{equation}
where $M|F$ denotes the restriction of $M$ to  flat $F$.
\end{thm}
 
\begin{proof}
For each flat $\alpha\in\mathcal{L}(M)$, define a pair of polynomials
\[
F(\alpha):=\sum_{\mathrm{cl}(T)=\alpha}w_{T},\quad G(\alpha):=\sum_{\beta\subset\alpha}F(\beta).
\]
Since the closure $\mathrm{cl}(T)\subset\alpha$ is equivalent to $T\subset\alpha$,
we have 
\[
G(\alpha)=\sum_{\beta\subset\alpha}\sum_{\mathrm{cl}(T)=\beta}w_{T}=\sum_{T\subset\alpha}w_{T}=\prod_{e\in\alpha}(1+w_{e})=v_{\alpha}.
\]
The Möbius inversion formula \cite[Proposition 2]{Rota1964} on $\mathcal{L}(M)$
shows 
\[
F(\alpha)=\sum_{\beta\subset\alpha}\mu(\beta,\alpha)\,v_{\beta}.
\]
Summing over all proper flats and using the standard identity 
\[
\sum_{\beta\subset\alpha\subsetneq E}\mu(\beta,\alpha)=-\mu(\beta,E),
\]
We have proved (\ref{eq:flat formula}). 

For each proper flat $F\in\mathcal{L}(M)$, applying (\ref{eq:flat formula})
to $S_{M|F}$, summing them over, and interchanging the order of summation,
we have 
\begin{align*}
\sum_{\substack{\text{proper }F\in\mathcal{L}(M)}
}S_{M|F} & =\sum_{\substack{\text{proper }F\in\mathcal{L}(M)}
}\sum_{\substack{G\in\mathcal{L}(M|F)}
}\mu(G,F)\,v_{G}.\\
 & =\sum_{\text{proper\ }G\in\mathcal{L}(M)}\left(\sum_{\substack{F\in\mathcal{L}(M):\;G\subseteq F\subsetneq E}
}\mu(G,F)v_{G}\right)
\end{align*}
Use the Möbius identity
\[
\sum_{\substack{F\in\mathcal{L}(M):\;G\subseteq F\subsetneq E}
}\mu(G,F)=-\mu(G,E), \qquad\text{for }G\subsetneq E,
\]
and (\ref{eq:flat formula}) to conclude (\ref{eq:recursive}).
\end{proof}
\begin{rem}
Since the CSGF of a matroid $M$ is the SSGF of its dual $M^{*}$, similar formulae hold for the CSGF.
\end{rem}

\section{S-\gar{} and C-\gar{} matroids}
In this section, we study spanning and co-spanning generating functions of matroids. The main goal is to prove Theorem~\ref{thm:main} and Corollary~\ref{thm:extension} which give examples of $S$- and $C$-\gar{} matroids. The proof combines closure properties of the G{\aa}rding class with verification on fundamental classes of matroids. 

We begin with concrete examples, including uniform matroids, then show that the G{\aa}rding property is preserved under standard matroid operations, and finally verify the result for the base classes appearing in Theorem~\ref{thm:main}. For completeness, we also record formulas for both the SSGF and the CSGF.

\subsection{Trees, cycles, uniform matroids and their modifications}

We introduce some concrete examples, beginning with graphic ones.
Let $G=(V,E)$ be a finite graph, and let $M(G)$ be the associated (cycle)
matroid.
\begin{example}
If $T$ is a tree, then $E$ is the unique spanning tree. Consequently,
\[
S_{M(T)}(\w)=\prod_{e\in E}w_{e}=w_{E},\qquad C_{M(T)}(\w)=\prod_{e\in E}(1+w_{e})=v_{E}.
\]
Both polynomials are stable, and hence G{\aa}rding. Therefore, every tree is  both $S$--G{\aa}rding and $C$--G{\aa}rding.
\end{example}

\begin{example}
Let $C_{n}$ be the cycle graph on $n$ edges. A subset is spanning in $M(C_{n})$ if it contains at least $n-1$ elements; it is cospanning if it is nonempty. Therefore,
\[
S_{M(C_{n})}(\w)=\sigma_{n}(\w)+\sigma_{n-1}(\w),\quad C_{M(C_{n})}(\w)=\prod_{j=1}^{n}(1+w_{j})-1.
\]
Both polynomials are symmetric and their diagonal specializations satisfy the monotone root sequence (MRS) property. Hence, by Theorem~\ref{thm:recursivedef}, they are both G{\aa}rding. Thus, the cycle $C_{n}$ is both $S$--G{\aa}rding and $C$--G{\aa}rding.
\end{example}

In general, graphic matroids are not $S$-\gar{}. 
\begin{example}\label{exam:K5}
Take the lexicographic order on the edges of $K_5$.  Let \[\mathbf{b}=(1, 2, 1, -2, -2, -2, 2, -2, 1, 2).\] Then \[
\begin{aligned}
p(t):=&S_{M(K_5)}(\mathbf{b}+t\mathbf{1})=t^{10}+11t^9+39t^8+19t^7-168t^6-252t^5\\
&+215t^4+475t^3-140t^2-266t+112.
\end{aligned}
\]
Since $r(p)\approx-1.398364 < 1.360999 \approx r(p')$, by Remark \ref{rem:MRSnecessary}, $S_{M(K_5)}$ is not \gar{}. 
\end{example}

Next, we give a simple non-graphic example.

\begin{example}
\label{exa:uniform}
Let $U_{r,n}$ be the uniform matroid of rank $r$ on the ground set $E=[n]$. A subset of $E$ is spanning if and only if it has cardinality at least $r$. Hence
\[
S_{U_{r,n}}(\w)=\sum_{j=r}^{n}\sigma_{j}(\w).
\]
Let $P_{r,n}=\proj S_{U_{r,n}}$. It is straightforward to check that the root sequence of $P_{r,n}$ is monotone and consists only of $0$ and $-1$. Hence $P_{r,n}$ is univariate G{\aa}rding. Therefore, $U_{r,n}$ is $S$-\gar{}. By duality, $U_{r,n}$ is also $C$-\gar{}.
\end{example}

We now consider a simple modification of the uniform matroid. Fix a basis
$B_{0}=\{1,2,...,r\}$. Define the deleted-basis matroid 
\[
\widehat{U}_{r,n}:=U_{r,n}\setminus B_{0}.
\]
Equivalently, $\widehat{U}_{r,n}$ is the rank-$r$ matroid on $n$
elements with exactly one circuit-hyperplane of size $r$. The class
of matroids $\widehat{U}_{r,n}$ is closed under duality; in fact,
$(\widehat{U}_{r,n})^{*}\cong\widehat{U}_{n-r,n}.$ This is a class
of  sparse paving matroids, where any subset of $r$-element is
either a basis or a circuit-hyperplane. 

For their corresponding SSGFs, we have the following
\begin{equation}
S_{\widehat{U}_{r,n}}(w)=S_{U_{r,n}}(w)-w_{B_{0}}\label{eq:removed}.
\end{equation}
Let $q=n-r$ and let
\[
g_{r,q}(x,y):=S_{U_{r,n}}(\underbrace{x,\cdots,x}_{r},\underbrace{y,\cdots,y}_{q}),\qquad f_{r,q}(x,y):=S_{\hat{U}_{r,n}}(\underbrace{x,\cdots,x}_{r},\underbrace{y,\cdots,y}_{q}).
\]
Then 
\(
f_{r,q}(x,y)=g_{r,q}(x,y)-x^{r}.
\)

We prove the following theorem.
\begin{thm}
\label{thm:Uhat}Notations as above. The SSGF of $\widehat{U}_{r,n},$
and its two variable specialization $f_{r,q}(x,y)$ are both Gårding
polynomials.
\end{thm}

Theorem \ref{thm:Uhat} establishes the Gårding property of an infinite
class of non-regular matroids. See Appendix A for its proof.

\subsection{Closure under minor operations}

We relate the SSGF and CSGF of a matroid to those of their minors. 

Given a fixed matroid $M$ on $E$, for any minor $N$, we regard
the SSGF and CSGF of $N$ as a polynomial in variables $\{w_{e}:e\in E(N)\}$. 

Let $M=(E,\mathcal{I})$ be a matroid and let $e\in E$. The \emph{deletion}
$M\setminus e$ is the matroid on ground set $E\setminus\{e\}$ whose
independent sets are those independent sets of $M$ deleting $e$. The \emph{contraction} $M/e$ is the matroid on ground set $E\setminus\{e\}$
whose independent sets are those sets $I\subseteq E\setminus\{e\}$
such that $I\cup\{e\}$ is independent in $M$. An element $e$ is
a \emph{loop} if it is contained in no independent set, and a \emph{coloop}
if it is contained in every basis of $M$.
\begin{lem}[Deletion--contraction formulae]
 \label{lem:DCforCSGF} For any matroid $M$ and $e\in E$, we have
\begin{align}
\SSGF_M(w)=\begin{cases}
w_{e}S_{M/e}(\w), & e\text{ is a coloop},\\
(w_{e}+1)S_{M\backslash e}(\w), & e\text{ is a loop},\\
w_{e}S_{M/e}(\w)+\SSGF_{M\backslash{}e}(\w), & \mathrm{otherwise,}
\end{cases}
\end{align}

\begin{equation}
\CSGF_{M}(\w)=\begin{cases}
w_{e}\CSGF_{M\backslash e}(\w), & e\text{ is a loop},\\
(w_{e}+1)\CSGF_{M\backslash e}(\w), & e\text{ is a coloop},\\
w_{e}\CSGF_{M\backslash e}(\w)+\CSGF_{M/e}(\w), & \text{ otherwise}.
\end{cases}\label{eq:CSGFDeletioncontraction}
\end{equation}
\end{lem}

\begin{proof}
For $S_{M}$, partition the spanning sets $T\subseteq E$ according
to whether $e\in T$. If $e\notin T$, then $T$ is a spanning set
of $M\setminus e$, contributing a term in $S_{M\setminus e}(\mathbf{w})$.
If $e\in T$ and $e$ is neither a loop nor a coloop, then $T\setminus\{e\}$
is a spanning set of $M/e$, contributing $w_{e}\,S_{M/e}(\mathbf{w})$.
The loop and coloop cases follow directly from the definitions.

The proof for $C_M$ is similar. 
%For $C_{M}$, partition the cospanning sets $X\subseteq E$ according to whether $e\in X$. If $e\notin X$, then $X$ is a cospanning set of $M\setminus e$. If $e\in X$ and $e$ is neither a loop nor a coloop, then $X\setminus\{e\}$ is a cospanning set of $M/e$, yielding the stated formula. Again, the loop and coloop cases are immediate.
%
\end{proof}

\begin{prop}
\label{prop:minor-closed}The classes of $S$- and $C$-\gar{} matroids are both minor closed.
\end{prop}
\begin{proof} 
Let $M$ be a $C$-\gar{} matroid. Then
since $C_M(\w)$ has non-negative coefficients, $\pdv_{w_{e}}C_M$ and $C_M|_{w_{e}=0}$
are both \gar{} polynomials by Lemma \ref{lem:Suppose-that--2} and \ref{lem:For-any-.}. If $e$ is a loop, then $M\backslash{}e=M/e$. Otherwise, by the deletion--contraction
formula in Lemma \ref{lem:DCforCSGF}, both $M\backslash{}e$ and $M/e$
are $C$-\gar{} matroids. 

The proof for $S$-\gar{} matroids is similar and omitted.
\end{proof}

\begin{lem} Both $\SSGF_{M/e}-\SSGF_{M\backslash{}e}$ and $\CSGF_{M\backslash{}e}-\CSGF_{M/e}$ have non-negative coefficients. \label{lem:minorcomparison}
\end{lem}
\begin{proof}
We prove for SSGF. The same statements hold
for CSGF by duality. 

If $e\in E$ is a coloop, then $M/e=M\backslash e$. Hence the result holds trivially. 

Suppose that $e\in E$ is not a coloop of $M$. It suffices to show that a spanning set  $S\subset E-e$ in $M\backslash{}e$ is also spanning in $M/e$. Notice that  $\text{rk}_{M}(S)=\text{rk}(M)$.
We have 
\[
\text{rk}_{M/e}(S)=\text{rk}(M)-1=\text{rk}(M/e).
\]
Hence, $S$ is also a spanning set in $M/e$. 
\end{proof}
Define
\[
\xi_{e}(M)(\w)=\CSGF_{M\backslash e}(\w)-\CSGF_{M/e}(\w).
\]
By Lemma \ref{lem:minorcomparison}, $\xi_e(M,\w)$ has non-negative coefficients. We then obtain the following characterization of the $C$-\gar{} matroids.

\begin{prop}
With the above notation, the following are equivalent:
\begin{enumerate}
\item $M$ is a $C$-\gar{} matroid.
\item There exists $e\in E$ such that $M\backslash e$ is a $C$-\gar{} matroid and $\xi_{e}(M;\w)\vtl \CSGF_{M\backslash e}(\w)$.
\end{enumerate}
\end{prop}

\begin{proof}
It suffices to prove for matroids without loops. Denote $f(\w)=\CSGF_{M\backslash e}(\w)$, $g(\w)=\CSGF_{M/e}(\w)$
and $\xi(\w)=\xi_{e}(M)(\w).$ Then, $\xi=f-g$. By \eqref{eq:CSGFDeletioncontraction},
\[
\CSGF_{M}(\w)=f(\w)(w_{e}+1)-\xi(\w).
\]
By Theorem \ref{thm:Suppose-that-is},  
$\CSGF_{M}$ is \gar{} if and only if either $\xi\vtl f$ or $f \prec -\xi$. The latter possibility is excluded  by Lemma~\ref{lem:minorcomparison}. Thus, $\CSGF_M$ is \gar{} if and only if $\xi\vtl f$.
\end{proof}

\subsection{Series-Parallel connections and 2-sums}

Consider 2 matroids $M_{1}$ and $M_{2}$ over ground sets $E_{1}$
and $E_{2}$ respectively, and $E_{1}\cap E_{2}=\{e\}$. Let $E=E_{1}\cup E_{2}$.
Denote 
\begin{enumerate}
\item $S(M_{1},M_{2})$ the \emph{series connection }of $M_{1}$ and $M_{2}$,
\item \emph{$P(M_{1},M_{2})$ }the\emph{ parallel connection} of $M_{1}$
and $M_{2}$,
\item $M_{1}\oplus_{2}M_{2}=S(M_{1},M_{2})/e$ the $2$-sum of $M_{1}$
and $M_{2}$.
\end{enumerate}
We refer to \cite[Chapter 7]{oxley_matroid_2011} for precise definitions
of series-parallel connections and $2$-sums.
\begin{lem}
Let $M_{1}$ and $M_{2}$ be 2 matroids over ground sets $E_{1}$
and $E_{2}$ respectively. 
\begin{enumerate}
\item For $E_{1}\cap E_{2}=\emptyset$, let $E=E_{1}\cup E_{2}$.
Then 
\begin{equation}
\CSGF_{M_{1}\oplus M_{2}}(\w)=\CSGF_{M_{1}}(\w)\CSGF_{M_{2}}(\w).\label{eq:CSGFDirectsum}
\end{equation}
\item For $\{e\}=E_{1}\cap E_{2}$, let $f_{e}=\CSGF_{M_{1}\backslash e}$,
$f^{e}=\CSGF_{M_{1}/e}$, $g_{e}=\CSGF_{M_{2}\backslash e}$, and
$g^{e}=\CSGF_{M_{2}/e}$. Then 
\begin{align}
\CSGF_{P(M_{1},M_{2})}(\w) & =\begin{cases}
w_{e}f_{e}g^{e}, & e\ \mathrm{is\ a\ loop\ of\ }M_{1},\\
w_{e}\left(f_{e}g^{e}+f^{e}g_{e}-f^{e}g^{e}\right)+f^{e}g^{e}, & \mathrm{else.}
\end{cases}\label{eq:CSGFParallel}\\
\CSGF_{S(M_{1},M_{2})}(\w) & =w_{e}f_{e}g_{e}+f_{e}g^{e}+f^{e}g_{e}-f^{e}g^{e}.\label{eq:CSGFSeries}
\end{align}
\end{enumerate}
\end{lem}

\begin{proof}
For the first statement, if $M=M_{1}\oplus M_{2}$, then any independent
set of $M$ is a disjoint union of independent subsets of $M_{1}$
and $M_{2}$ respectively. (\ref{eq:CSGFDirectsum}) follows immediately. 

For the second statement, if $e$ is a loop of $M_{1}$, then the
result reduces to the first statement. Thus we may assume that $e$
is neither a loop of  $M_{1}$ nor $M_{2}$. There are 3 families
of independent subsets in $P(M_{1},M_{2})$. 
\begin{enumerate}
\item $\mathcal{I}^{e}=\{I_{1}\cup I_{2}\cup e:I_{1}\in\mathcal{I}(M_{1}/e),I_{2}\in\mathcal{I}(M_{2}/e)\}$. 
\item $\mathcal{I}_{e}'=\{I_{1}\cup I_{2}:I_{1}\in\mathcal{I}(M_{1}/e),I_{2}\in\mathcal{I}(M_{2}\backslash{}e)\}.$ 
\item $\mathcal{I}_{e}''=\{I_{1}\cup I_{2}:I_{1}\in\mathcal{I}(M_{1}\backslash{}e),I_{2}\in\mathcal{I}(M_{2}/e)\}.$ 
\end{enumerate}
Also notice that 
\[
\mathcal{I}_{e}'\cap\mathcal{I}_{e}''=\{I_{1}\cup I_{2}:I_{1}\in\mathcal{I}(M_{1}/e),I_{2}\in\mathcal{I}(M_{2}/e)\}.
\]
Therefore, by inclusion-exclusion principle, we have 
\begin{align*}
\CSGF_{M}(\w) & =\sum_{I\in\mathcal{I}_{e}'}w_{E\backslash{I}}+\sum_{I\in\mathcal{I}_{e}''}w_{E\backslash{I}}-\sum_{I\in\mathcal{I}_{e}'\cap\mathcal{I}_{e}''}w_{E\backslash{I}}+\sum_{I\in\mathcal{I}^{e}}w_{E\backslash{I}}\\
 & =f_{e}g^{e}w_{e}+f^{e}g_{e}w_{e}-f^{e}g^{e}w_{e}+f^{e}g^{e}
\end{align*}
The proof of (\ref{eq:CSGFSeries}) is similar. 
\end{proof}
\begin{lem}
\label{lem:CSGFvtl} Suppose that $f,g\in\G$ and $p\vtl f,q\vtl g$.
Then 
\[
f(\x)g(\y)\prec f(\x)g(\y)-p(\x)q(\y)\prec(f(\x)-p(\x))\cdot(g(\y)-q(\y)).
\]
\end{lem}

\begin{proof}
Both relations follow directly from Theorem \ref{proper-position},
\ref{enu:BBB}$\Rightarrow$\ref{enu:AAA}.
\end{proof}

We next show that series-parallel connections and $2$-sums preserve
$S/C$-\gar{} properties.
\begin{prop}
\label{thm:SP_CSGF} The classes of S-\gar{} and C-\gar{} matroid
are preserved under series-parallel connections and 2-sums.
\end{prop}

A direct consequence of Theorem \ref{thm:SP_CSGF} shows that $M$
is $S$- or $C$-\gar{} if and only if the simplification matroid
$\text{si}(M)$ is $S$- or $C$-\gar{} respectively.
\begin{proof}
By duality, it suffices to treat the $C$-\gar{} case. We proceed
in three steps, treating the parallel connection, the series connection,
and the $2$-sum using lemmas above. 

Let $p=\xi_{e}(M_{1})$, $q=\xi_{e}(M_{2}),f=f^{e},g=g^{e}$. By (\ref{eq:CSGFParallel}), and (\ref{eq:CSGFSeries}), 
\[\begin{aligned}
\CSGF_{P(M_{1},M_{2})}(\w)&=w_{e}\left(f_{e}g_{e}-\xi_{e}(M_{1})\xi_{e}(M_{2})\right)+f^{e}g^{e},\\
\CSGF_{S(M_{1},M_{2})}(\w)&=w_{e}f_{e}g_{e}+f_{e}g_{e}-\xi_{e}(M_{1})\xi_{e}(M_{2}).
\end{aligned}\]
By Lemma \ref{lem:CSGFvtl}, we conclude that  series-parallel connections
preserve $C$-\gar{} matroids.

Since $M_{1}\oplus_{2}M_{2}$ is a minor of $S(M_{1},M_{2})$, 2-sums also preserve $C$-\gar{} matroids. 
\end{proof}
As an immediate corollary, the cycle matroids
for series-parallel networks are $S/C$-\gar{}. For the definition
of series-parallel network, we refer to \cite[Section 5.4]{oxley_matroid_2011}.

\subsection{Matroids of at most 6 elements}

\begin{thm}
\label{thm:Any-matroid-on6elements}Any matroid on at most 6 elements
is $S$-Gårding. 
\end{thm}
See Appendix \ref{6elements} for its proof. By the standard decomposition theory of connected matroids into
\(3\)-connected components via direct sums and \(2\)-sums, together with

Proposition~\ref{thm:SP_CSGF}, it suffices to verify the required
property for \(3\)-connected matroids on six elements.

Up to isomorphism, there are only five such matroids; see
Figure~\ref{fig:3-connnected-rank-3}. In each case, we compute the SSGF
using Theorem~\ref{thm:SSGF} with the labeling shown in the figure, and
verify that the resulting SSGF is \(\gar{}\).

\begin{figure}
\subfloat[\small$M(K_{4})$]{
\begin{tikzpicture}[scale=0.75, line cap=round, line join=round]

% ----- coordinates (Wikipedia-style layout) -----
\coordinate (P3) at (60:2);     % top
\coordinate (P1) at (210:2.6);    % bottom-left
\coordinate (P5) at (330:2.6);    % bottom-right

% points on edges (chosen as midpoints for the look)
\path (P1) -- (P3) coordinate[midway] (P2); 
\path (P1) -- (P5) coordinate[midway] (P4); 

\draw [thick, name path=line1] (P5) -- (P2);
\draw [thick, name path=line2] (P3) -- (P4);
\path [name intersections={of=line1 and line2, by=P6}];

\fill (P6) circle (0.03);
% ----- the 7 lines (flats) -----
% Outer triangle edges: {1,2,3}, {1,4,5}, {2,4,6}
\draw[thick, name path=l1] (P1) -- (P2);
\draw[thick] (P1) -- (P3);
\draw[thick] (P1) -- (P5);

% ----- points -----
\foreach \i in {1,2,3,4,5,6}{
  \fill (P\i) circle (2.5pt);
}
% ----- labels (slight offsets to mimic the SVG) -----
\node[above] at (P1) {$1$};
\node[above left] at (P2) {$3$};
\node[below right] at (P4) {$6$};

\node[right] at (P3) {$2$};
\node[right] at (P5) {$5$};
\node[above right] at (P6) {$4$};

\end{tikzpicture}

}$\quad$\subfloat[$W^{3}$]{\begin{tikzpicture}[scale=0.75]
% ----- coordinates (Wikipedia-style layout) -----
\coordinate (P1) at (90:2);     % top
\coordinate (P2) at (210:2);    % bottom-left
\coordinate (P4) at (330:2);    % bottom-right

% points on edges (chosen as midpoints for the look)
\path (P1) -- (P2) coordinate[midway] (P3); % on 1-2
\path (P1) -- (P4) coordinate[midway] (P5); % on 1-4
\path (P2) -- (P4) coordinate[midway] (P6); % on 2-4

% ----- the 7 lines (flats) -----
% Outer triangle edges: {1,2,3}, {1,4,5}, {2,4,6}
\draw[thick] (P1) -- (P2);
\draw[thick] (P1) -- (P4);
\draw[thick] (P2) -- (P4);

% Circle line: {3,5,6}

% ----- points -----
\foreach \i in {1,2,3,4,5,6}{
  \fill (P\i) circle (2.5pt);
}

% ----- labels (slight offsets to mimic the SVG) -----
\node[above] at (P1) {$1$};
\node[below left] at (P2) {$3$};
\node[below right] at (P4) {$5$};

\node[left] at (P3) {$2$};
\node[right] at (P5) {$6$};
\node[below] at (P6) {$4$};

\end{tikzpicture}

}$\quad$\subfloat[$Q_{6}$]{\begin{tikzpicture}[scale=0.75]
  % Put 2 as the left vertex
  \coordinate (P2) at (0,0);

  % Line {1,2,4}: horizontal ray to the right
  \coordinate (P1) at (1.8,0.0);
  \coordinate (P4) at (3.6,0.0);

  % Line {2,3,5}: ray at 45 degrees (slope 1)
  \coordinate (P3) at (1.4,1.4);
  \coordinate (P5) at (2.8,2.8);

  % Element 6 hanging between the two lines
  \coordinate (P6) at (3.6,1.5);

  % Draw supporting lines (extend a bit past endpoints)
  \draw[thick] (-0,0) -- (3.6,0);          % supports {1,2,4}
  \draw[thick] (-0,-0) -- (2.8,2.8);     % supports {2,3,5}

  % Draw points
  \fill (P2) circle (2.5pt) node[left] {$3$}; % emphasize vertex 2

  \fill (P1) circle (2.5pt) node[below] {$4$};
  \fill (P4) circle (2.5pt) node[below] {$5$};

  \fill (P3) circle (2.5pt) node[above left] {$2$};
  \fill (P5) circle (2.5pt) node[above left] {$1$};

  \fill (P6) circle (2.5pt) node[right] {$6$};

  % Optional line labels
\end{tikzpicture}}

\subfloat[\mbox{$P_{6}$}]{\parbox{3.6cm}{
    \centering
\begin{tikzpicture}[scale=0.7] 
	% Points 
	\fill (0,1.8) circle (2.5pt) node[above] {$4$}; 
\fill (1.8,2) circle (2.5pt) node[above] {$5$};
\fill (3.6,1.8) circle (2.5pt) node[above] {$6$}; 
\fill (0,0) circle (2.5pt) node[below] {$1$};
\fill (1.8,0) circle (2.5pt) node[below] {$2$};
\fill (3.6,0) circle (2.5pt) node[below] {$3$}; % The only line: 4-5-6 
	\draw[thick] (0,0) -- (3.6,0); 
\end{tikzpicture}}

}$\quad$$\quad\ $\subfloat[$U_{3,6}$]{\parbox{3.6cm}{
    \centering
\begin{tikzpicture}[scale=0.8, line cap=round, line join=round]

% ----- coordinates (Wikipedia-style layout) -----
\coordinate (P1) at (0:1.4);     % top
\coordinate (P2) at (60:1.4);    % bottom-left
\coordinate (P3) at (120:1.4);    % bottom-right
\coordinate (P4) at (180:1.4);    % bottom-right
\coordinate (P5) at (240:1.4);    % bottom-right
\coordinate (P6) at (300:1.4);    % bottom-right
\foreach \i in {1,2,3,4,5,6}{
  \fill (P\i) circle (2.5pt);
}
% ----- labels (slight offsets to mimic the SVG) -----
\node[right] at (P1) {$1$};
\node[right] at (P2) {$2$};
\node[left] at (P4) {$4$};

\node[left] at (P3) {$3$};
\node[left] at (P5) {$5$};
\node[right] at (P6) {$6$};

\end{tikzpicture}}

}$\ $

\caption{3-connected rank 3 matroids with 6 elements}\label{fig:3-connnected-rank-3}
\end{figure}
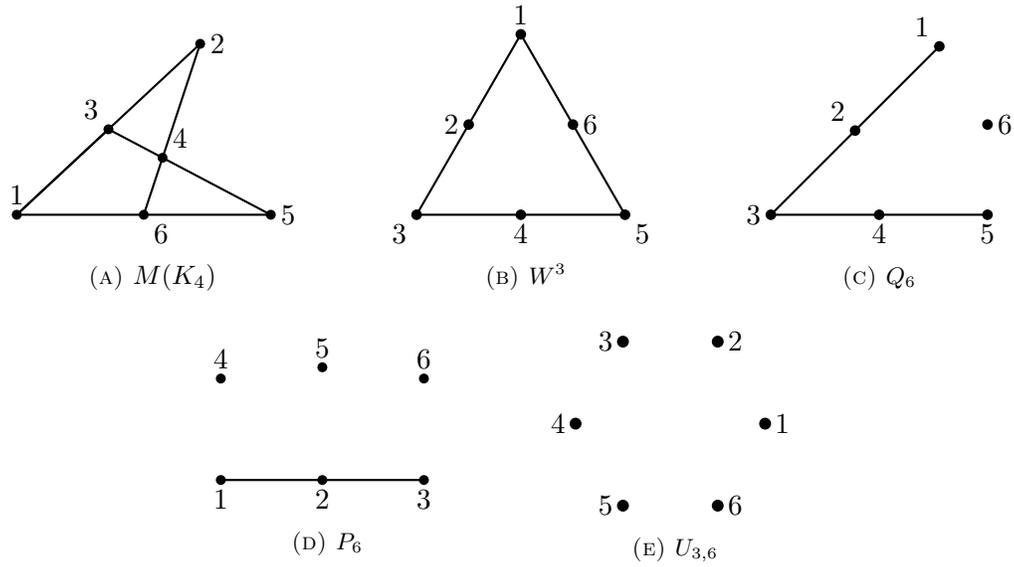

\subsection{Conclusion of the proof}
\begin{proof}[Proof of Theorem~\ref{thm:main}]
The result follows by combining the closure properties from Propositions~\ref{prop:minor-closed}
and \ref{thm:SP_CSGF}, Example \ref{exa:uniform}, Theorem \ref{thm:Uhat}
and Theorem \ref{thm:Any-matroid-on6elements}.
\end{proof}
\begin{proof}[Proof of Corollary~\ref{thm:extension}]
Let $M$ be a matroid in one of the classes described in Theorem~\ref{thm:main}.
We first consider the cospanning-set generating function. A subset
$X\subseteq E$ is cospanning in $M$ if and only if its complement
is spanning in the dual matroid $M^{\ast}$. Therefore 
\[
C_{M}(w)=S_{M^{\ast}}(w).
\]
Since the relevant classes are closed under duality, Theorem~\ref{thm:main}
implies that $S_{M^{\ast}}(w)$ is G{\aa}rding. Hence $C_{M}(w)$
is G{\aa}rding.

We now turn to the basis generating function. A basis is precisely
a spanning set of cardinality $\operatorname{rk}(M)$, so $B_{M}(w)$
is obtained from $S_{M}(w)$ by taking the homogeneous component of
the lowest degree;
\[
B_{M}\;=\;(S_{M})^{\mathrm{bot}}.
\]
The conclusion follows from Proposition \ref{prop:truncateLet--be}. 

Thus both $C_{M}(w)$ and $B_{M}(w)$ are G{\aa}rding, as claimed. 
\end{proof}
\begin{rem}
Our method can be applied to many more special matroids. A systematic
discussion will be included in future works.
\end{rem}

\section{Rayleigh properties }

We record some consequences of our main theorems in terms of Rayleigh
properties.

\subsection{Relations among generating functions}

We introduce some algebraic relations among the generating functions.
\begin{defn}[Combinatorial Inversion]
For a multivariate polynomial $f(w)=f((w_{e})_{e\in E})$, we define
its \emph{combinatorial inversion} by 
\[
\Inv(f)(\w):=\Bigl(\prod_{e\in E}w_{e}\Bigr)\,f\!\left(\left(\frac{1}{w_e}\right)_{e\in E}\right).
\]
\end{defn}
\begin{rem}
\label{inversion keeps rayleigh} The inversion operator $\Inv$ is
purely combinatorial and should not be confused with the inversion
operator $T_{\tau}$ introduced earlier in the paper. The two operations
differ by explicit sign factors: $T_{\tau}$ involves a signed transformation
arising from complex-analytic considerations, whereas $\Inv$ is sign-free
and is induced simply by complement of subsets. Especially,
$\Inv$ preserves the Rayleigh property for multi-affine polynomials
with nonnegative coefficients. 
\end{rem}

\medskip{}

The four generating functions are related as follows. The basis generating
function is the top-degree homogeneous component of the independent-set
generating function. It is also the bottom-degree homogeneous component
of the spanning-set generating function: 
\begin{equation}
B_{M}\;=\;(I_{M})^{\mathrm{top}}=(S_{M})^{\mathrm{bot}}.\label{eq:new1}
\end{equation}
Moreover, 
\begin{equation}
C_{M}=\Inv(I_{M}),\label{eq:new2}
\quad  B_{M}\;=\;\Inv((\Inv(S_{M}))^{\mathrm{top}}).
\end{equation}
If $M^{*}$ is the dual matroid of $M$, then the spanning-set and
cospanning-set generating functions are interchanged under duality:
\[
S_{M^{*}}\;=\;C_{M}\qquad\text{and}\qquad C_{M^{*}}\;=\;S_{M}.
\]

\begin{rem}
In the graphic case, the CSGF admits a concrete interpretation. Let
$G=(V,E)$ be a graph, and let $M(G)$ be its cycle matroid. The \emph{unrooted}
\emph{spanning-forest generating function} (SFGF) of $G$ is 
\[
F_{G}(\w):=\sum_{F\subseteq E,\;F\text{ a spanning forest}}\prod_{e\in F}w_{e}.
\]
It is easy to see that the cospanning-set generating function $C_{M(G)}$
is the combinatorial inversion of $F_{G}$, namely 
\[
C_{M(G)}(\w)=\Inv(F_{G})(\w).
\]
For a connected graph $G=(V,E)$, the spanning--set generating function
can be recovered from the multivariate $q$-Potts model partition
function $Z_{G}(q,\w)$ \cite{sokal_multivariate_2005} by
\[
S_{M(G)}(\mathbf{w})=\lim_{q\to0}\,\frac{1}{q}\,Z_{G}(q,\mathbf{w}).
\]
\end{rem}

\begin{rem}
By Proposition \ref{prop:truncateLet--be}, truncation to the top
or bottom homogeneous component preserves the multi-affine \gar{}
property for polynomials with nonnegative coefficients. Also, it is
well known that both truncation and inversion preserve the Rayleigh
property.
\end{rem}

\subsection{Rayleigh consequences}

We record several basic facts about the  \gar{} and Rayleigh properties
of the spanning-set and cospanning-set generating functions.
\begin{lem}
\label{lem:SCG-implies-Rayleigh} Let $M$ be a matroid with generating
functions $I_{M},B_{M},S_{M},C_{M}$ as above. Then

\begin{enumerate}
\item If $S_{M}$ (resp. $C_{M}$) is G{\aa}rding, then $S_{M}$ (resp. $C_{M}$) is Rayleigh.
\item $I_{M}$ is Rayleigh if and only if $C_{M}$ is Rayleigh.
\item Let $G$ be a graph, and let $M(G)$ be its cycle matroid. Then $C_{M(G)}$ is Rayleigh if and only if the unrooted spanning-forest generating function $F_{G}(w)$ is Rayleigh.
\item If for some $X\in\{I,S,C\}$, $X_{M}$ is Rayleigh, then $B_{M}$ is Rayleigh.
\item If $S_{M}$ or $C_{M}$ is \gar{}, then $B_{M}$ is \gar{}.
\label{enu:necessary condition}
\end{enumerate}
\end{lem}
\begin{proof}
Part~(1) is an immediate consequence of Lemma~\ref{lem:GardingToRayleigh}. Parts~(2) and~(3) follow from Remark~\ref{inversion keeps rayleigh}. Since taking the top-degree part of a polynomial preserves the Rayleigh property, part~(4) follows from \eqref{eq:new1} and \eqref{eq:new2}. Part~(5) follows from Proposition~\ref{prop:truncateLet--be}.
\end{proof}

\begin{cor}
\label{cor:consequence}
For all matroids covered in Theorem~\ref{thm:main}. Namely, series-parallel networks, uniform matroids and their one basis removal, and matroids on at most six elements—their SSGFs, CSGFs, ISGFs, and BSGFs are Rayleigh. For the corresponding graph examples, the SFGFs are also Rayleigh.
\end{cor}

\begin{rem}
Although several statements in Corollary~\ref{cor:consequence} are known \cite{Erickson12,MR2248326,MR2428906}, their proofs do not arise from the classical strong Rayleigh framework and were generally obtained through case-by-case analysis. In contrast, our unified approach via the \gar{} property recovers these results in a systematic manner. Moreover, the corresponding statements for $\widehat{U}_{r,n}$ appear to be new.

Erickson proved the Rayleigh property for the unrooted spanning-forest generating polynomial of a series-parallel graph by 
a direct method \cite{Erickson12}. The argument above recovers this result as a consequence of the G{\aa}rding theory, providing a structural explanation.
\end{rem}

\subsection{2-symmetric rank 4 matroids}To highlight the analytic nature of our approach, we consider all rank 4 matroids with 2 block symmetry. Let
\[
\mathrm{x}\in \mathbb R^m,\qquad \mathrm{y}\in \mathbb R^n,\qquad m,n\ge 4.
\]
For \(0\le \ell\le u\le 4\), set
\[
\mathcal B^{\ell,u}_{m,n}(\mathrm{x},\mathrm{y})
=
\sum_{i=\ell}^{u}
\sigma_i(\mathrm{x})\sigma_{4-i}(\mathrm{y}).
\]
This is the basis generating polynomial of the rank four matroid whose bases
are the four-element subsets \(S\subseteq A\sqcup B\), with \(|A|=m\), \(|B|=n\),
satisfying
\[
\ell\le |S\cap A|\le u.
\]
The basis exchange axiom follows immediately from the condition on
\(|S\cap A|\). Moreover,
\[
\mathcal B^{\ell,u}_{m,n}
=
\pol_{(m,n)}(F^{\ell,u}_{m,n}),
\qquad
F^{\ell,u}_{m,n}(x,y)
=
\sum_{i=\ell}^{u}
\binom mi\binom n{4-i}x^iy^{4-i}.
\]

\begin{thm}\label{thm:rank4-list}
Assume \(m,n\ge 4\). For \(0\le \ell\le u\le 4\), let
\[
\mathcal B^{\ell,u}_{m,n}(\mathrm{x},\mathrm{y})
=
\sum_{i=\ell}^{u}
\sigma_i(\mathrm{x})\sigma_{4-i}(\mathrm{y}).
\]
Then the corresponding rank four two-symmetric matroid is \(B\)-G{\aa}rding
exactly for the following choices of \((\ell,u)\):
\[
\begin{array}{c|c}
(\ell,u) & \text{condition} \\
\hline
(0,0),(0,1),(0,4),(1,1),(1,2) & \text{always} \\
(2,2),(2,3),(3,3),(3,4),(4,4) & \text{always} \\
\hline
(0,2) & mn\le 5m+3n-9 \\
(2,4) & mn\le 3m+5n-9 \\
(0,3),(1,3),(1,4) & mn\le 5m+5n-13.
\end{array}\label{asd}
\]
\end{thm}

The proof of Theorem~\ref{thm:rank4-list} is contained in Appendix~\ref{bgarding}.

Wagner proved that every rank-three matroid is Rayleigh \cite{WagnerRank3}.
The rank-four case is substantially less understood. By Corollary \ref{cor:consequence}, Theorem~\ref{thm:rank4-list} 
provides explicit rank-four Rayleigh families. For example, for $x\in \R^m,$ $\y\in \R^n$, 
\[
\mathcal B^{1,3}_{m,n}(\x,\y)
=
\sigma_1(\mathrm{x})\sigma_3(\mathrm{y})
+
\sigma_2(\mathrm{x})\sigma_2(\mathrm{y})
+
\sigma_3(\mathrm{x})\sigma_1(\mathrm{y})
\]
is Rayleigh whenever
\[
mn\le 5m+5n-13.
\]
This strengthens the known Lorentzian property for these matroid basis
generating polynomials. It is also distinct from stability: for the same
family, stability is equivalent to
\[
7mn\le 23m+23n-55.
\]
Thus, \(\mathcal B^{1,3}_{5,n}\) is \(B\)-G{\aa}rding, hence Rayleigh, for all
\(n\ge 4\), but is stable only for \(n\le 5\). The same recursive MRS
method works beyond this example, including to higher ranks and other
block-symmetric families.

\subsection{Fano matroid}

The Fano matroid $F_{7}$ is a rank-$3$ matroid on $7$ elements with seven nontrivial flats
\[
\mathcal{L}=\{\{1,2,3\},\{1,4,5\},\{1,6,7\},\{2,4,6\},\{2,5,7\},\{3,4,7\},\{3,5,6\}\}.
\]
Its BSGF, SSGF and CSGF are given below
\[
    B_{F_7}(\mathbf w)
= \sigma_3(\mathbf w)
- \sum_{L\in\mathcal L} w_L,
\quad S_{F_7}(\mathbf w)
= \sum_{k=3}^{7} \sigma_k(\mathbf w)
- \sum_{L\in\mathcal L} w_L,
\]
\[
   C_{F_7}(\mathbf w)
= \sigma_7(\mathbf v)
+ 2\sigma_2(\mathbf v)
- 6\sigma_1(\mathbf v)
+ 13 - \sum_{L\in\mathcal L} v_L
- \sum_{C\in\mathcal C_4} v_C,  
\]
where $v_i=w_i+1$, and $\mathcal{C}_4 = \{\, [7]\setminus L :\  L \in \mathcal{L} \}.$

Erickson, in his master thesis \cite{EricksonNegativeCorrelation}, proved that rank 3 matroids are $I$-Rayleigh, which implies that $F_7$ is $I$-Rayleigh. We show that $F_7$ is $C$-\gar{} and hence recover the $I$-Rayleigh property.

\begin{thm}\label{F7csgf} 
The cospanning-set generating function of $F_7$ is Gårding.
\end{thm}
See Appendix C for a proof of Theorem~\ref{F7csgf}.

On the other hand, we have the following result.
\begin{prop}
\label{prop:F7-rayleigh-not-garding}
For the Fano matroid $F_{7}$, its basis generating function
$B_{F_{7}}$ and spanning-set generating function $S_{F_{7}}$ are Rayleigh,
but not G{\aa}rding.
\end{prop}

\begin{proof}

Let $\w_{0}=(0,1,-1,1,-1,-1,1)$ and $\mathbf{1}=(1,\dots,1)$.

For the basis-set generating function,
\[
h(t):=B_{F_{7}}(\w_{0}+t\mathbf{1})=28t^{3}-12t+4.
\]
Its derivatives are
\[
h'(t)=84t^{2}-12, \qquad h''(t)=168t.
\]
A direct computation shows that it violates the monotone root sequence (MRS) property.
Hence $B_{F_{7}}$ is not G{\aa}rding.

Similarly,
\[
g(t):=S_{F_{7}}(\w_{0}+t\mathbf{1})
= t^{7}+7t^{6}+18t^{5}+20t^{4}+t^{3}-21t^{2}-4t+6,
\]
and a direct computation of the largest real roots of $g^{(k)}(t)$ shows
that its root sequence is not monotonic nonincreasing. Thus, $S_{F_{7}}$ is not G{\aa}rding.

\medskip

To establish the Rayleigh property, it suffices to verify one representative pair, say $(6,7)$. Let
\[
\hat{E} := \prod_{i=2}^{5}(1+w_i).
\]
A direct expansion gives
\[
\Delta_{67}(S_{F_{7}})
= C_2\,w_1^2 + C_1\,w_1 + C_0,
\]
where
\[
C_2 = \hat{E}\bigl(w_2+w_3+w_4+w_5+w_2w_3+w_4w_5\bigr),\qquad
C_1 = 2\hat{E}(w_2w_3+w_4w_5),
\]
and
\[
\begin{aligned}
C_0
&= (w_2w_3)^2(1+w_4)(1+w_5)
+ (w_4w_5)^2(1+w_2)(1+w_3)
- 2\,w_2w_3w_4w_5 \\
&= (w_2w_3 - w_4w_5)^2 + (w_2w_3)^2\bigl(w_4 + w_5 + w_4w_5\bigr)
+ (w_4w_5)^2\bigl(w_2 + w_3 + w_2w_3\bigr).
\end{aligned}
\]
For $w_e\ge 0$, we have   $\Delta_{67}(S_{F_{7}})\ge 0$.
All other pairs $(i,j)$ follow by symmetry, so $S_{F_{7}}$ is Rayleigh.
Since $B_{F_{7}}$ is the lowest-degree homogeneous component of $S_{F_{7}}$,
and the Rayleigh property is preserved by taking bottom homogeneous components,
it follows that $B_{F_{7}}$ is also Rayleigh.

\medskip

Therefore, $B_{F_{7}}$ and $S_{F_{7}}$ are Rayleigh but not G{\aa}rding.
\end{proof}

We summarize the results in the following theorem.

\begin{thm}\label{F7summary}
Let $F_7$ be the Fano matroid and $F_7^*$ its dual. Then:
\begin{itemize}
\item $F_7$ is $C$-G{\aa}rding.
\item $F_7^*$ is $S$- and $B$-G{\aa}rding.
\item Both $F_7$ and $F_7^*$ satisfy the properties $C$-, $S$-, $B$-, and $I$-Rayleigh.
\end{itemize}
\end{thm}

\begin{rem}
$F_7$ exhibits a subtle distinction between the Rayleigh and G{\aa}rding properties.
\end{rem}

\begin{cor}
\label{thm:Fano}$F_{7}$ is an excluded minor of S-\gar{} matroids.
$F_{7}^{*}$ is an excluded minor of C-Gårding matroids.
\end{cor}

\begin{example}
The binary matroid $S_{8}$ is not negatively correlated \cite[pp.~495]{seymour_combinatorial_1975}, and therefore is neither Rayleigh nor $S/C$-\gar{} by Lemma~\ref{lem:SCG-implies-Rayleigh}.
In fact, $S_{8}$ contains $F_{7}$ and its dual as minors. Similarly, the matroid $AG(3,2)$, which is self dual and has both $F_{7}$ and $F_7^*$
as minors. By Corollary \ref{thm:Fano}, $AG(3,2)$ is neither $C$-Gårding
nor $S$-Gårding.
\end{example}

\subsection{Conjectures and Remarks}
Wagner \cite{MR2428906} raised several conjectures on $X$-Rayleigh
matroids including the following.
\begin{conjecture}[Wagner]
Any regular matroid is S- and I-Rayleigh. \label{conj:Any-regular-matroid}
\end{conjecture}

Notice that regular matroids include graphic and cographic matroids.
A positive answer to Conjecture \ref{conj:Any-regular-matroid} would
confirm a conjecture of Kahn \cite{kahnNormalLawMatchings2000} and
Grimmett-Winkler \cite{grimmettNegativeAssociationUniform2004}, which
asserts that the (unrooted) spanning-forest distribution on a graph
is negatively correlated. 
\begin{conjecture}[Wagner]
\label{conj:KGWconj}For any graph $G=(V,E)$,  the SFGF of $G$,  $F_{G}(\mathbf{w})$,
is Rayleigh.
\end{conjecture}

It is instructive to contrast $F_{G}(\w)$ with rooted spanning-forest
generating functions. The latter admits a determinantal representation
via the all-minors Matrix--Tree Theorem \cite{ChaikenMatrixTree1982},
and is therefore real-stable, hence strongly Rayleigh. By contrast,
the generating function of unrooted spanning forests generally fails
to be stable. Lemma \ref{lem:SCG-implies-Rayleigh} indicates that,
should the CSGF of a graphic matroid be Rayleigh, this Rayleigh behavior
propagates naturally to the unrooted spanning forest generating function
via inversion.

It is known that the basis generating function $B_{M}$ of a matroid
$M$ is a Lorentzian polynomial and plays a central role in the Lorentzian
approach to matroid theory developed by Brändén-Huh \cite{BrandenHuh20}.
The Lorentzian property successfully captures Hodge-theoretic and
log-concavity phenomena associated with bases. However, Lorentzian
polynomials are in general not Rayleigh: they satisfy only the weaker
c-Rayleigh inequalities \cite[Sec 4.5]{BrandenHuh20}.  In recent work with S.wu Wang~\cite{Fang-Ma-Wang}, we investigate the Gårding property for matroids on up to eight elements. 

In addition,
Lorentzian theory does not account for the Rayleigh behavior of non-homogeneous
generating functions such as SSGFs and CSGFs considered here. 
The c-Rayleigh inequalities were also studied in \cite{HuhCorrelationBoundsFields2021a}, where Huh--Schr\"oter--Wang gave an interpretation of the basis generating function in terms of toric intersection numbers, leading to various correlation bounds for matroids.

We conclude with several open problems suggested by the results of this
section, which will be investigated in future work.

A fundamental objective is to obtain a combinatorial characterization of
\(C\)- and \(S\)-\(\gar{}\) matroids.

Motivated by Conjecture~\ref{conj:Any-regular-matroid} and
Conjecture~\ref{conj:KGWconj}, we seek to identify regular and graphic
matroids that satisfy \(\gar{}\) properties. Since graphic matroids are
not \(S\)-\(\gar{}\) (see Example~\ref{exam:K5}), we are led to the
following question:

\begin{problem}\label{first}
Are graphic matroids \(C\)-\(\gar{}\)?
\end{problem}

An affirmative answer to Problem~\ref{first} would imply that all graphic
matroids are \(C\)-\(\gar{}\). In particular, the unrooted spanning forest
generating function of any graph would be Rayleigh. Consequently,
Conjecture~\ref{conj:KGWconj} would follow.

On the other hand, there exist generating functions in matroid theory
that are Rayleigh but fail to be Gårding. The precise boundary between
these two notions remains unclear and warrants further investigation.

%Another problem is related to the \gar{} property of the diagonal specialization of generating functions. 
%\begin{problem}
%Characterize the class of matroids with $S_M(t,t,\cdots,t)\in\GS[t]$.
%\end{problem}
%We speculate that any matroid belongs to this class.  In fact, with computer, we tested all matroids with 12 elements or less (1003638 in total) and confirm that all the diagonal specializations of SSGF and CSGF are \gar{}.  

\section{Ultra log-concave sequences and \gar{}-Rayleigh measures }\label{sec:ULCpolynomials-and}
Log-concavity is one of the most persistent regularity phenomena in
combinatorics and probability. Classically, the logarithmic
concavity of the coefficient sequences can be derived through Newton's inequalities 
and hence ultimately from real-rootedness. More recently, Hodge-theoretic
methods and Lorentzian polynomials have achieved great success in the
study of log-concave sequences in combinatorics \cites{adiprasitoHodgeTheoryCombinatorial2018}{HuhCorrelationBoundsFields2021a}{BrandenHuh20}{ALOVi}{ALOVii}{ALOViii}.

In this section, we show that the Gårding framework provides a further
geometric extension of log-concavity. Univariate  specializations of Gårding polynomials along positive affine
directions yield ultra log-concave sequences, thereby identifying ultra
log-concavity as a consequence of the Gårding property rather than stability.

We first recall the relevant notions 
\begin{defn}
Let $\{a_{k}\}$ be a finite or infinite sequence of nonnegative real
numbers. $\{a_{k}\}$ is said to \emph{have no internal zeros} if
for any $r_{1}<r_{2}<r_{3}$ with $a_{r_{1}},a_{r_{3}}$ positive,
then $a_{r_{2}}>0$. $\{a_{k}\}$ is called \emph{log-concave} if
$\{a_{k}\}$ has no internal zeros and satisfies 
\[
a_{k}^{2}\geq a_{k+1}a_{k-1}.
\]
A finite sequence $\{a_{k}\}_{k=0}^{d}$ is called \emph{ultra log-concave
(ULC)} if $\{a_{k}/\binom{d}{k}\}_{k=0}^{d}$ is log-concave. 
\end{defn}

Suppose that $\{a_{k}\}_{k=0}^{d}$ is a sequence of nonnegative real
numbers and consider 
\[
f(x)=\sum_{k=0}^{d}a_{k}x^{k}\in\mathbb{R}[x].
\]
If $f(x)$ is real--rooted, then the classical Newton inequalities
imply that $\{a_{k}\}_{k=0}^{d}$ is ultra log--concave, which immediately
implies log--concavity. However, ultra log-concavity alone does not
imply real-rootedness. On the other hand, define the polarization 
\begin{equation}
p(\x)=\pol f:=\sum_{k=0}^{d}a_{k}\frac{\sigma_{k}(\x)}{\binom{d}{k}}.\label{eq:Pemantlemeasure}
\end{equation}
Pemantle \cite[Theorem 2.7]{Pemantle2000} observed that $p(\x)$
in the form of (\ref{eq:Pemantlemeasure}) is Rayleigh if and only
if $\{a_{k}\}_{k=0}^{d}$ forms a ULC sequence. 

If $f(x)\in\GS_{+}[x]$, then by Theorem \ref{thm:polarization_keeps_garding},
$p(\x)\in\G_{+}[\x]$. See also \cite[Theorem 1.2]{Lin23JFA}. Hence,
by Lemma \ref{lem:GardingToRayleigh}, $p(\x)$ is Rayleigh.
Pemantle's result indicates that $\{a_{k}\}_{k=0}^{d}$ is ULC.
Since strictly positive affine transforms preserve the Gårding property,
we have the following theorem. 
\begin{thm}
\label{thm:ULCseq}Suppose that $g(\x)\in\GS_{n}$. Let $\a\in\Gamma_{n}^{+},\mathbf{b}\in\overline{\C_{g}}$. Then  $g(\mathbf{a}t+\mathbf{b})=\sum_{i=0}^da_kt^k$ has a ULC coefficient sequence. If $\sum_{k=0}^{d}a_{k}t^{k}\in\GS_{+}[t]$, then $\{a_{k}\}_{k=0}^{d}$
is ULC. 
\end{thm}

The converse of Theorem~\ref{thm:ULCseq} is not true. For instance,
let
\[
f(x)=x^{3}+9x^{2}+24x+21.
\]
The coefficient sequence of $f$ is ULC, but $f$ does not have a monotone
root sequence.

For a matroid $M$ and $X\in\{I,S,C\}$, the coefficients 
$a_{k}$ of the univariate generating function
\[
X_{M}(t,t,\cdots,t)=\sum_{k=0}^{n}a_{k}t^{k}
\]
counts the number of independent sets, spanning sets, and cospanning
sets of size $k$ for $X\in\{I,S,C\}$, respectively. Moreover, $a_{k}$
forms a ULC sequence if $M$ is $X$-\gar{}. 

In probability theory, Rayleigh and ULC property
are interpreted as negative dependence conditions for probability
measures on Boolean lattices \cite[Theorem 2.7]{Pemantle2000}. 

Denote $2^{[n]}$ the collection of subsets of $[n]$, a set of $n$
elements. Let $\mu$ be a probability measure defined on $2^{[n]}$.
The generating polynomial $g_{\mu}(\x)$ of measure $\mu$ is defined
to be
\begin{equation}
g_{\mu}(\x)=\sum_{S\subset[n]}\mu(S)\x^{S}.\label{eq:generatingmeasure}
\end{equation}
$\mu$ is called \emph{Rayleigh}
if the generating polynomial is Rayleigh, and\emph{ strong Rayleigh
(SR)} if $g_{\mu}$ is real stable. $\mu$ is called \emph{ultra log-concave
}if the diagonal restriction
\[
g_{\mu}(t,t,\cdots,t)=\sum_{k=0}^{n}a_{k}t^{k}
\]
has ultra log-concave coefficients. 

In their seminal work \cite{BBL09}, Borcea-Brändén-Liggett established 
various negative dependence
properties for strong Rayleigh measures including ultra log-concavity.  In contrast, general Rayleigh
measures are not ULC. They \cite[pp. 531]{BBL09} raised the question
of finding new natural sub-families of Rayleigh measures that are
ULC. We identify one such family, which we call the Gårding-Rayleigh measures. 
\begin{defn}
\label{def:G-Rmeasure} A measure $\mu$ is called a \emph{\gar{}-Rayleigh
measure (GR)} if $g_{\mu}(\x)\in\G_{+}[\x]$. If in addition, $g_{\mu}(\x)$ 
is homogeneous, then $\mu$ is called a homogeneous \gar{}-Rayleigh
measure. $\mu$ is called a \emph{projected homogeneous \gar{}-Rayleigh
measure} \emph{(PHGR)} if $\mu$ a projection to the first $n$ variables
of a homogeneous \gar{}-Rayleigh measure on $2^{[m]}$ with
$m\geq n$.
\end{defn}

By Theorem \ref{thm:ULCseq}, we obtain a partial answer
to the question raised in \cite[pp. 531]{BBL09}
\begin{thm}
\label{thm:ULCmeasure}Any \gar{}-Rayleigh measure is ULC. Any projected
homogeneous \gar{}-Rayleigh measure is negatively associated (NA). 
\end{thm}

See \cite{BBL09,Pemantle2000} for precise definitions of the negative
association property. The NA property is proved by Feder-Mihail \cite{Feder1992BalancedM}
for projected homogeneous Rayleigh measures and holds \emph{a fortiori}
for projected homogeneous \gar{}-Rayleigh measures.

\gar{}-Rayleigh measures are  preserved under the \emph{symmetric
exclusion process} (SEP), which is a continuous Markov process. See
\cite[Section 4.4]{BBL09} and \cite[Section 7]{Wagner11} for further
discussion on this topic. The proof uses Example \ref{exa:partialsymme} and
the argument in \cite[Section 7]{Wagner11}. 

Table \ref{tab:Measures-that-ULC,} summarizes whether the aforementioned
measures are always ULC, NA, or preserved under SEP.
\begin{table}
\begin{tabular}{|c|c|c|c|c|}
\hline 
 & R & SR & GR & PHGR\tabularnewline
\hline 
\hline 
ULC & No & Yes & Yes & Yes\tabularnewline
\hline 
NA & No & Yes & ? & Yes\tabularnewline
\hline 
SEP & No & Yes & Yes & Yes\tabularnewline
\hline 
\end{tabular}

\medskip{}

\caption{Measures that are ULC, NA, or preserved under SEP}\label{tab:Measures-that-ULC,}
\end{table}
 It is natural to ask whether \gar{}-Rayleigh measures satisfy other
negative dependence conditions proposed in \cite{Pemantle2000}. We
will leave the discussion for later works.

In matroid theory, Mason's conjecture \cite{Mason72} asserts that, for any matroid, the
numbers of independent sets of size $k$, denoted $I_k(M)$, form an
ultra log-concave sequence. Using Lorentzian polynomials, Br\"and\'en-Huh
\cite{BrandenHuh20} and Anari-Liu-OveisGharan-Vinzant \cite{ALOViii}
proved Mason's conjecture. From our perspective, if the diagonal CSGF
of a matroid $M$ is G{\aa}rding, then Mason's conjecture follows.

Matroids such as $F_7^*, S_8$, and $AG(3,2)$ are not $C$-G\aa rding.
However, for all such examples we have investigated, the diagonal
specializations of their spanning or cospanning-set generating functions are still
G\aa rding.
Since diagonal generating functions arise naturally in enumeration
problems, this suggests that the G\aa rding condition may continue to play
a significant role in these cases.

Based on strong numerical evidence, we pose the following analytic version
of Mason's problem.

\begin{problem}
Does the diagonal cospanning-set generating function of any matroid
satisfy the monotone root sequence property?
Equivalently, are these diagonal specializations G\aa rding in general?
\end{problem}

\section{Characteristic polynomials of Z-matrices and M-matrices}\label{sec:-matrix,--matrix,-and}

In this section, we investigate examples of Gårding polynomials in
the context of matrix theory, with emphasis on $Z$- and $M$-matrices. 

Let $I_n$ denote the $n \times n$ identity matrix. An $n \times n$ matrix $A$ is called a $Z$-matrix if all its off-diagonal entries are nonpositive; that is, $A = (a_{ij})$ with $a_{ij} \le 0$ for all $i \ne j$. For
a $Z$-matrix $A$, the following conditions are equivalent (see \cite[Chapter 6]{BP79Nonnegative}):
\begin{enumerate}
\item \label{enu:All-principal-minors}All principal minors of $A$ are
nonnegative.
\item \label{enu:The-matrix-}For any $\epsilon>0$, the matrix $A+\epsilon I_{n}$
is non-singular.
\item \label{enu:-where-}$A=sI_{n}-B$ where $B=(b_{ij})$ with $b_{ij}\geq0$,
and $s\geq\rho(B)$ is a nonnegative real number. Here, $\rho(B)$
denotes the spectral radius of $B$. 
\end{enumerate}
If any of the above conditions holds and hence all of the above conditions hold, then $A$ is called
a $M$-\emph{matrix}. 

%By the Perron-Frobenius theorem, when $B$ has nonnegative entries, its spectral radius $\rho(B)$ is an eigenvalue of $B$.

Given an $n\times n$ matrix $A$, let $X=\text{diag}(\x)=\text{diag}(x_{1},\cdots,x_{n})$.
The multivariate characteristic polynomial of $A$ is defined by 
\begin{equation}
q_{A}(x_{1},\cdots,x_{n}):=\det(X+A).\label{eq:charforM-matrix}
\end{equation}
Holtz \cite{Holtz05} proved that the coefficients of the univariate
characteristic polynomial of an $M$-matrix (or inverse $M$-matrix)
$A$ form an ultra log-concave sequence. Brändén-Huh \cite{BrandenHuh20}
later showed that the homogeneous multivariate characteristic polynomial
of $A$ is Lorentzian, thus recovering Holtz's result. By Theorem
\ref{thm:ULCseq}, we obtain an alternative proof of Holtz's inequality
by showing that $q_{A}\in\G_{+}$. 
\begin{thm}
\label{thm:MmatrixChar}If $A$ is a $Z$-matrix, then $q_{A}\in\G_{n}$.
If $A$ is an $M$-matrix, then $q_{A}\in\G_{+}$.
\end{thm}

\begin{proof}
Any $Z$-matrix $A'$ can be written in the form $A'=A-sI$ for some
real number $s$ and some $M$-matrix $A$. Then 
\[
q_{A'}(x_{1},\cdots,x_{n})=q_{A}(x_{1}-s,\cdots,x_{n}-s).
\]
Therefore,  it suffices to prove the second statement for an $M$-matrix $A$.  Denote 
\[
\mathcal{C}=\{\mathbf{x}\in\R^{n}:A(\x):=A+X\ \text{is a non-singular }M\text{-matrix}\}.
\]
Then $\C$ is an open subset of $\R^n$. 
% We claim that $\mathcal{C}$ is open in $\mathbb{R}^{n}$. For any
% $\mathbf{x}\in\mathcal{C}$, we may write 
% \begin{equation}
% A+X=s(\mathbf{x})I_{n}-B(\mathbf{x}),\label{eq:M-matrix}
% \end{equation}
% where $B(\mathbf{x})$ has nonnegative entries and $s(\mathbf{x})>\rho(B(\mathbf{x}))$.
% Choose $0<\epsilon<\frac{1}{3}(s(\mathbf{x})-\rho(B(\mathbf{x}))$,
% and let $\mathbf{y}\in\mathbb{R}^{n}$ satisfy $|\mathbf{y}|<\epsilon$.
% Writing $Y=\text{diag}(y_{1},\cdots,y_{n})$, we obtain 
% \[
% A+X+Y=(s(\mathbf{x})+|\mathbf{y}|)I_{n}-(B(\mathbf{x})+|\mathbf{y}|I_{n}-Y).
% \]
% Since the spectral radius of $B(\mathbf{x})+|\mathbf{y}|I_{n}-Y$
% is at most $\rho(B(\mathbf{x}))+2\epsilon$, it follows that $A+X+Y$
% is again a non-singular $M$-matrix. 
By the equivalent condition (\ref{enu:The-matrix-}) for an $M$-matrix,
a boundary point of $\mathcal{C}$ can only be a singular $M$-matrix.
Hence $q_{A}(\x)=0$ on $\partial\C$, and $\mathcal{C}\subset\{q_{A}>0\}$.

Finally, since $A(\mathbf{x})$ is a $M$-matrix for $\mathbf{x}\in\mathcal{C}$,
all principal minors are positive.  $\pdv_{i}q_{A}(\mathbf{x})$ is the principal minor
of $A(\mathbf{x})$ obtained by removing the $i^{\text{th}}$-row
and column. Therefore, 
$\pdv_{i}q_{A}(\mathbf{x})>0$ which implies that $\mathcal{C}$ passes PRT. Consequently, $q_{A}\in\G$.
Since $A\in\mathcal{C}$, we conclude that $q_{A}\in\G_{+}$. 
\end{proof}
\begin{rem}
 Brändén-Huh \cite{BrandenHuh20} showed that 
\[
p_{A}(x_{0},\cdots,x_{n}):=\det\left(x_{0}I_{n}+XA\right)=x_{0}^{n}x_{1}\cdots x_{n}\cdot\text{H}q_{A}(\frac{1}{x_{0}},\cdots,\frac{1}{x_{n}})
\]
is a Lorentzian polynomial. In fact, $\text{H}q_{A}$ is also Lorentzian.
The proof will be given in a separate paper.
\end{rem}

On the other hand, $p_{A}$ and $p_{A}(1,\x)$ are, in general, not
\gar{} polynomials.
\begin{example}
Let 
\[
A=\begin{pmatrix}2 & -1 & 0 & -1\\
0 & 1 & 0 & -1\\
0 & -1 & 1 & 0\\
-1 & 0 & 0 & 1
\end{pmatrix}.
\]
$A$ is a diagonal dominating $Z$-matrix and hence a $M$-matrix
(\cite[Chapter 6, M35]{BP79Nonnegative}). Take
\[
g(t)=p_{A}(1,t,2t,t,t)=7t^{3}+12t^{2}+6t+1.
\]
Since the root sequence $(-1,\frac{-4+\sqrt{2}}{7},-\frac{4}{7})$ is not a MRS,  $p_{A}(1,\x)$ and $p_{A}$ are not \gar{}
by Remark \ref{rem:MRSnecessary}. 
\end{example}

A non-singular matrix $A$ is called an\emph{ inverse $M$-matrix}
if $A^{-1}$ is an $M$-matrix. The following is an immediate corollary
of Theorem \ref{thm:MmatrixChar}.
\begin{cor}
\label{cor:inverseM}If $A$ is an inverse $M$-matrix, then $p_{A}(1,\mathbf{x})$
is a \gar{} polynomial. 
\end{cor}

Inverse $M$-matrices provide natural examples of \gar{}-Rayleigh measures. 
\begin{example}
Let $A$ be an $n\times n$ inverse $M$-matrix with spectral radius  $\rho(A)<1$. Then
there is a probability measure $\mu$ on $2^{[n]}$ such that 
for any $S\subset[n]$,
\[
\mu(\{T:T\supset S\})=\det(A_{S})
\]
where $A_{S}$ is the principal minor with rows and columns in $S$.
Then $\mu$ is a \gar{}-Rayleigh measure but not a strong Rayleigh measure
in general. In fact, from \cite[Proposition 3.5]{BBL09},
\[
g_{\mu}(\x)=\det((X-I_n)A+I_n).
\]
Then $g_\mu\in \G$ follows from Corollary \ref{cor:inverseM}. 
\end{example}

\appendix
\section{Proof of Theorem~\ref{thm:Uhat}}
In this appendix, we record the proof of Theorem~\ref{thm:Uhat}. Recall that  
\[
g_{r,q}(x,y):=S_{U_{r,n}}(\underbrace{x,\cdots,x}_{r},\underbrace{y,\cdots,y}_{q}),\qquad f_{r,q}(x,y):=S_{\hat{U}_{r,n}}(\underbrace{x,\cdots,x}_{r},\underbrace{y,\cdots,y}_{q}).
\]
By Theorem~\ref{thm:polarization_keeps_garding}, it suffices to prove $f_{r,q}(x,y)\in\GS_{+}(x,y)$ to establish Theorem~\ref{thm:Uhat}.

We prepare the proof with the following deformation. For $t\in[0,1]$, define
\[
f_{r,q,t}(x,y):=g_{r,q}(x,y)-t\,x^{r}.
\]
Thus, $f_{r,q,1}=f_{r,q}$ and $f_{r,q,0}=g_{r,q}$.  For $t\in[0,1]$, the polynomials $f_{r,q,t}$ have nonnegative coefficients.

We will often suppress the dependence on $r,q,y$ when these parameters are fixed. A direct computation shows that for all $r,q\geq 1$,
\begin{equation}
\frac{\partial}{\partial x}f_{r,q,t}(x,y)=r\,f_{r-1,q,t}(x,y).
\label{eq:april1}
\end{equation}
We also define
\[
h_{r,q}(x,y):=\frac{\partial}{\partial y}f_{r,q}(x,y)=\frac{\partial}{\partial y}g_{r,q}(x,y).
\]

Recall that for a univariate polynomial $f(x)\in \R[x]$,  $r(f)$ denotes the largest real root of $f$. For each fixed $q\geq1,r\geq1$, $y>0$ and $t\in(0,1]$,  we view each $f_{r,q,t}(x,y),\ g_{r,q}(x,y),\ h_{r,q}(x,y)$ as polynomial in $x$ and denote 
\[\alpha_{r,q,t}(y)=r(f_{r,q,t}(x,y)),\ \beta_{r,q}(y)=r(g_{r,q}(x,y)),\ \gamma_{r,q}(y)=r(h_{r,q}(x,y)).\]

By Example~\ref{exa:uniform}, $g_{r,q},h_{r,q}\in\GS[x,y]$. We collect the following facts based on their G{\aa}rding properties.
\begin{lem}
\label{lem:known}
Fix $q\geq 1,r\geq1$ and $y>0$. Then, $\beta_{r,q}(y)$ and $\gamma_{r,q}(y)$ exist. Moreover,  
\begin{enumerate}
\item If $r>q$, then $\beta_{r,q}(y)=0$.
\item If $r\leq q$, then $\beta_{r,q}(y)<0$.
\item If $r\leq q$ and $y>0$, then $\beta_{r,q}(y)<\beta_{r+1,q}(y)$.
\item $\partial_{y}h_{r,q}(\gamma_{r,q}(y),y)\leq 0$, and $\gamma_{r,q}(y)$ is nonincreasing in $y$.
\end{enumerate}
\end{lem}
Next, we establish the following lemma.

\begin{lem}
\label{lem:april5}
For fixed $q\geq1,r\geq1$, $y>0$ and $t\in(0,1]$, $\alpha_{r,q,t}(y)$ is a well-defined non-positive number.
\end{lem}
\begin{proof}
Since $f$ has nonnegative coefficients, $f_{r,q,t}(\cdot,y)$ has no positive real roots, so any real root is non-positive.

If $q<r$, then $f_{r,q,t}(0,y)=0$, and hence its largest real root $\alpha_{r,q,t}(y)=0$.

Now assume $q\geq r$. If $r$ is odd, then $f_{r,q,t}(\cdot,y)$ has an odd degree with positive leading coefficient, so its largest real root exists and is at most $0$. If $r$ is even, we compare with the uniform case. 
\begin{equation}
f_{r,q,t}(\beta_{r,q}(y),y)=g_{r,q}(\beta_{r,q}(y),y)-t\,\beta_{r,q}(y)^{r}=-t\,\beta_{r,q}(y)^{r}<0.
\label{eq:compare}
\end{equation}
On the other hand, $f_{r,q,t}(0,y)=g_{r,q}(0,y)\geq 0$. Thus, $f_{r,q,t}(\cdot,y)$ has a real root in $(\beta_{r,q}(y),0]$, so its largest real root exists and is nonpositive.
\end{proof}

Furthermore, by (\ref{eq:compare}) and the parity, we have
\begin{cor}
notation as above. Fix $q\ge1$, $0<t\le1$, and $y>0$. Then:
\begin{enumerate}
\item If $r\le q$ is even, then $\alpha_{r,q,t}(y)\ge\beta_{r,q}(y).$
\item If $r\le q$ is odd, then $\alpha_{r,q,t}(y)\le\beta_{r,q}(y).$
\end{enumerate}
Furthermore, for $r$ even and $r\le q$, we have 
\begin{equation}
\alpha_{r-1,q,t}(y)\le\beta_{r-1,q}(y)<\beta_{r,q}(y)\le\alpha_{r,q,t}(y).\label{eq:april3}
\end{equation}
\end{cor}
We record a key estimate.

\begin{lem}
\label{lem:april10}
Fix $q\geq 1$, $t\in[0,1]$, and $y>0$, and use the notation above. If $r$ is even and $\alpha_{r+1}<\alpha_{r}$, then
\begin{equation}
\alpha_{r+1}\leq \alpha_{r-1}<\alpha_{r}.
\label{eq:jump}
\end{equation}
\end{lem}

\begin{proof}
By definition, $\alpha_{r+1}$ is the largest real root of $f_{r+1,q,t}$, which has nonnegative coefficients. Therefore,
\[
\frac{\partial}{\partial x}f_{r+1,q,t}(\alpha_{r+1},y)\geq 0,
\]
which, by \eqref{eq:april1}, implies
\begin{equation}
f_{r,q,t}(\alpha_{r+1},y)\geq 0.
\label{eq:april2}
\end{equation}

Now suppose, for contradiction, that $\alpha_{r+1}>\alpha_{r-1}$. By the definition of $\alpha_{r-1}$ and \eqref{eq:april1}, the function $x\mapsto f_{r,q,t}(x,y)$ is strictly increasing on $[\alpha_{r-1},\infty)$. Since $\alpha_{r}$ is the largest real root of $f_{r,q,t}(\cdot,y)$ and $\alpha_{r+1}<\alpha_{r}$, this strict monotonicity yields
\[
f_{r,q,t}(\alpha_{r+1},y)<0,
\]
contradicting \eqref{eq:april2}. Hence $\alpha_{r+1}\leq \alpha_{r-1}$. Combining this with \eqref{eq:april3}, we prove \eqref{eq:jump}.
\end{proof}

We now analyze the geometric setting. For fixed $r,q\geq 1$, let $C_{r,q}$ be the connected component of $\{(x,y)\in\mathbb{R}^{2}:f_{r,q,1}(x,y)>0\}$ that contains the open first quadrant. By Lemma~\ref{lem:april5},
\[
C_{r,q}=\{(x,y)\in\mathbb{R}^{2}: y>0,\ x>\alpha_{r,q,1}(y)\}.
\]

The following proposition establishes a monotonicity relation between the roots, which corresponds to the monotonicity of the regions $C_{r,q}$.

\begin{prop}
\label{prop:april6}
For every $t\in[0,1]$, $y>0$, and $r>1$, one has
\begin{equation}
\alpha_{r-1,q,t}(y)\leq \alpha_{r,q,t}(y).
\label{eq:april7}
\end{equation}
Equivalently, $C_{r,q}\subseteq C_{r-1,q}$.
\end{prop}

\begin{proof} When $r$ is even, the claim follows from (\ref{eq:april3}). We assume now $r$ is odd.
We use a continuity method argument.

\textbf{Step 1.} Equation~\eqref{eq:april7} holds at $t=0$. For $y=1$ and fixed $r,q$, the largest real roots $\alpha_{r,q,t}(1)$ depend continuously on $t$. Therefore, there exists $t_{0}\in(0,1)$ such that \eqref{eq:april7} holds for all $0\leq t<t_{0}$.

\textbf{Step 2.} Fix $t\in(0,t_{0})$. We claim that \eqref{eq:april7} holds for every $y>0$, not just for $y=1$.

Assume, in  contradiction, that this fails. Then there exist $y>0$ and an index $r$ such that
\[
\alpha_{r+1,q,t}(y)<\alpha_{r,q,t}(y).
\]
Let
\[
Y=\{\,y\neq 1:\ \alpha_{r+1,q,t}(y)<\alpha_{r,q,t}(y)\text{ for some }r\,\}.
\]
Choose $y_{0}\in Y$ such that $|y_{0}-1|=\inf\{|y-1|:\ y\in Y\}$. Then $y_{0}>0$. By the continuity of the largest real roots with respect to $y$, we have
\begin{equation}
\alpha_{r+1,q,t}(y_{0})\geq \alpha_{r,q,t}(y_{0}).
\label{eq:bigger}
\end{equation}
On the other hand, by Lemma~\ref{lem:april10}, any first failure of monotonicity forces
\[
\alpha_{r+1,q,t}(y_{0})\leq \alpha_{r-1,q,t}(y_{0})\le\alpha_{r,q,t}(y_{0}),
\]
which indicates $$\alpha_{r-1,q,t}(y_0)=\alpha_{r,q,t}(y_0)=\alpha_{r+1,q,t}(y_0),$$ 
which contradicts  (\ref{eq:april3}).
Hence, such $y$ does not  exist, and \eqref{eq:april7} holds for all $y>0$.

\textbf{Step 3.} Let
\[
D:=\{t\in [0,1]: \text{the conclusion of Step 2 holds for this } t\}.
\]
Run the same argument as in Step 2, $D$ is open. It is also closed, since the roots $\alpha_{r,q,t}(y)$ depend continuously on $(t,y)$.

Hence $D=[0,1]$, and  $1\in D$. This completes the proof.
\end{proof}

We now turn to vertical slice of the region. First, we have

\begin{lem}
    \label{lem:endpoint-verification}
Fix $r,q\ge 1$. Let $\alpha_{r,q}(y)$ and $h_{r,q}=\partial_y f_{r,q}$ be defined as before. 
Then
\[
h_{r,q}(\alpha_{r,q}(y),y)\ge0
\]
for all sufficiently small $y>0$ and all sufficiently large $y>0$.
\end{lem}
\begin{proof}
If $q<r$, then $f_{r,q}(0,y)=0$, so $\alpha_{r,q}(y)=0$ and
$h_{r,q}(\alpha_{r,q}(y),y)=0$.

\medskip
\noindent\textbf{Small $y$.}
Assume $q\ge r$ and set $x=yz$. Then
\[
f_{r,q}(yz,y)=y^r Q_{r,q}(z)+O(y^{r+1}),
\]
where
\[
Q_{r,q}(z)=\sum_{j=1}^r \binom qj \binom r{r-j} z^{r-j}.
\]
Let $\rho$ be the largest real root of $Q_{r,q}$. Then
\[
\alpha_{r,q}(y)=\rho y+O(y^2).
\]
Substituting into $h_{r,q}$ gives
\[
h_{r,q}(\alpha_{r,q}(y),y)
=
y^{r-1}\!\sum_{j=1}^r j\binom qj \binom r{r-j}\rho^{r-j}
+O(y^r).
\]
Since
\[
\sum_{j=1}^r j\binom qj \binom r{r-j} z^{r-j}
=
rQ_{r,q}(z)-zQ'_{r,q}(z),
\]
the leading term is $-\rho Q'_{r,q}(\rho)\ge0$, and hence,
$h_{r,q}(\alpha_{r,q}(y),y)\ge0$ for small $y$.

\medskip
\noindent\textbf{Large $y$.}
Write $x=u-1$. Then
\[
f_{r,q}(u-1,y)
=
u^r((1+y)^q-1)
-
\sum_{j=1}^{r-1}\binom qj y^j D_{r,j}(u),
\]
with $D_{r,r-1}(u)=1$. Writing $\alpha_{r,q}(y)=-1+u(y)$ and using
$f_{r,q}(\alpha_{r,q}(y),y)=0$, we obtain
\[
u(y)^r
=
\binom q{r-1}y^{r-1-q}
+O(y^{r-2-q}).
\]
Then
\[
h_{r,q}(\alpha_{r,q}(y),y)
=
q\binom q{r-1}y^{r-2}
-
(r-1)\binom q{r-1}y^{r-2}
+O(y^{r-3}),
\]
so
\[
h_{r,q}(\alpha_{r,q}(y),y)
=
(q-r+1)\binom q{r-1}y^{r-2}
+O(y^{r-3})\ge0
\]
for large $y$.
\end{proof}
\begin{prop}
\label{prop:april10}$\alpha_{r,q}(y)$ is a non-increasing function
of $y$. 
\end{prop}

\begin{proof}
Since $\alpha_{r,q}(y)$ is the largest real root of $f_{r,q}(x,y)$,
by the definition of $\gamma_{r,q}(y)$, and the fact that $h_{r,q}$
is Gårding,  to establish the result, it is sufficient to
prove that for $y>0$,
\begin{equation}
\partial_{y}f_{r,q}(\alpha_{r,q}(y),y)\geq0.\label{eq:star}
\end{equation}
Or, equivalently, $\gamma_{r,q}(y)\leq\alpha_{r,q}(y).$

Suppose, for the sake of contradiction, that (\ref{eq:star}) fails. Then there
exists $y_{0}$ such that 
\begin{equation}
\partial_{y}f_{r,q}(\alpha_{r,q}(y_{0}),y_{0})<0.\label{eq:<}
\end{equation}
By Lemma~\ref{lem:endpoint-verification}, for $y\to0^{+}$ and for $y\to\infty$,
(\ref{eq:star}) holds. Hence, by continuity, there exists $y_{1}<y_{0}<y_{2}$
such that 
\[
\partial_{y}f_{r,q}(\alpha_{r,q}(y_{i}),y_{i})=h_{r,q}(\alpha_{r,q}(y_{i}),y_{i})=0\quad(i=1,2),
\]
and 
\begin{equation}
\partial_{y}f_{r,q}(\alpha_{r,q}(y),y)\leq0\quad\text{for }y\in(y_{1},y_{2}).\label{eq:<=00003D}
\end{equation}
Thus,
\begin{equation}
\alpha_{r,q}(y_{i})=\gamma_{r,q}(y_{i})\quad(i=1,2),\label{eq:=00003D}
\end{equation}
By Lemma \ref{lem:known}, $\gamma_{r,q}(y)$ is strictly decreasing.
Therefore, $\gamma_{r,q}(y_{1})\geq\gamma_{r,q}(y_{2}).$ However,
by (\ref{eq:<}) and (\ref{eq:<=00003D}), $\alpha_{r,q}(y_{1})<\alpha_{r,q}(y_{2}),$
contradicting (\ref{eq:=00003D}). This proves the claim. 
\end{proof}
Finally, we establish the proof of Theorem \ref{thm:Uhat}.
\begin{proof}[Proof of Theorem \ref{thm:Uhat}]
 The Gårding property of $f_{r,q}$ follows by a direct induction on
$r$, Propositions~\ref{prop:april10}, \ref{prop:april6} and Definition
\ref{def:gard_recursive}. The Gårding property of $S_{\hat{U}}$
then follows from Theorem \ref{thm:polarization_keeps_garding}.
\end{proof}

\section{Proof of Theorem~\ref{thm:Any-matroid-on6elements}}\label{6elements}
In this appendix, we give the proof of Theorem~\ref{thm:Any-matroid-on6elements}.
By the classification of small matroids, and the reduction via minors
and 2-sums, it suffices to verify the Gårding property for SSGFs of
3-connected matroids on 6 elements. There are only $5$ such matroids.
See Figure \ref{fig:3-connnected-rank-3} for their geometric representations.
By Example \ref{exa:uniform}, uniform matroids are $S$-\gar{}.
Therefore, it  remains to check the remaining cases: $M(K_{4})$, $W^{3},$
$P_{6}$ and $Q_{6}$. 

Case 1. $M=M(K_{4})$. A direct computation shows that
\[
S_{M(K_{4})}=f_{K_{4}}(\mathbf{v})-g_{K_{4}}(\mathbf{v}),
\]
where 
\[
\begin{aligned}
f_{K_{4}}(\mathbf{v})&=v_{123}\bigl(v_{456}-1\bigr)-\bigl(v_{4}+v_{5}+v_{6}-3\bigr),\\
g_{K_{4}}(\mathbf{v})&=v_{1}\bigl(v_{56}+v_{4}-2\bigr)+v_{2}\bigl(v_{46}+v_{5}-2\bigr)+v_{3}\bigl(v_{45}+v_{6}-2\bigr)-3\bigl(\sum_{i=4}^{6}v_{i}-3\bigr).
\end{aligned}
\]
Since $v_{4}+v_{5}+v_{6}-3\vtl v_{456}-1$, $f_{K_{4}}(\mathbf{v})$
is \gar{}. For $\v\in\C_{f_{K_{4}}},$ $v_{i}>0$ for $i\in\{1,2,3\}$
and $v_{i}v_{j}+v_{k}-2>0$ for $\{i,j,k\}=\{4,5,6\}$. Applying the
AM-GM inequality yields 
\begin{equation}
g_{K_{4}}\geq3\left[v_{123}\bigl(v_{56}+v_{4}-2\bigr)\bigl(v_{46}+v_{5}-2\bigr)\bigl(v_{45}+v_{6}-2\bigr)\right]^{\frac{1}{3}}-3\bigl(v_{4}+v_{5}+v_{6}-3\bigr).\label{eq:new11}
\end{equation}
Observe that for any real numbers $x,y,z$,
\begin{equation}
(xy+z-2)(yz+x-2)(xz+y-2)-(xyz-1)(x+y+z-3)^{2}=(x-1)^{2}(y-1)^{2}(z-1)^{2}.\label{eq:algebrafact}
\end{equation}
(\ref{eq:new11}) and (\ref{eq:algebrafact}) imply that $g_{K_{4}}\geq0$
for $\v\in\C_{f}$, and $g_{K_{4}}\vtl f_{K_{4}}$. As a result, $S_{M(K_{4})}(\mathbf{v})$
is \gar{} by Theorem \ref{thm:dominategardnew}.

Case 2. $M=W^{3}$. We have the following decomposition:
\[
S_{W^{3}}=v_{2}f_{W^{3}}(\mathbf{v})-g_{W^{3}}(\mathbf{v}),
\]
where
\[
\begin{aligned}
f_{W^{3}}(\mathbf{v})&=v_{13}\bigl(v_{456}-1\bigr)-\bigl(v_{4}+v_{5}+v_{6}-3\bigr),\\
g_{W^{3}}(\mathbf{v})&=v_{1}\bigl(v_{56}+v_{4}-2\bigr)+v_{3}\bigl(v_{45}+v_{6}-2\bigr)+\bigl(v_{46}+v_{5}-2\bigr)-3\bigl(\sum_{i=4}^{6}v_{i}-3\bigr).
\end{aligned}
\]
Then the similar argument as in Case 1 shows that $f_{W^{3}}(\mathbf{v})\in\G$
and $g_{W^{3}}\vtl f_{W^{3}}.$ Therefore, $S_{W^{3}}$ is \gar{}
by Theorem \ref{thm:iffdominate}.

Case 3. $M=Q_{6}$. By Theorem \ref{thm:SSGF} , 
\[
\begin{aligned}S_{Q_{6}} & =v_{123456}-\bigl(v_{123}+v_{345}\bigr)\\
 & \quad-\bigl(v_{14}+v_{15}+v_{16}+v_{24}+v_{25}+v_{26}+v_{36}+v_{46}+v_{56}\bigr)\\
 & \quad+2v_{3}+3(v_{1}+v_{2}+v_{4}+v_{5})+4v_{6}-8.
\end{aligned}
\]
Define
\begin{align*}
F(x,y,z,u):= & S_{Q_{6}}|_{v_{1}=v_{2}=x,v_{3}=y,v_{4}=v_{5}=z,v_{6}=u}\\
= & f(x,y,z)u-g(x,y,z),
\end{align*}
where 
\[
f=(x^{2}z^{2}-1)y-2(x+z-2),\ g=y(x^{2}+z^{2}-2)+4xz-6(x+z)+8.
\]
We claim $F$ is \gar{}. A direct computation shows that $f(x,y,z)\in\GS[x,y,z]$
and $\pdv_{i}g\vtl\pdv_{i}f$ for $i\in\{x,y,z\}$. If $(x,y,z)\in\C_{f}$,
then $y>\frac{2(x+z-2)}{x^{2}z^{2}-1},x>0,z>0,xz>1$. 
\begin{equation}
g|_{\C_{f}}=\frac{2(x-1)^{2}(z-1)^{2}(2xz+x+z)}{x^{2}z^{2}-1}+\left(y-\frac{2(x+z-2)}{x^{2}z^{2}-1}\right)(x^{2}+z^{2}-2)>0,\label{eq:addd1}
\end{equation}
which leads to $g\vtl f$; and $F(x,y,z,u)\in\GS[x,y,z,u].$ Since
$S_{Q_{6}}(\mathbf{v})$ is a polarization of $F(x,y,z,u)$, it follows
from Theorem \ref{thm:polarization_keeps_garding} that $S_{Q_{6}}\in\G$.

Case 4. $M=P_{6}$. This follows from Theorem \ref{thm:Uhat}. 

We have finished the proof of the theorem.

\section{Proof of Theorem~\ref{thm:rank4-list}}\label{bgarding}

We start with a technical result.

\begin{lem}\label{lem:quartic-reduced}
Let \(F(x,y)\) be homogeneous of degree \(4\) with non-negative coefficients.
Set \(F_x=\pdv_xF\) and \(F_y=\pdv_yF\). Let 
\[
f_1(t)=F_x(t,1),\qquad
f_2(t)=F_y(1,t),
\]
and
\[
f_3(t)=F_y(t,1),\qquad
f_4(t)=F_x(1,t).
\]
Then $F$ is \gar{} if and only if \(f_1\) and \(f_2\)  are real stable and satisfy
\[
r(f_1)\leq  r(f_3),\qquad
r(f_2)\leq r(f_4).\]
\end{lem}

\begin{proof} "$\Leftarrow$" For simplicity, we denote $r_i=r(f_i)$ for $i=1,2,3,4$. Denote $\C$ the connected component of $\{f>0\}$ which contains the first quadrant.  If $f_1,f_2$ are real stable, then their homogenization $F_x,F_y$ are real stable. Thus, it suffices to prove that  $\C\subset \C_{F_x}\cap\C_{F_y}$. Since $F$ is homogeneous of degree 4, we have 
\[F(x,y)=xF_x+yF_y.\]
Since $r_3\leq0$, and $r_1\leq r_3$, we have $f_1(r_3)\geq0$ and \[F(r_3,1)=r_3f_1(r_3)\leq0.\]
Similarly, \[F(1,r_4)=r_4f_2(r_4)\leq0.\] Therefore, $\C\subset \{x-r_3y>0\}\cap\{y-r_4x>0\}.$ Since $r_1\leq r_3,r_2\leq r_4$, we have $$ \{(x,y):x-r_3y>0,y-r_4x>0\}=\C_{F_x}\cap\C_{F_y}$$ which implies $\C\subset \C_{F_x}\cap\C_{F_y}$.  

"$\Rightarrow$"  Suppose that $F$ is \gar{}. The partial derivatives \(F_x\) and \(F_y\) are homogeneous cubic \gar{} and hence real  stable by Proposition~\ref{H3}. Thus $f_1,f_2$ are real stable. Notice that 
\[F(r(f_1),1)=r(f_1)f_1(r(f_1))+f_3(r(f_1))=f_3(r(f_1)).\] Since $f_1\prec F(t,1)$, $f_3(r(f_1))\leq0.$ Hence, $r(f_1)\leq r(f_3)$. Similarly, we have $r(f_2)\leq r(f_4).$

% By Definition~\ref{def:gard_recursive}, it suffices to show that \(F_x,F_y\) are
% G{\aa}rding and that
% \[
% C_F\subseteq C_{F_x}\cap C_{F_y}.
% \]

% The partial derivatives \(F_x\) and \(F_y\) are homogeneous cubics. By Proposition~\ref{H3},
% homogeneous cubic G{\aa}rding polynomials are stable. Since
% \[
% F_x(t,1)=f_1'(t),\qquad F_y(1,t)=f_2'(t),
% \]
% and \(f_1,f_2\) satisfy MRS, both derivatives \(f_1'\) and \(f_2'\) are
% real-rooted. Hence \(F_x\) and \(F_y\) are stable, and therefore G{\aa}rding.

% For the cone inclusion, let \(H\) be homogeneous bivariate and set
% \[
% p(t)=H(t,1),\qquad q(t)=H(1,t).
% \]
% Then
% \[
% C_H\cap\{y=1\}=(r(p),\infty),\qquad
% C_H\cap\{x=1\}=(r(q),\infty).
% \]
% Applying this to \(F\), we obtain
% \[
% C_F\subseteq C_{F_y}
% \iff r(F(t,1))\ge r(F_y(t,1))
% \iff r(f_1)\ge r(f_3),
% \]
% and
% \[
% C_F\subseteq C_{F_x}
% \iff r(F(1,t))\ge r(F_x(1,t))
% \iff r(f_2)\ge r(f_4).
% \]
% The assumed inequalities yield
% \[
% C_F\subseteq C_{F_x}\cap C_{F_y},
% \]
% and hence \(F\) is G{\aa}rding.
\end{proof}

\begin{proof}[Proof of Theorem \ref{thm:rank4-list}]
The one- and two-term cases are immediate. The case \((0,4)\) is uniform.

For \((0,2)\),
\[
F(t,1)
=
\binom n4 + m\binom n3 t + \binom m2\binom n2 t^2,
\]
whose discriminant yields
\[
mn\le 5m+3n-9.
\]
By symmetry, \((2,4)\) gives \(mn\le 3m+5n-9\).

For \((1,3)\), writing \(F=Axy^3+Bx^2y^2+Cx^3y\), the condition in Lemma \ref{lem:quartic-reduced} reduces to
\[
B^2-3AC\ge 0
\iff
mn\le 5m+5n-13.
\]
The same condition governs \((0,3)\) and \((1,4)\).
\end{proof}

\begin{rem}
\label{rem:middle-stable}
For  $\mathcal B^{1,3}_{m,n}$, stability holds if and only if
\[
7mn \le 23m+23n-55.
\]
Thus the stability condition differs from the \(B\)-G{\aa}rding condition (\ref{asd}). For  every \(n \ge 6\), the corresponding matroid with \(m=5\)
is \(B\)-G{\aa}rding but not strong Rayleigh.
\end{rem}

\section{Proof of Theorem~\ref{F7csgf}}
\noindent\textbf{Acknowledgment.}
The results in this appendix were obtained in collaboration with Shouda Wang.

\begin{proof}[Proof of Theorem~\ref{F7csgf}]
We denote $f(\v)=C_{F_7}(\v-1)$. By the deletion-contraction formula \eqref{eq:CSGFDeletioncontraction} and Corollary \ref{thm:extension}, $\partial_i f\in\G$ for all $i\in[7]$. Let $\Omega=\partial \C_{\partial_7 f}$, and let $\Omega'=\{\v\in \Omega,\ \partial_{7i}f(\v)>0,\ i\in[6]\}$, which is a dense open set of $\Omega$.

For  $\v\in \Omega'$, $v_i\geq0$, $\partial_\alpha\partial_7 f>0$ for any nontrivial $\alpha\subset[6]$.
By Proposition~\ref{prop:testing method} and symmetry, it suffices to prove that for $\v\in\Omega'$, $f(\v)\leq 0$. 

We decompose the polynomial $f$ as a multi-affine function in the variables $v_6$ and $v_7$:
\[
f(v_1,\dots,v_7)
=
A\,v_6 v_7 + B\,v_6 + C\,v_7 + D,
\]
where the coefficients $A,B,C,D$ depend only on $(v_1,\dots,v_5)$ and are given by
\[
\begin{aligned}
A &= v_{12345} - v_{23} - v_{45} - v_1 + 2=f_{67}, \\[4pt]
B &= -\,v_{134} - v_{125} - v_{24} - v_{35}
+ 2\,\sigma_1(v_1,\dots,v_5) - 6, \\[4pt]
C &= -\,v_{135} - v_{124} - v_{25} - v_{34}
+ 2\,\sigma_1(v_1,\dots,v_5) - 6, \\[4pt]
D &= -\,v_{2345} - v_{123} - v_{145}
+ 2\,\sigma_2(v_1,\dots,v_5)
- 6\,\sigma_1(v_1,\dots,v_5) + 13.
\end{aligned}
\]

Denote the Rayleigh difference $$R_{6,7}=\partial_6 f  \pdv_7f-f\partial_{67}f=BC-AD.$$ For $\v\in \Omega'$, $f_{67}>0$, and $-f(\v)f_{67}=R_{6,7}.$ Thus, it suffices to prove $R_{6,7}\geq0$ when $\v\in\Omega'$.

We reduce the verification of the Gårding property to a quadratic inequality 
and analyze its discriminant. For convenience, we introduce the following variables:
\[
\alpha := v_1,\quad
\beta_1 := v_2 + v_3,\quad
\beta_2 := v_4 + v_5,\quad
\gamma_1 := v_2 v_3,\quad
\gamma_2 := v_4 v_5.
\]
Since  for $\v\in \Omega'$, $v_i > 0$, we have 
\begin{equation}
\label{newadd1}
\alpha > 0,\quad \beta_i > 0, \quad 
\gamma_i > 0,\quad 
\beta_i^2 \geq 4\gamma_i, \quad \alpha\gamma_i - 1>0. 
\end{equation}
The last inequality is because in $\Omega'$, $\partial_{2367} f>0$ and $\pdv_{4567}f > 0$. 

A long but direct computation shows that the Rayleigh difference
\[
R_{6,7} = P\,\beta_1^2 + Q\,\beta_1 + R,
\]
where
\[
\begin{aligned}
P &=
\alpha \beta_2^2 + (\alpha-1)^2 \gamma_2 - 2(\alpha+1)\beta_2 + 4,\\
Q &=
-2(\alpha+\beta_2-3)\big((\alpha+1)\beta_2 - 4 + A\big),\\
R &=
(\alpha-1)^2(\gamma_1 \beta_2^2 - 4\gamma_1\gamma_2)
+ (2(\alpha+\beta_2)-6)^2 \\
&- A\Big(
- \gamma_1\gamma_2
+ (2-\alpha)(\gamma_1+\gamma_2)
+ (2\alpha-6)\beta_2
- 6\alpha + 13
\Big).
\end{aligned}
\]

We claim that $P(\mathbf v) > 0$ for all $\mathbf v \in \Omega'$. 
Indeed, viewing $P$ as a quadratic polynomial in $\beta_2$, its discriminant is
\[
\Delta_P=
4(\alpha-1)^2(1-\alpha\gamma_2).
\]
By \eqref{newadd1}, $\Delta_P\leq0$. Thus,  
$P(\mathbf v) \geq 0$ in $\Omega'$. If $P=0$ then $\alpha=1,\beta_2=2$, which implies $\gamma_2\leq 1$ by \eqref{newadd1} and violates the  inequality $\alpha\gamma_2>1$.

Regarding $R_{6,7}$ as a quadratic function of $\beta_1$,  we compute the discriminant:
\[
\Delta := Q^2 - 4PR=
-4(\alpha-1)^2(\beta_2-\gamma_2-1)^2\,\Delta',
\]
where
\[
\Delta' =
(\alpha\gamma_1-1)(\beta_2^2-4\gamma_2)
+ (\alpha+\gamma_1+2)A.
\]
By \eqref{newadd1}, $\Delta'\geq 0$. It follows that $\Delta \le 0$ on $\Omega'$. Since $P>0$, we conclude that
\[
R_{6,7} \ge 0 \quad \text{on } \Omega'.
\]
This completes the analytic verification of the Gårding property for the CSGF of $F_7$.
\end{proof}
\printbibliography

\end{document}